\documentclass[twoside,11pt]{article}

\usepackage{bbm}
\usepackage{enumerate}
\usepackage{mathrsfs}
\usepackage{mathtools}
\usepackage{microtype}
\usepackage{tikz}
\usepackage{times}
\usepackage[Symbolsmallscale]{upgreek}
\usepackage{thm-restate} 
\usepackage{mdframed} 
\usepackage{amsfonts}
\usepackage{amssymb}
\usepackage{amsmath}
\usepackage{euscript}
\usepackage{textcomp}
\usepackage[export]{adjustbox}
\usepackage{graphicx}
\makeatletter
\def\set@curr@file#1{\def\@curr@file{#1}} 
\makeatother
\usepackage[load-configurations=version-1]{siunitx} 

\usepackage{comment}
\usepackage{xcolor}
\usepackage[normalem]{ulem}
\usepackage{listings}
\usepackage{float}

\usepackage{style}

\usepackage{jmlr2e}
\usepackage{caption}
\hypersetup{ hidelinks }

\newcommand{\eps}{\varepsilon}

\usepackage{lastpage}
\jmlrheading{24}{2023}{1-\pageref{LastPage}}{11/21; Revised 11/22}{1/23}{21-1363}{George Stepaniants}

\ShortHeadings{Learning PDEs in Reproducing Kernel Hilbert Spaces}{Stepaniants}
\firstpageno{1}

\begin{document}

\title{Learning Partial Differential Equations in\\
Reproducing Kernel Hilbert Spaces}

\author{\name George Stepaniants \email gstepan@mit.edu\\
        \addr Department of Mathematics\\
        Massachusetts Institute of Technology\\
        77 Massachusetts Ave, Cambridge, MA 02139}

\editor{Jean-Philippe Vert}

\maketitle

\begin{abstract}
We propose a new data-driven approach for learning the fundamental solutions (Green's functions) of various linear partial differential equations (PDEs) given sample pairs of input-output functions. Building off the theory of functional linear regression (FLR), we estimate the best-fit Green's function and bias term of the fundamental solution in a reproducing kernel Hilbert space (RKHS) which allows us to regularize their smoothness and impose various structural constraints. We derive a general representer theorem for operator RKHSs to approximate the original infinite-dimensional regression problem by a finite-dimensional one, reducing the search space to a parametric class of Green's functions. In order to study the prediction error of our Green's function estimator, we extend prior results on FLR with scalar outputs to the case with functional outputs. Finally, we demonstrate our method on several linear PDEs including the Poisson, Helmholtz, Schr\"{o}dinger, Fokker--Planck, and heat equation. We highlight its robustness to noise as well as its ability to generalize to new data with varying degrees of smoothness and mesh discretization without any additional training.
\end{abstract}

\begin{keywords}
    Green's functions, partial differential equations, reproducing kernel Hilbert spaces, functional linear regression, simultaneous diagonalization
\end{keywords}

\section{Introduction}

The rapid development of data-driven scientific discovery holds the promise of new and faster methods to analyze, understand, and predict various complex phenomena whose physical laws are still beyond our grasp. Of central interest to this development is the ability to solve efficiently a broad range of differential equations, and more precisely partial differential equations (PDEs) which still largely require advanced numerical techniques tailored for specific problems. 

In this paper, we study what is certainly one of the most inspiring outcomes of this program: solving PDEs from input-output data. Let $u(x,t)$ denote the state at time $t$ and location $x$ of a system evolving according to a PDE such as the dynamics of a swarm of particles or the propagation of a wave through a complex medium. The goal of this paper is to predict $u(x,t)$ under initial conditions $u(x,0)$, under boundary conditions $u(x, t) = b(x, t)$ for $x$ on the boundary of a domain, or under an external forcing $f(x,t)$ that represents the ambient conditions of an evolving system. In other words, we propose to learn an operator that maps these inputs $\{u(x,0), b(x, t), f(x,t)\}$ to output solutions $u(x,t)$ from input-output data. While such operators can be, and often are, nonlinear, we focus on linear operators in this paper. Not only is this a natural first step for this program but linear approximations to nonlinear phenomena are often useful and are, in general, more robust to model misspecification.

The study of learning input-output maps where the inputs or the outputs (or both) are functions traditionally falls under the umbrella of \textit{functional data analysis} introduced in the seminal paper of~\citet{ramsay1991} with much of the theory and practical applications to physical, biological, and economic data reviewed in the monographs by~\citet{ramsay2004functional,ramsay2007applied}.

As a concrete driving example, consider a system whose state $u$ is the solution of a PDE on a compact domain $D \subset \mathbb{R}^d$,
\begin{equation}\label{eq:pde_ex}
\begin{split}
    \cP u &= f \ \text{on} \ D\\
    \cB u &= 0 \ \ \text{on} \ \partial D\,.
\end{split}
\end{equation}
Here $\cP$ is a differential operator and $\cB$ encodes the boundary conditions. Given input $f$, and  knowledge of the operators $\cP, \cB$, solving the PDE~\eqref{eq:pde_ex} requires sophisticated numerical methods such as finite differences, spectral decompositions, or finite elements~\citep{zienkiewicz1977finite, trefethen2000spectral, leveque2007finite}. These methods can, and will, be used to create input-output pairs $\{(f_i, u_i), i=1, \ldots, n\}$ on the domain $D$ usually at some fixed level of discretization.

Building on this observation, we propose a natural goal: to learn a \textit{surrogate model} $\cT: f \mapsto u$ which takes input $f \in L^2(D_\cX)$ to output $u \in L^2(D_\cY)$. Solving this supervised learning problem is of paramount importance to unlock the potential of data-driven methods and understand real-world systems under new and unseen conditions. Surrogate modeling combines elements of numerical analysis, statistics, and machine learning to efficiently learn solution maps $\cT$ that are both physically relevant and fast to evaluate.

Two approaches have been predominantly used in the study of surrogate models. The first approach discretizes the functions $f \in \mathbb{R}^{m_y}, u \in \mathbb{R}^{m_x}$ and regresses a map (i.e. neural network) $F: \mathbb{R}^{m_y} \to \mathbb{R}^{m_x}$ on the data. This methodology has been successfully applied to surrogate modeling of flow fields, computed tomography, and porous media~\citep{guo2016convolutional, adler2017solving, zhu2018bayesian, bhatnagar2019prediction}. However, this approach is not robust to mesh-refinement which is a serious issue as it becomes data-hungry with increasing mesh sizes $m_x, m_y$~\citep{bhattacharya2020model}. Furthermore, evaluating such a model on more finely sampled input-output data requires an entire retraining of the architecture. The second predominant approach to surrogate modeling attempts to learn the solution $u$ of the PDE by parameterizing $u$ itself as a map $F_\theta: D_\cY \to \mathbb{R}$ where $\theta$ is a set of model parameters. For example, \citet{chen2021solving} optimize the solution map $F_\theta: D_\cY \to \mathbb{R}$ in a reproducing kernel Hilbert space (RKHS) which can be viewed as the maximum a posteriori estimator of a Gaussian process. Using RKHSs to estimate solutions of PDEs is a large area of research stemming from the foundational work of Fasshauer on mesh-free approximation methods~\citep[Chapter~38]{fasshauer2007meshfree} with many recent extensions to fractional PDEs and integro-differential equations~\citep{arqub2018numerical, arqub2019numerical, al2019computational, arqub2020adaptive}. A review of kernel-based numerical methods for PDEs can be found in~\citet{fornberg2015solving} and the classical text of~\citet[Chapters 5 \& 6]{saitoh2016theory} contains numerous examples of kernel methods specialized for solving forward and inverse problems for ODEs and PDEs. As an alternative to kernel methods, recent approaches have proposed to learn the solutions map $F_\theta: D_\cY \to \mathbb{R}$ of a PDE as a neural network where $\theta$ are the network weights. This idea has found applications in many applied problems such as the study of electrical impedance tomography, reaction-diffusion systems, and wave propagation~\citep{weinan2018deep, raissi2019physics, bar2019unsupervised} along with grid-independent generative modeling of images~\citep{dupont2021generative}. In general, this second approach of learning PDE solution maps $F_\theta: D_\cY \to \mathbb{R}$ is indeed independent of mesh discretization. However, its dependence on the initial conditions, boundary conditions, and forcings of the PDE are all fixed thus requiring complete retraining for a new set of parameters. Furthermore, this approach requires knowledge of the underlying PDE which is not always available.

The first results to propose surrogate models between function spaces which are independent of mesh discretization and do not rely on stringent modeling assumptions have appeared in works on neural operators~\citep{lu2019deeponet, bhattacharya2020model, li2020neural, li2020fourier} and operator-valued kernels~\citep{nelsen2020random, bao2022operator} for estimating PDE solution maps. By making little to no assumptions on the domain geometry and mesh discretization of the data, these works produced efficient surrogate models which could be applied to general nonlinear PDEs. Below we follow the same guiding principle to learn surrogate maps for a large class of linear PDEs by means of learning their Green's function. Restricting ourselves to linear systems enables us to prove rates on the prediction error of our model as in~\citet{de2021convergence, boulle2021learning} and places us in a setting where the learned surrogate models become interpretable.

A broad class of linear PDEs including the Poisson equation, wave equation, and heat equation are solved by an integral operator or \textit{fundamental solution} of the form
\begin{equation}\label{eq:fundsol}
u(y) = \cT(f)(y) = \beta(y) + \int_{D_\cX} G(x, y)f(x)\ud x
\end{equation}
where $G \in L^2(D_\cX \times D_\cY)$ is called the \textit{Green's function} and $\beta \in L^2(D_\cY)$ is a bias term that satisfies the boundary conditions of the PDE and solves the homogeneous equation $\cP u = 0$. In general, the domain $D_\cX$ of the input function $f$ can be different from the domain $D_\cY$ of the solution $u$, if for example $f$ is an initial or boundary condition of a PDE.

In recent years, a large body of work has developed fast and accurate approximations to Green's functions of linear PDEs. Simulating the solution of a linear PDE with a new input function is equivalent to integrating this input against the Green's function of the PDE as in~\eqref{eq:fundsol}. If the Green's function can be efficiently constructed from input-output pairs $(f_i, u_i)$ where $u_i \approx Gf_i$, then it can be used to solve the PDE under new forcings, initial, and boundary conditions. \textit{Matrix probing} algorithms discretize the Green's function at a set of collocation points or in a basis and reconstruct the matrix $G$ from a small number of matrix-vector products $Gf_i$. By assuming that $G$ lies in the span of prespecified basis matrices $\{B_1, \hdots, B_m\}$,~\citet{chiu2012matrix} solve a least squares problem to learn a Green's function independent of the mesh discretization. This methodology is applied to solve the Helmholtz equation with absorbing boundary conditions~\citep{belanger2015compressed} as well as linearized seismic inversion problems~\citep{demanet2012matrix}. For a large class of elliptic PDEs, the Green's function under mild regularity assumptions exhibits hierarchical low-rank structure~\citep{bebendorf2003existence}. Several matrix probing algorithms~\citep{lin2011fast, boulle2021learning, boulle2022learning} leverage this structure by evaluating the PDE with a few random forcings and use randomized SVD to efficiently learn the low-rank sub-blocks of $G$. For elliptic PDEs with symmetric Green's functions, sparse Cholesky factorization and operator-adapted wavelets~\citep{schafer2017compression, owhadi2019statistical, schafer2021sparse} can be applied to compress $G$ and $G^{-1}$ in order to expedite simulations of boundary layer problems and sparse ill-conditioned PDEs arising from computer graphics~\citep{chen2021multiscale}. We refer the readers to \citet{owhadi2015bayesian, owhadi2017multigrid, owhadi2019operator} for a comprehensive review of Bayesian numerical homogenization, operator-adapted wavelets and their application to fast multigrid and multiresolution methods for linear PDEs. Finally, recent approaches have modeled Green's functions of PDEs using autoencoders~\citep{gin2021deepgreen} and rational neural networks~\citep{boulle2022data}.

Similar to the approaches above, we are interested in learning the Green's function $G$ of a linear PDE in order to learn a surrogate model $\cT: f \mapsto u$ from the input of a PDE (e.g. forcing, initial condition, boundary condition) to its solution. We restrict ourselves to only observing input-output samples $(f_i, u_i)$ of the PDE and develop a Green's function estimator that is robust to high levels of noise in the data, a property that is crucial for learning on real data sets but is less addressed in prior work on Green's function estimation. Compared to the learning methods outlined above, we do not make assumptions on the form of the underlying PDE or its Green's function (e.g. elliptic, hyperbolic, hierarchical low-rank).

For different physical problems, the Green's function of a PDE satisfies certain sparsity, continuity, smoothness conditions, or combinations thereof. This motivates us to search for linear forward maps $\cT$ of the form given in~\eqref{eq:fundsol} where the Green's function $G$ and bias $\beta$ belong to a pair of reproducing kernel Hilbert spaces (RKHSs) $\cG, \cB$ respectively. This setting belongs to a subclass of operator RKHSs studied in~\citet{kadri2016operator} and restricting ourselves to the space of integral operators leads to more flexibility and insight about the choice of the RKHSs $\cG, \cB$ for many physical problems. 

Optimizing $\cT$ over the space of integral operators with the Green's function $G$ and bias $\beta$ in an RKHS offers four concrete advantages:
\begin{enumerate}
    \item The estimators $\widehat{G}, \widehat{\beta}$ of the Green's function and bias term are robust to significant levels of noise in the input-output samples $\{(f_i, u_i)\}_{i=1}^n$ due to the penalization of their RKHS norm.
    \item It gives an explicit closed-form (representer theorem) for the best-fit functions $\widehat{G}, \widehat{\beta}$ over data samples $\{(f_i, u_i)\}_{i=1}^n$ which is independent of the mesh discretization of the samples and, crucially, can extrapolate to new meshes.
    \item It allows us to interpret our learned model by inspecting the estimated Green's function $\widehat{G}$ which specifies the impulse response of the system.
    \item The RKHSs can be designed to enforce specific structure and symmetries in $\widehat{G}, \widehat{\beta}$ based on prior knowledge about the system.
\end{enumerate}
The rest of this paper is organized as follows. We discuss our data generating model in Section~\ref{sec:data} along with the full representer theorem for the best-fit Green's function and bias term. The implementation of our Green's function and bias term RKHS estimators are detailed in Section~\ref{sec:implementation}. In Section~\ref{sec:rates}, we extend the analysis of~\citet{yuan2010reproducing} for real outputs to the case of functional outputs and derive the corresponding error bounds for the RKHS Green's function estimator. A concrete list of examples to estimate Green's functions and bias terms of linear PDEs in different RKHSs are given in Section~\ref{sec:examples} with all proofs deferred to the appendix.

\subsection{Motivating Example}
All approaches for learning Green's function, including our proposed method, begin by taking a set of input-output functions $\{(f_i(x), u_i(y))\}_{i=1}^n$ discretized on a finite set of grid points $\{x_j\}_{j=1}^{m_x}$ and $\{y_k\}_{k=1}^{m_y}$ respectively. The goal is to learn a function $G(x, y)$ such that numerical integrations of $G$ against the discretized inputs $f_i$ are as close as possible to the discretized outputs $u_i$. This can be written as
\begin{equation}\label{eq:gf_regression}
    u(y_k) \approx \sum_{j=1}^{m_x}G(x_j, y_k)f_i(x_j)\Delta_j^x
\end{equation}
where $\Delta_j^x$ are numerical quadrature weights.

The classical grid-based methods outlined above for Green's function estimation including matrix probing and sparse factorization have focused almost exclusively on the noiseless setting when data collected from the underlying PDE can be perfectly measured. These methods aim to learn $G$ solely on the grid points $G(x_j, y_k)$ and do not enforce smoothness by constraining the values of $G$ at neighboring grid points to be close. They are then able to fully exploit the linearity of~\eqref{eq:gf_regression} to obtain remarkably efficient and accurate algorithms on noiseless data. In practice however, measured data from real-world systems are often corrupted with high levels of noise which make these prior approaches inapplicable. On a moderate number of input-output samples, these classical grid based approaches tend to overfit noisy data regardless of the number of grid points $\{x_j\}_{j=1}^{m_x}$, $\{y_k\}_{k=1}^{m_y}$ used to construct the estimator.

Here we show how estimating the Green's function of a PDE in an RKHS allows us to fit~\eqref{eq:gf_regression} while also penalizing the RKHS norm of our estimator. The additional RKHS norm leads to a convex objective for $G$ and naturally penalizes the smoothness of the learned Green's function allowing for significant robustness to noise (see Sections~\ref{sec:data}~\&~\ref{sec:implementation}).

In Figure~\ref{fig:poisson_interp}, we compare our approach to a classical grid-based method for learning the Green's function of the Poisson equation. Taking the Poisson equation
$\Delta u = f$ on $[0, 1]$ with zero Dirichlet boundary conditions, we estimate it's Green's function from 500 functional samples $(f_i, u_i)$ discretized on a 100 point uniform grid $\{x_j = y_j = \frac{j-1}{99}\}_{j=1}^{100}$. The input forcings $f_i$ are simulated using a Karhunen--Loeve expansion (KLE) with a squared exponential kernel of lengthscale $\ell = 0.01$ and the solutions $u_i$ are generated with a standard finite difference solver and corrupted with 10\% Gaussian noise (see Appendix~\ref{app:dataexp} for details). The true analytic Green's function of the Poisson equation is given by
\begin{equation}\label{eq:poisson_greens_function}
    G_\text{Poisson}(x, y) = \frac{1}{2}(x + y - |x - y|) - xy
\end{equation}
depicted in the rightmost plot. With noise corrupted data, naively learning the Green's function $G$ as a discretized matrix $G_{jk} = G(x_j, x_k)$ by solving the least squares problem $\sum_{i=1}^{500} \|\mathbf{G}^Tf_i - u_i\|_2^2$ leads to a nonsmooth estimator that is corrupted by the noise in the train samples (left plot). Instead, by learning $G(x, y)$ as a \textit{function} in a squared exponential RKHS (e.g. sum of 2D Gaussian kernels of width $\sigma = 5 \times 10^{-2}$) we can penalize the smoothness of $G$ to learn a much more faithful estimate of the true Green's function (center plot). Our learned estimator is smooth and has an analytic closed form which can be reevaluated on finer mesh sizes. As opposed to the matrix estimator (left plot), the RKHS estimator (center plot) is much more interpretable as it allows us to conclude that perturbations $f$ concentrated around a point $x_0 \in [0, 1]$ produce a smoothed response in $u$ around that same point and that such perturbations get weaker as $x_0$ approaches the boundary of the domain. This example demonstrates the importance of learning Green's functions of PDEs in function spaces that enforce structure such as continuity and smoothness.

\begin{figure}[h]
\centering
\includegraphics[width=\textwidth]{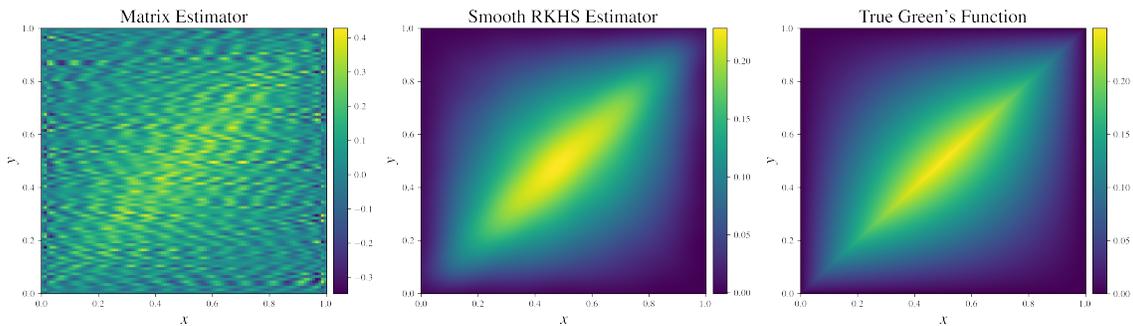}
\caption{Estimating Green's function of the Poisson equation with zero Dirichlet boundary conditions on 500 noisy functional samples $(f_i, u_i)$ discretized on a 100 point uniform grid. Learning a Green's function as a $100 \times 100$ matrix without enforcing smoothness leads to a noise corrupted estimator that overfits the training samples. Alternatively, learning the Green's function as a sum of 2D Gaussian kernels (e.g. in a squared exponential RKHS) and penalizing its smoothness with regularization $\lambda = 10^{-4}$ faithfully estimates the true Green's function. This estimator has a known functional form which allows us to resample it on a finer mesh size of $500 \times 500$ (center plot). The true Green's function of the Poisson equation is shown in the rightmost plot.}\label{fig:poisson_interp}
\end{figure}

\subsection{Definitions and Notation}
Before we describe our data model and estimator, we introduce mathematical definitions and notation that will be used throughout the paper. 

\subsubsection{Function Spaces}
\begin{definition}[Real Hilbert Space]
    A real Hilbert space $\cH$ is a real inner product space with inner product $\inprod{x, y}_\cH$ for all $x, y \in \cH$. The inner product induces a norm given by $\|x\|_\cH = \inprod{x, x}_\cH$ such that the Hilbert space $\cH$ is a complete metric space with respect to the metric $d(x, y) = \|x - y\|_\cH$.
\end{definition}

The space $L^2(D)$ denotes the Hilbert space of square-integrable functions on the domain $D$ with the standard inner product $\inprod{f, g}_{L^2(D)} = \int_D f(x)g(x) \ud x$ and norm $\|f\|_{L^2(D)}^2 = \int_D f(x)^2 \ud x$ for all $f, g \in L^2(D)$. Likewise, the Euclidean space $\mathbb{R}^d$ equipped with the inner product $\inprod{u, v}_2 = \sum_{i=1}^d u_iv_i$ is a simple example of a Hilbert space.

\begin{definition}[Reproducing Kernel Hilbert Space]
    A reproducing kernel Hilbert space (RKHS) $\cH$ is a Hilbert space of functions on a domain $D$ with an inner product $\inprod{f, g}_\cH$ for all $f, g \in \cH$. For each $x \in D$, the Hilbert space $\cH$ has a unique element $K_x \in \cH$ such that
    \begin{equation}
        \inprod{f, K_x}_\cH = f(x) \quad \forall f \in \cH
    \end{equation}
    which is called the \underline{reproducing property}. The function $K: D \times D \to \mathbb{R}$ defined as $K(x, y) = \inprod{K_x, K_y}$ for all $x, y \in D$ is called the \underline{reproducing kernel}. Often we will denote the Hilbert space norm by $\|\cdot\|_\cH = \|\cdot\|_K$ and the inner product by $\inprod{\cdot, \cdot}_\cH = \inprod{\cdot, \cdot}_K$.
\end{definition}

\subsubsection{Kernel Functions}
Given a set or domain $D$, a real-valued kernel is a function $K: D \times D \to \mathbb{R}$ that is symmetric which means $K(x, y) = K(y, x)$ for all $x, y \in D$. Furthermore, a kernel is called \textit{positive semidefinite} if
\begin{equation}\label{eq:psd_kernel}
    \sum_{i=1}^n\sum_{j=1}^n c_ic_jK(x_i, x_j) \geq 0
\end{equation}
for all $x_1, \hdots, x_n \in D$ for any integer $n \geq 1$ and $c_1, \hdots, c_n \in \mathbb{R}$. A kernel is called $\textit{positive definite}$ if the inequality in~\eqref{eq:psd_kernel} is strict.

For a domain $D \subseteq \mathbb{R}^d$, we say that $K$ is a \textit{Mercer kernel} if it is a continuous function on $D \times D$ that is symmetric and positive semidefinite.

\subsubsection{Probability and Expectation}
\begin{itemize}
\item For a distribution $\mathbb{P}$ over functions in $L^2(D)$, the notation $F \sim \mathbb{P}$ denotes that $F$ is a functional sample from this distribution.

\item The expectation with respect to a distribution $\mathbb{P}$ is denoted by $\mathbb{E}_\mathbb{P}[\cdot]$ and the subscript is dropped for convenience when it is clear which distribution is being used.
\end{itemize}

\subsubsection{Functions and Operators}
\begin{itemize}
\item Given two functions $f(x)$ and $g(y)$ we define their tensor product as $(f \otimes g)(x, y) = f(x)g(y)$.

\item Given two domains $D_\cX$ and $D_\cY$ we can define a function $M \in L^2(D_\cX \times D_\cY)$. For any $f \in L^2(D_\cY)$ we use the shorthand $M(f) = \int_{D_\cY} M(x, y)f(y) \ud y$ and similarly for any $f \in L^2(D_\cX)$ we write $M^T(f) = \int_{D_\cX} M(x, y)f(x) \ud x$.

\item An operator $\Gamma: L^2(D) \to L^2(D)$ for some domain $D$ is called positive semidefinite if $\inprod{\Gamma(f), f}_{L^2(D))} \geq 0$ for all $f \in L^2(D)$. If this inequality is always strictly larger than zero, then it is called positive definite. For two operators $\Sigma: L^2(D) \to L^2(D)$ and $\Gamma: L^2(D) \to L^2(D)$, we use the notation $\Sigma \preceq \Gamma$ when the difference $\Gamma - \Sigma$ is a positive semidefinite operator and the notation $\Sigma \prec \Gamma$ when their difference is strictly positive definite.
\end{itemize}

\subsubsection{Sets and Sequences}
\begin{itemize}
    \item For a positive integer $m$, we use the notation $[m]$ to denote the set of integers from 1 to $m$.

    \item We use $\mathbb{R}$ to denote the set of real numbers and $\mathbb{R}_+$ to denote the set of nonnegative real numbers.

    \item For two positive real sequences $a_n, b_n \in \mathbb{R}_+$ for $n \geq 1$ we use the symbol $a_n \asymp b_n$ to denote that the ratio $a_n/b_n$ is bounded away from zero and infinity as $n \to \infty$.
    
    \item We write $a_n \lesssim b_n$ to signify that $a_n \leq Cb_n$ for all $n \geq 1$ for some real constant $C > 0$.
\end{itemize}

\section{Model and Estimator}\label{sec:data}
Given compact domains $D_\cX \subset \mathbb{R}^{d_\cX}$ and $D_\cY \subset \mathbb{R}^{d_\cY}$, our goal is to learn a map from input functions $f: D_\cX \to \mathbb{R}$ to output solutions $u: D_\cY \to \mathbb{R}$. We note that $D_\cX$ and $D_\cY$ can be different domains such as when $f$ is the boundary condition of a PDE and $u$ is the solution on the interior of the domain.

We consider the model for a random input-output pair $(F, U) \in L^2(D_\cX) \times L^2(D_\cY)$:
\begin{equation}\label{eq:model}
    U = \cT^*(F) + \varepsilon
\end{equation}
where  $\cT^*: L^2(D_\cX) \to L^2(D_\cY)$ is a possibly nonlinear operator (parameter of interest) and $\eps$ is a centered random variable in $L^2(D_\cY)$. We assume further that $F\sim \subG(\Gamma_F)$ and $\eps\sim \subG(\Gamma_\eps)$ are subgaussian where $\Gamma_F: L^2(D_\cX) \to L^2(D_\cX)$ and $\Gamma_\eps: L^2(D_\cY) \to L^2(D_\cY)$ are positive semidefinite trace-class linear operators known as \textit{covariance proxies}; see Appendix~\ref{app:subG}. The \textit{covariance operator} of $F$ is the function $\Sigma_F \in L^2(D_\cX \times D_\cX)$ given by
\begin{equation}
    \Sigma_F(x, x') = \mathbb{E}[(F(x) - \mathbb{E}[F(x)]) \cdot (F(x') - \mathbb{E}[F(x')])]
\end{equation}
where the expectation is taken over the randomness of the process $F$. The covariance can also be interpreted as a linear operator $\Sigma_F: L^2(D_\cX) \to L^2(D_\cX)$ and as shorthand we will write $\Sigma_F = \mathbb{E}[(F - \mathbb{E}[F]) \otimes (F - \mathbb{E}[F])]$ where $\otimes$ is the tensor product. For our theoretical analysis, we make the assumption that $F$ is strictly subgaussian as defined in Appendix~\ref{app:subG}.
\begin{assumption}\label{assump:strictsubg}
    The subgaussian process $F \sim \subG(\Gamma_F)$ is strictly subgaussian meaning that $\Gamma_F \preceq C\Sigma_F$ for some constant $C > 0$. Additionally, we assume that $\Gamma_F$ is strictly positive definite. In other words, the covariance function and covariance proxy of $F$ are both strictly positive definite and on the same order
    \begin{equation}
        0 \prec c\Sigma_F \preceq \Gamma_F \preceq C\Sigma_F
    \end{equation}
    for some $0 < c < C$.
\end{assumption}
Modeling the input functions $F$ by a subgaussian distribution includes as a subset all inputs which can be constructed from random subgaussian-weighted combinations of basis functions with sufficient decay in their weights (see Karhunen--Loeve expansion in Appendix~\ref{app:data_gen}). Such random functions are used extensively as initial conditions, boundary conditions, and forcing functions for learning PDEs~\citep{bhattacharya2020model, boulle2021learning}.

To learn an operator for the solution map of a linear PDE, the affine representation~\eqref{eq:fundsol} suggests to consider operators  of the form
\begin{equation}\label{eq:linop}
    \cT_{\beta, G}(f) = \beta + \int_{D_{\cX}} G(x,\cdot )f(x) \ud x\,, \qquad  \beta \in L^2(D_\cY),\ G \in L^2(D_\cX \times D_\cY)
\end{equation}
This simple representation is still too flexible to be learned from a finite amount of data. To overcome this limitation, we impose additional regularity on $G$ and $\beta$, namely that $G \in \cG$ and $\beta \in \cB$, where $\cG$ and $\cB$ are two RKHSs with continuous, square integrable, and strictly positive definite reproducing kernels
\begin{equation}
K: (D_\cX \times D_\cY) \times (D_\cX \times D_\cY) \to \mathbb{R} \ \text{and} \ Q: D_\cY \times D_\cY \to \mathbb{R}.
\end{equation}

In fact we establish below an oracle inequality that holds for a much more general class of estimators for $\cT^*$. In particular, it shows that our estimator performs well not only when $\cT^*$ is affine as in~\eqref{eq:linop} (well-specified model) but also if it is well approximated by such estimators (mis-specified case); see Theorem~\ref{thm:oracle} below. In general, our results hold under the following sublinear growth condition.

\begin{assumption}\label{assump:linear_growth}
    The true solution map $\cT^*: L^2(D_\cX) \to L^2(D_\cY)$ of the PDE has at most linear growth, that is, for all $f \in L^2(D_\cX)$,
    \begin{equation}
        \|\cT^*(f)\|_{L^2(D_\cY)} \leq c + M\|f\|_{L^2(D_\cX)}
    \end{equation}
    where $c, M \geq 0$ are constants.
\end{assumption}
If we set $c = 0$ then Assumption~\ref{assump:linear_growth} is equivalent to requiring that $\cT^*$ be a bounded operator. As an example, a general class of elliptic PDEs on compact domains have bounded solution maps $\cT^*$ by the bounded inverse theorem~\citep[Section~6.2, Theorem~6]{evans1998partial}. The condition above is also clearly satisfied by all linear PDEs that have a square integrable Green's function.

\medskip

We are now in a position to describe our estimator. Assume that we observe independent samples $(F_1, U_1), \ldots, (F_n, U_n)$ of $(F,U)$ from~\eqref{eq:model} and define the empirical risk
\begin{equation}\label{eq:emprisk}
\widehat{R}(\beta, G) := \frac{1}{n}\sum_{i=1}^n\Big\|U_i - \beta - \int_{D_\cX} G(x, \cdot)F_i(x)\ud x\Big\|_{L^2(D_\cY)}^2.
\end{equation}
Likewise, define the penalized empirical risk
\begin{equation}\label{eq:penregempriskbias}
    \widehat{R}_{\rho, \lambda}(\beta, G) := \widehat{R}(\beta, G) + \rho P(\beta) + \lambda J(G)
\end{equation}
where $P$ and $J$ are the penalty functionals for the Green's function and bias term respectively. Then our estimators are defined as the RKHS minimizers
\begin{equation}
    \widehat{\beta}_{n, \rho, \lambda}, \widehat{G}_{n, \rho, \lambda} := \argmin_{\beta \in \cB, G \in \cG}\widehat{R}_{\rho, \lambda}(\beta, G)
\end{equation}
Here, the subscripts $n, \rho, \lambda$ on our estimators indicate the number of samples $n$ and the regularization values $\rho, \lambda$ for which our estimators were optimally chosen. We remind the reader that $\cG$ is the RKHS of Green's functions with reproducing kernel $K: (D_\cX \times D_\cY)^2 \to \mathbb{R}$ and $\cB$ is the RKHS of bias terms with reproducing kernel $Q: D_\cY \times D_\cY \to \mathbb{R}$ where both reproducing kernels are continuous, symmetric, and strictly positive definite. The penalty functionals with which we will regularize the Green's function and bias term are the respective RKHS norms
\begin{equation}
    J(G) = \|G\|_\cG^2 = \|G\|_K^2, \quad P(\beta) = \|\beta\|_\cB^2 = \|\beta\|_Q^2
\end{equation}
which make the penalized empirical risk $\widehat{R}_{\rho, \lambda}(\beta, G)$ a strictly convex objective such that our estimators $\widehat{\beta}_{n, \rho, \lambda}, \widehat{G}_{n, \rho, \lambda}$ are unique.

\subsection{Common Examples of RKHS Kernels}\label{subsec:kernel_examples}
By the classical Moore-Aronzajn theorem~\citep{aronszajn1950theory}, any positive-semidefinite kernel defines a unique RKHS. Hence, for nonzero measure sets $D_\cX \subset \mathbb{R}^{d_\cX}, D_\cY \subset \mathbb{R}^{d_\cY}$ the \textit{squared-exponential (SE) kernels}
\begin{equation}
    K(x, y, \xi, \eta) = \exp\Big(-\frac{\|x-\xi\|^2}{2\sigma_x^2}\Big)\exp\Big(-\frac{\|y-\eta\|^2}{2\sigma_y^2}\Big), \quad Q(y, \eta) = \exp\Big(-\frac{\|y - \eta\|^2}{2\sigma_y^2}\Big)
\end{equation}
generate unique RKHSs $\cG$ and $\cB$ for the Green's function and bias term respectively. The RKHS norms defined by such SE kernels strongly penalize derivatives of a function and hence bias the choice of $G$ and $\beta$ towards very smooth functions. In the example above, the kernel $K(x, y, \xi, \eta)$ is constructed as a products of simpler SE kernels in $(x, \xi)$ and $(y, \eta)$. Constructing kernels through tensor products is a standard procedure outlined in~\citet[Theorem~2.20]{saitoh2016theory} and is an important tool for building RKHSs of higher-dimensional functions.

Another popular example are RKHSs generated by \textit{exponential kernels}
\begin{equation}
    K(x, y, \xi, \eta) = \exp\Big(-\sqrt{\frac{\|x-\xi\|^2}{\sigma_x^2} + \frac{\|y-\eta\|^2}{\sigma_y^2}}\Big), \quad Q(y, \eta) = \exp\Big(-\frac{\|y - \eta\|}{\sigma_y}\Big)
\end{equation}
whose RKHS norms do not penalize any derivatives and allow us to represent functions $G$ and $\beta$ which are nondifferentiable.

An important family of kernel functions known as \textit{Mat\'{e}rn kernels}~\citep{genton2001classes} are given by
\begin{equation}\label{eq:matern}
\begin{gathered}
    C_\nu(d) = \frac{1}{\Gamma(\nu)2^{\nu-1}}\Big(\frac{\sqrt{2\nu}}{l}d\Big)^\nu K_\nu\Big(\frac{\sqrt{2\nu}}{l}d\Big)\\
    K(x, y, \xi, \eta) = C_\nu\Big(\sqrt{\frac{\|x-\xi\|^2}{\sigma_x^2} + \frac{\|y-\eta\|^2}{\sigma_y^2}}\Big), \quad Q(y, \eta) = C_\nu\Big(\frac{\|y - \eta\|}{\sigma_y}\Big)
\end{gathered}
\end{equation}
and interpolate between the exponential kernel at $\nu = 0$ and the Gaussian kernel as $\nu \to \infty$. The parameter $\nu$ controls how strongly the magnitudes of higher-order derivatives of $G$ and $\beta$ are penalized.

 
The positive definite kernel functions described above are all examples of \textit{radial basis functions} (RBFs) or anisotropic variants of RBFs; functions that only depend on the distances $\|x - \xi\|$ and $\|y - \eta\|$. In general, reproducing kernels are not restricted to be of this form. For example, given any finite or infinite set of orthonormal functions $\{\psi_k\}_{k=1}^m$ on $L^2(D)$ we have that
\begin{equation}
    K(x, y) = \sum_{k=1}^m \lambda_k\psi_k(x)\psi_k(y)
\end{equation}
with $\lambda_k > 0$ and $\sum_{k=1}^m \lambda_k^2 < \infty$ defines an RKHS of functions on the domain $D$ with inner product
\begin{equation}
    \inprod{f, g} = \sum_{k=1}^m \frac{\inprod{f, \psi_k}_{L^2(D)}\inprod{g, \psi_k}_{L^2(D)}}{\lambda_k}.
\end{equation}
for all $f, g$ in this RKHS. In particular, a basis of $L^2(D)$ such as a Fourier or polynomial basis truncated to a finite number of terms is an example of an RKHS. As a concrete example, on the box domain $D = [0, 1]^d$ the kernel
\begin{equation}
    K(x_1, \hdots, x_d, \xi_1, \hdots, \xi_d) = 2^d\sum_{k_1, \hdots, k_d=1}^\infty \frac{\sin(\pi k_1x_1) \hdots \sin(\pi k_dx_d)\sin(\pi k_1\xi_1) \hdots \sin(\pi k_d\xi_d)}{\pi^2(k_1^2 + \hdots + k_d^2)}.
\end{equation}
is a reproducing kernel for the Sobolev-Hilbert space
\begin{equation}
\begin{split}
    W_1^2(D) = \Big\{f:f \ \text{absolutely continuous}, \ f \equiv 0 \ \text{on} \ \partial D, \ \frac{\partial f}{\partial x_i} \in L^2(D), \ \forall 1 \leq i \leq d\Big\}
\end{split}
\end{equation}
with inner product
\begin{equation}
    \inprod{f, g}_{W_1^2(D)} = \int_D \nabla f(x) \cdot \nabla g(x)\ud x.
\end{equation}
This can be seen by noting that $K$ is the Green's function of the Poisson equation
\begin{equation}\label{eq:poisson_highdim}
    -\Delta u(x) = s(x) \quad \forall x \in D, \qquad u(x) = 0 \quad \forall x \in \partial D
\end{equation}
with homogeneous Dirichlet boundary conditions~\cite[Section~8.2.2-16]{polyanin2015handbook}.

In this paper, all experiments described in Section~\ref{sec:examples} only use the exponential, squared exponential, and Mat\'{e}rn kernels as they are used ubiquitously in the kernel methods literature, are simple to implement, and can be computed efficiently through fast kernel matrix-vector products (see Section~\ref{sec:implementation}). These kernels are strictly positive definite although our analysis can also be extended to degenerate kernels which are positive semidefinite.

\begin{remark}
    We refer the reader to the classical texts of~\cite[Chapters~1, 2, 10]{wahba1990spline},~\citet[Chapters~1, 7]{berlinet2011reproducing} and~\citet[Chapter~1]{saitoh2016theory} for detailed examples of reproducing kernels and their associated Hilbert spaces. In particular, these texts outline the deep connection between Green's functions of differential equations and reproducing kernels. As shown on the example of the Poisson equation in~\eqref{eq:poisson_highdim}, Green's functions of classical PDEs can be seen as natural reproducing kernels over the space of their solutions. Similar reproducing kernels can be derived from Green's functions of the Helmholtz and heat equations~\citep[Sections~1.7.2-1.7.3]{saitoh2016theory}. In this paper, we take the opposite perspective and use a reproducing kernel to learn a Green's function of an unknown PDE from data. Our choice of reproducing kernel for the RKHS gives rise to a best-fit estimator for the Green's function of a PDE.
\end{remark}

\subsection{Representer Theorem}\label{sec:repthm}
Given random input-output function samples $\{(F_i, U_i)\}_{i=1}^n$ from~\eqref{eq:model}, we would like to minimize the regularized cost function $\widehat{R}_{\rho, \lambda}$ defined in~\eqref{eq:penregempriskbias}. Our cost function is composed of a convex mean-squared error $\widehat{R}(\beta, G)$ given in~\eqref{eq:emprisk} as well as two strictly convex RKHS regularizers $J(G) = \|G\|_\cG^2$ and $P(\beta) = \|\beta\|_\cB^2$. Hence, it is strictly convex implying that it has a unique minimizer
\begin{equation}\label{eq:estimators}
    \widehat{\beta}_{n, \rho, \lambda}, \widehat{G}_{n, \rho, \lambda} := \argmin_{\beta \in \cB, G \in \cG} \widehat{R}_{\rho, \lambda}(\beta, G).
\end{equation}
which are the estimators for the Green's function and bias term of our PDE.

In practice, the functional inputs $F_i(x)$ and ouputs $U_i(y)$ are given to us at discretized mesh points $\{x_j\}_{j=1}^{m_x}$ and $\{y_k\}_{k=1}^{m_y}$. Given discretized data, our original objective function
\begin{equation}\label{eq:full_continuous_cost}
    \widehat{R}_{\rho, \lambda}(\beta, G) = \frac{1}{n}\sum_{i=1}^n\Big\|U_i - \beta - \int_{D_\cX} G(x, \cdot)F_i(x)\ud x\Big\|_{L^2(D_\cY)}^2 + \rho\|\beta\|_Q^2 + \lambda\|G\|_K^2
\end{equation}
is numerically approximated by a Riemann sum
\begin{equation}\label{eq:riemannsum}
    \frac{1}{n}\sum_{i=1}^n\sum_{k=1}^{m_y}\Big(U_i(y_k) - \beta(y_k) - \sum_{j=1}^{m_x}G(x_j, y_k)F_i(x_j)\Delta_j^x\Big)^2\Delta_k^y + \rho\|\beta\|_Q^2 + \lambda\|G\|_K^2
\end{equation}
where $\Delta_j^x$ and $\Delta_k^y$ are the quadrature weights for the Riemann sums in $x$ and $y$ respectively. Importantly, here we discretize the square loss term using a Riemann sum but the Hilbert norms $\|\beta\|_Q^2, \|G\|_K^2$ are kept continuous. The semi-discrete objective function above is now in the right form for us to apply the traditional representer theorem to $\beta$ and $G$. First we see that any minimizer $\beta$ must have the form
\begin{equation}\label{eq:biasrepresenter}
    \widehat{\beta}_{n, \rho, \lambda}(y) = \sum_{k=1}^{m_y}Q(y, y_k)w_k \Delta_k^y
\end{equation}
where $\mathbf{w} = (w_1, \hdots, w_{m_y})^T \in \mathbb{R}^{m_y}$ is any weight vector and $Q$ is once again the reproducing kernel for the RKHS $\cB$ of the bias term. Now fixing $\beta$, our discretized loss function~\eqref{eq:riemannsum} in $G$ has the form
\begin{equation}\label{eq:restrictedloss}
    L\Big(\Big\{\Big\langle{\mathbf{G}, \mathbf{A}_{ik} \Big\rangle_2}\Big\}_{i \in [n], k \in [m_y]}\Big) + \lambda\|G\|_K^2, \qquad \mathbf{G} = \{G(x_j, y_k)\}, \ \ \mathbf{A}_{ik} = (\mathbf{F}_i \odot \boldsymbol{\Delta}^x)\mathbf{e}_k^T \in \mathbb{R}^{m_x \times m_y}
\end{equation}
for some loss function $L: \mathbb{R}^n \to \mathbb{R}$ where $\mathbf{F}_i = (F_i(x_1), \hdots, F_i(x_{m_x})^T \in \mathbb{R}^{m_x}$ and $\boldsymbol{\Delta}^x = (\Delta_1^x, \hdots, \Delta_{m_x}^x)^T \in \mathbb{R}^{m_x}$ and $\mathbf{e}_k \in \mathbb{R}^{m_y}$ denotes the unit vector with a one in the $k$th position. Here $\odot$ denotes the element-wise product and $\langle\mathbf{G}, \mathbf{A}_{ik} \rangle_2$ is the matrix trace inner product. The loss function as written above is a function of the $m_x \times m_y$ function evaluations $\{G(x_j, y_k)\}$ plus a regularization term so it is directly amenable to the classical representer theorem in $G$. Hence, any minimizer $G$ must have the form
\begin{equation}\label{eq:discreterepresenter_simple}
    \widehat{G}_{n, \rho, \lambda}(x, y) = \sum_{j=1}^{m_x}\sum_{k=1}^{m_y} K(x, y, x_j, y_k)W_{jk}
\end{equation}
where $\mathbf{W} = \{W_{jk}\} \in \mathbb{R}^{m_x \times m_y}$ is any weight matrix and $K$ is the reproducing kernel for the RKHS $\cG$ of the Green's function. Surprisingly, the particular form of our loss function $L(\{\langle{\mathbf{G}, \mathbf{A}_{ik} \rangle}\}_{i \in [n], k \in [m_y]})$ allows us to give a more constrained description of $G$. Since the function evaluations $\mathbf{G} = \{G(x_j, y_k)\}$ only enter our loss function through inner products with $\{\mathbf{A}_{ik}\}$, we can in fact show that
\begin{equation}\label{eq:discreterepresenter}
    \widehat{G}_{n, \rho, \lambda}(x, y) = \sum_{j=1}^{m_x}\sum_{k=1}^{m_y} K(x, y, x_j, y_k)W_{jk}, \quad \mathbf{W} \in \spn\{\mathbf{A}_{ik}: i \in [n], k \in [m_y]\}.
\end{equation}
A concise proof of this statement is detailed at the start of Appendix~\ref{app:repthm}. If we expand
\begin{equation}
    \mathbf{W} = \sum_{i=1}^n\sum_{k=1}^{m_y} \mathbf{A}_{ik}c_{ik}\Delta_k^y
\end{equation}
for any constants $c_{ik} \in \mathbb{R}$ then we can finally write
\begin{equation}
    \widehat{G}_{n, \rho, \lambda}(x, y) = \sum_{j=1}^{m_x}\sum_{k=1}^{m_y} K(x, y, x_j, y_k)F_i(x_j)c_{ik}\Delta_j^x\Delta_k^y.
\end{equation}
Hence we have derived a representer theorem for our Green's function $G$ that minimizes the discretized loss in~\eqref{eq:riemannsum}. Building on the derivations in~\citet[Section~1.3, Theorem~1.3.1]{wahba1990spline}, we also present a continuous version of this result when the mesh discretizations $\{x_j\}_{j=1}^{m_x}$ and $\{y_k\}_{k=1}^{m_y}$ are taken to the continuum limit, i.e. minimizing $\widehat{R}_{\rho, \lambda}(\beta, G)$ from~\eqref{eq:full_continuous_cost} directly without a Riemann sum approximation.

\begin{theorem}[Green's Function Representer Theorem]\label{thm:representer}
For any minimizer $\widehat{\beta}_{n, \rho, \lambda}, \widehat{G}_{n, \rho, \lambda}$ of the empirical risk $\widehat{R}_{\rho, \lambda}(\beta, G)$ from~\eqref{eq:full_continuous_cost} on functional data $\{(F_i, U_i)\}_{i=1}^n \subset L^2(D_\cX) \times L^2(D_\cY)$, the function $\widehat{G}_{n, \rho, \lambda}$ must have the form
\begin{equation}\label{eq:fullrepresenter}
    \widehat{G}_{n, \rho, \lambda}(x, y) = \sum_{i=1}^n \int_{D_\cX}\int_{D_\cY} K(x, y, \xi, \eta)F_i(\xi)c_i(\eta)\ud\xi \ud\eta
\end{equation}
where $c_i \in L^2(D_\cY)$ for $i \in [n]$ are coefficient functions which are free to be determined.
\end{theorem}
The theorem above can be seen as a consequence of the general representer theorem derived in the seminal work of~\citet[Section~4, Theorem~4.1]{micchelli2005learning} which laid the groundwork for vector-valued RKHSs, sometimes called operator RKHSs. We give an alternative proof of this result in Appendix~\ref{app:repthm} at the same level of generality for any loss function optimized over the space of an operator RKHS (i.e. not just integral operators). A similar representer theorem is also proven in~\citet[Appendix~B, Theorem~9]{kadri2016operator} albeit with restrictive assumptions that the loss function and regularization term are quadratic.

In the following sections, we describe the numerical implementation of our Green's function and bias term estimators and derive error bounds for approximating the Green's function $G$ of a PDE in an RKHS $\cG$ given functional data samples $(F_i, U_i)$.

\section{Implementation}\label{sec:implementation}
In practice, functional input-output data $\{(F_i, U_i)\}_{i=1}^n$ are almost always discretized on a set of mesh points $\{x_j\}_{j=1}^{m_x}, \{y_k\}_{k=1}^{m_y}$ so our implementation of the Green's function and bias term estimators are based on the discrete representer theorems from equations~\eqref{eq:biasrepresenter} and~\eqref{eq:discreterepresenter_simple} respectively
\begin{equation}
    \beta_\mathbf{w}(y) = \sum_{k=1}^{m_y}Q(y, y_k)w_k \Delta_k^y, \quad G_\mathbf{W}(x, y) = \sum_{j=1}^{m_x}\sum_{k=1}^{m_y} K(x, y, x_j, y_k)W_{jk}\Delta_j^x\Delta_k^y.
\end{equation}
Note that we do not use the more restricted form of the discrete representer theorem~\eqref{eq:discreterepresenter} for the Green's function $G$ as it is only efficient in the small data limit when $n \ll \min(m_x, m_y)$ but in practice we take $m_x, m_y$ on the order of $10^2$ and the number of training samples $n$ range from 100 to 500.

The full forward map $\cT: f \to u$ of the PDE from~\eqref{eq:linop} is estimated by
\begin{equation}\label{eq:discrete_model}
    [\cT_{\mathbf{w}, \mathbf{W}}(f)](y) = [\cT_{\beta_\mathbf{w}, G_\mathbf{W}}(f)](y) = \beta_\mathbf{w}(y) + \sum_{j=1}^{m_x} G_\mathbf{W}(x_j, y)f(x_j)\Delta_j^x.
\end{equation}
where $\mathbf{w}, \mathbf{W}$ are the weights of our estimator. Given that~\eqref{eq:riemannsum} is a convex objective, a natural choice is to learn $\mathbf{w}, \mathbf{W}$ through convex optimization. However, computation of the above estimators require fast evaluation of kernel matrix-vector products which are not supported by traditional Python convex optimization libraries. Instead we efficiently evaluate these summations on GPUs with the KeOps Python libraries~\citep{JMLR:v22:20-275} and obtain derivatives with respect to $\mathbf{w}, \mathbf{W}$ which seamlessly integrate with the PyTorch automatic differentiation library~\citep{NEURIPS2019_9015}. Optimization of these weights is performed by Adam with amsgrad, a popular gradient descent method, which uses gradients from previous iterations to stabilize its convergence~\citep{kingma2014adam, reddi2019convergence}. As an additional benefit, Pytorch libraries offer a \textit{parametrizations} class which allows us to easily constrain our RHKS estimators $\beta_\mathbf{w}, G_\mathbf{W}$ to satisfy various properties such as coordinate symmetries and time causality (see Section~\ref{sec:examples}).

We train the estimators $\beta_\mathbf{w}, G_\mathbf{W}$ stochastically on batches of size 100 with $n =$  100-500 training pairs $(F_i, U_i) \in \mathbb{R}^{m_y}, \mathbb{R}^{m_x}$ using between 100-1000 epochs such that the solution converges. The loss function minimized by gradient descent on the training data is
\begin{equation}\label{eq:opt_cost}
    \text{Loss}(\cT_{\beta_\mathbf{w}, G_\mathbf{W}}) = \text{MSE}(\cT_{\beta_\mathbf{w}, G_\mathbf{W}}) + \lambda P(\beta_\mathbf{w}) + \rho J(G_\mathbf{W})
\end{equation}
which is the mean squared $L^2$ error regularized by the RKHS norms of the Green's function and bias term $P(\beta) = \|\beta\|_Q$ and $J(G) = \|G\|_K$. The mean squared error above is approximated by the Riemann sum
\begin{equation}
    \text{MSE}(\cT) = \frac{1}{n}\sum_{i=1}^n \int_{D_\cY}\Big(U_i(y) - [\cT(F_i)](y)\Big)^2\ud y \approx \frac{1}{n}\sum_{i=1}^n \sum_{k=1}^{m_y}\Big(U_i(y_k) - [\cT(F_i)](y_k)\Big)^2\Delta_k^y
\end{equation}
and the regularization terms can be exactly evaluated as
\begin{equation}
\begin{gathered}
    P(\beta_\mathbf{w}) = \sum_{j=1}^{m_y}\sum_{l=1}^{m_y} w_jQ(y_j, y_l)w_l\Delta_j^y\Delta_l^y\\
    J(G_\mathbf{W}) = \sum_{i=1}^{m_x}\sum_{j=1}^{m_y}\sum_{k=1}^{m_x}\sum_{l=1}^{m_y} W_{ij}K(x_i, y_j, x_k, y_l)W_{kl}\Delta_i^x\Delta_j^y\Delta_k^x\Delta_l^y.
\end{gathered}
\end{equation}

Since the samples $(F_i, U_i)$ generated can vary significantly in magnitude, to evaluate the performance of our estimator $\cT_{\mathbf{w}, \mathbf{W}}$ we also investigate its $L^2$ mean relative error which is computed by
\begin{equation}\label{eq:forwardmap_re}
    \text{RE}(\cT) = \sqrt{\frac{1}{n}\sum_{i=1}^n \frac{\Big\|U_i - \cT(F_i)\Big\|_{L^2(D_\cY)}}{\Big\|U_i\Big\|_{L^2(D_\cY)}}} \approx \sqrt{\frac{1}{n}\sum_{i=1}^n\frac{\sum_{k=1}^{m_y}\Big(U_i(y_k) - [\cT(F_i)](y_k)\Big)^2\Delta_k^y}{\sum_{k=1}^{m_y}U_i(y_k)^2\Delta_k^y}}.
\end{equation}

Finally, given new meshes $\{\overline{x}_j\}_{j=1}^{\overline{m}_x}, \{\overline{y}_k\}_{k=1}^{\overline{m}_y}$ with quadrature weights $\overline{\Delta}_j^x, \overline{\Delta}_k^y$ we can extrapolate the predictions of $\cT_{\mathbf{w}, \mathbf{W}}: f \to u$ on these new meshes by writing
\begin{equation}
    [\cT_{\mathbf{w}, \mathbf{W}}(f)](y) = [\cT_{\beta_\mathbf{w}, G_\mathbf{W}}(f)](y) = \beta_\mathbf{w}(y) + \sum_{j=1}^{\overline{m}_x} G_\mathbf{W}(\overline{x}_j, y)f(\overline{x}_j)\overline{\Delta}_j^x.
\end{equation}
The mean and relative squared errors on these new meshes can be computed in the same way as shown above.

In Section~\ref{sec:examples} we implement the Green's function and bias term estimators described above and learn the solution maps of various space and time-varying PDEs. In all examples, the domains $D_\cX, D_\cY$ of the input and output data are rectangular domains of the form $\Pi [a_i, b_i]$ and the discretization meshes $\{x_j\}_{j=1}^{m_x}, \{y_k\}_{k=1}^{m_y}$ are equispaced with quadrature weights $\Delta_j^x, \Delta_k^y$ defined by the trapezoid rule. Our approach however can easily be extended to nonuniform grids by setting suitable quadrature weights or using Monte Carlo integration methods such as importance sampling.

In examples where we know the true Green's function and bias term, we can also compute the relative error of our estimated Green's functions and bias terms $G, \beta$ to the true functions $G_\text{true}, \beta_\text{true}$. In practice, we compute these relative errors as Riemann sums
\begin{equation}\label{eq:estimator_re}
\begin{gathered}
    \text{RE}(G) = \sqrt{\frac{\Big\|G - G_\text{true}\Big\|_{L^2(D_\cX \times D_\cY)}}{\Big\|G_\text{true}\Big\|_{L^2(D_\cX \times D_\cY)}}} \approx \sqrt{\frac{\sum_{j=1}^{\overline{m}_x}\sum_{k=1}^{\overline{m}_y}\Big(G(\overline{x}_j, \overline{y}_k) - G_\text{true}(\overline{x}_j, \overline{y}_k)\Big)^2\overline{\Delta}_j^x\overline{\Delta}_k^y}{\sum_{j=1}^{\overline{m}_x}\sum_{k=1}^{\overline{m}_y}G_\text{true}(\overline{x}_j, \overline{y}_k)^2\overline{\Delta}_j^x\overline{\Delta}_k^y}}\\
    \text{RE}(\beta) = \sqrt{\frac{\Big\|\beta - \beta_\text{true}\Big\|_{L^2(D_\cY)}}{\Big\|\beta_\text{true}\Big\|_{L^2(D_\cY)}}} \approx \sqrt{\frac{\sum_{k=1}^{\overline{m}_y}\Big(\beta(\overline{y}_k) - \beta_\text{true}(\overline{y}_k)\Big)^2\overline{\Delta}_k^y}{\sum_{k=1}^{\overline{m}_y}\beta_\text{true}(\overline{y}_k)^2\overline{\Delta}_k^y}}
\end{gathered}
\end{equation}
where the uniform meshes $\{\overline{x}_j\}_{j=1}^{\overline{m}_x}, \{\overline{y}_k\}_{k=1}^{\overline{m}_y}$ are finely discretized with mesh sizes 10 times larger (in each dimension) than the meshes $\{x_j\}_{j=1}^{m_x}, \{y_k\}_{k=1}^{m_y}$ on which our estimators $G$ and $\beta$ were trained.

\section{Error Analysis}\label{sec:rates}
In this section, we establish error bounds for the Green's function of a PDE when it is estimated in an RKHS from a finite number of samples. Throughout this section, we assume that the input-output data $\{(F_i, U_i)\}_{i=1}^n \subset L^2(D_\cX) \times L^2(D_\cY)$ are truly functional data and are not discretized on a finite mesh. Furthermore, we limit our theoretical analysis to the simpler case of only estimating the Green's function of a PDE. We assume that the bias term $\beta \in \cB$ is known and hence, can be subtracted from the observations $U$, which significantly simplifies the notation in our analysis. For the interested reader, we briefly discuss in Section~\ref{subsec:bias_term} below how our theoretical framework can be extended to the case where the bias term is also estimated.

Focusing on the Green's function estimation problem, we can define the empirical risk of our Green's function estimator as
\begin{equation}\label{eq:regemprisk}
    \widehat{R}(G) := \frac{1}{n}\sum_{i=1}^n\Big\|U_i - \int_{D_\cX} G(x, \cdot)F_i(x)\ud x\Big\|_{L^2(D_\cY)}^2.
\end{equation}
the penalized empirical risk as
\begin{equation}\label{eq:penregemprisk}
    \widehat{R}_\lambda(G) := \widehat{R}(G) + \lambda J(G)
\end{equation}
and, most importantly, the population risk as
\begin{equation}
    R(G) := \mathbb{E}\Big[\Big\|U - \int_{D_\cX} G(x, \cdot)F(x)\ud x\Big\|_{L^2(D_\cY)}^2\Big].
\end{equation}
where the expectation above is taken with respect to the randomness of $(F, U)$. This last objective, the population risk, is the quantity of interest that we study in order to bound the prediction error of our Green's function estimator. In Sections~\ref{subsec:mercer_kernels}-\ref{subsec:oracle_ineq} we show how the empirical RKHS estimator
\begin{equation}
    \widehat{G}_{n, \lambda} := \argmin_{G \in \cG}\widehat{R}_\lambda(G)
\end{equation}
compares to any oracle
\begin{equation}
    G_\cG \in \argmin_{G \in \cG}R(G)
\end{equation}
We prove an oracle inequality in Section~\ref{subsec:oracle_ineq} that controls the difference $R(\widehat{G}_{n, \lambda}) - R(G_\cG)$. For appropriately chosen regularizer $\lambda$, this prediction error tends to zero as ${n \to \infty}$. We assume throughout that $F \in L^2(D_\cX), U \in L^2(D_\cY)$ are mean zero random variables.

As shorthand, we denote $G^T \in L^2(D_\cY \times D_\cX)$ as the function $G^T(x, y) = G(y, x)$ for all $x \in D_\cX, y \in D_\cY$. We frequently use the notation $G^T(F) = \int_{D_\cX} G(x, \cdot)F(x) \ud x$ as the Green's function $G(x, y)$ integrated against $F(x)$ in its first coordinate.

\subsection{Accommodating the Bias Term}\label{subsec:bias_term}
The error analysis for $\widehat{G}_{n, \lambda}$ derived in the following sections can be generalized to include the bias term $\beta \in \cB$. In this case, we need to bound the full population risk $R(\widehat{\beta}_{n, \rho, \lambda}, \widehat{G}_{n, \rho, \lambda})$ of the RKHS estimators
\begin{equation}
    \widehat{\beta}_{n, \rho, \lambda}, \widehat{G}_{n, \rho, \lambda} := \argmin_{\beta \in \cB, G \in \cG} \widehat{R}_{\rho, \lambda}(\beta, G).
\end{equation}
Here the full population risk is defined as
\begin{equation}
    R(\beta, G) := \mathbb{E}\Big[\Big\|U - \beta - \int_{D_\cX} G(x, \cdot)F(x)\ud x\Big\|_{L^2(D_\cY)}^2\Big].
\end{equation}
Including the bias term, our full affine operator is
\begin{equation}
    \cT_{\beta, G}(f) = \beta + \int_{D_{\cX}} G(x, \cdot)f(x) \ud x.
\end{equation}
Denote $|D_\cX|$ by the Lebesgue measure of $D_\cX$. By taking our input function $f: \mathbb{R} \to \mathbb{R}$ and appending a constant to obtain $\tilde{f}: \mathbb{R} \to \mathbb{R}^2$ where $\tilde{f}(x) = [f(x), \frac{1}{|D_\cX|}]^T$ we can rewrite our affine operator as
\begin{equation}
    \cT_{\tilde{G}}(\tilde{f}) = \int_{D_{\cX}} \inprod{\tilde{G}(x, \cdot), \tilde{f}(x)}_2 \ud x, \quad \tilde{G}(x, y) = [G(x, y), \beta(y)]^T
\end{equation}
where now $\tilde{G}: D_\cX \times D_\cY \to \mathbb{R}^2$ and $\inprod{\cdot, \cdot}_2$ denotes the standard Euclidean vector inner product. Here, $\tilde{G}$ is an element of a new Cartesian product RKHS of vector-valued functions
\begin{equation}
    \tilde{\cG} := \Big\{[G(x, y), \beta(y)]^T: G \in \cG, \beta \in \cB\Big\}
\end{equation}
with inner product defined as
\begin{equation}
    \inprod{\tilde{G}, \tilde{H}}_{\tilde{\cG}} = \inprod{\beta, \gamma}_\cB + \inprod{G, H}_\cG
\end{equation}
for all $\tilde{G}(x, y) = [G(x, y), \beta(y)]^T$ and $\tilde{H}(x, y) = [H(x, y), \gamma(y)]^T$ in $\tilde{\cG}$ where $G, H \in \cG$ and $\beta, \gamma \in \cB$. As defined in~\citet[Section~2, Theorem~2.1]{micchelli2005learning}, the reproducing kernel of this RKHS of vector-valued functions $\tilde{\cG}$ is $\tilde{K}: (D_\cX \times D_\cY)^2 \to \mathbb{R}^{2 \times 2}$ given by
\begin{equation}
    \tilde{K}(x, y, \xi, \eta) = \begin{bmatrix}K(x, y, \xi, \eta) & 0\\
    0 & Q(y, \eta)\end{bmatrix}
\end{equation}
where, as before, $K: (D_\cX \times D_\cY)^2 \to \mathbb{R}$ and $Q: D_\cY^2 \to \mathbb{R}$ are the reproducing kernels of $\cG$ and $\cB$ respectively.
Using these definitions, the error analysis for Green's functions with no bias term extends to the case when the Green's function and bias term are jointly optimized. In this more general setting, proving error bounds for $\widehat{\beta}_{n, \rho, \lambda}, \widehat{G}_{n\rho\lambda}$ similarly involves studying the eigenvalues and eigenvectors of $\tilde{K}, \tilde{\Sigma}_F$ to bound the population risk $R(\widehat{\beta}_{n, \rho, \lambda}, \widehat{G}_{n\rho\lambda})$. The covariance operator of the input data $\tilde{F} = [F, \frac{1}{|D_\cX|}]^T$ now becomes a matrix-valued function $\tilde{\Sigma}_F: D_\cX^2 \to \mathbb{R}^{2 \times 2}$. Furthermore, the spectra and vector-valued eigenfunctions of $\tilde{K}$ become concatenations of the eigenvalues and eigenfunctions of the original reproducing kernels $K$ and $Q$. In this paper, we choose to avoid these additional notational complexities by studying error bounds for the Green's function only.\\

\subsection{Eigenbases of Mercer Kernels $\Sigma_F$ and $K$}\label{subsec:mercer_kernels}
To derive an oracle inequality that bounds $R(\widehat{G}_{n, \lambda}) - R(G_\cG)$, we make the following assumptions on the covariance operator $\Sigma_F = \mathbb{E}[(F - \mathbb{E}[F]) \otimes (F - \mathbb{E}[F])]$ and the reproducing kernel $K$ of $\cG$.

\begin{assumption}[Mercer Kernels]
    We assume that the covariance operator $\Sigma_F \in L^2(D_\cX \times D_\cX)$ and the reproducing kernel $K \in L^2((D_\cX \times D_\cY)^2)$ are continuous, square integrable, and positive definite. Kernels that satisfy these three conditions are called Mercer kernels. Since $\Sigma_F$ is a Mercer kernel, we know by Mercer's theorem~\cite[Section~2, Theorem~1]{cucker2002mathematical} that it has nonnegative eigenvalues $\mu_1 \geq \mu_2 \geq \dots$ and $L^2$ orthonormal eigenfunctions $\{\phi_k\}_{k=1}^\infty \subseteq L^2(D_\cX)$ with the spectral decomposition
    \begin{equation}
        \Sigma_F(x, \xi) = \sum_{k=1}^\infty\mu_k\phi_k(x)\phi_k(\xi).
    \end{equation}
    Similarly, $K$ is a Mercer kernel so it has nonnegative eigenvalues $\rho_1 \geq \rho_2 \geq \dots$ and $L^2$ orthonormal eigenfunctions $\{\Psi_k\}_{k=1}^\infty \subset L^2(D_\cX \times D_\cY)$ with the spectral decomposition
\begin{equation}\label{eq:K1_mercer}
    K(x, y, \xi, \eta) = \sum_{k=1}^\infty\rho_k\Psi_k(x, y)\Psi_k(\xi, \eta)
\end{equation}
where $\inprod{\Psi_i, \Psi_j}_{L^2(D_\cX \times D_\cY)} = \delta_{ij}$ and $\rho_i\inprod{\Psi_i, \Psi_j}_\cG = \delta_{ij}$.
\end{assumption}

The decay of the eigenvalues of the reproducing kernel $K$ play a key role in the estimation of $\widehat{G}_{n, \lambda}$. The rate at which these eigenvalues decay to zero determines the rate at which $\widehat{G}_{n, \lambda}$ converges to $G_\cG$ in terms of their prediction error.

\begin{assumption}\label{assump:spectradecay}
    The eigenvalues of the reproducing kernel $K$ satisfy $\rho_k \lesssim k^{-r}$ for some $r > \frac{1}{2}$.
\end{assumption}
In Appendix~\ref{app:eig_decay}, we provide several examples of reproducing kernels which have this rate of decay in their spectrum. For the RKHSs considered in the following sections, we focus our analysis on kernels with polynomial decay in their eigenvalues. The same proof technique in Sections~\ref{subsec:symmdiag} and~\ref{subsec:oracle_ineq} for deriving the error bounds can be applied to RKHSs whose kernels have a stricter, exponential decay in their eigenvalues such as smooth radial kernels.\\

\subsection{Simultaneous Diagonalization}\label{subsec:symmdiag}
The variance of our RKHS estimator $\widehat{G}_{n, \lambda}$ is related to the quadratic form $\inprod{(\Sigma_F \otimes I)G, G}_{L^2(D_\cX \times D_\cY)}$ where $\Sigma_F: D_\cX \to D_\cX$ is the covariance operator of the input data and $I: D_\cY \to D_\cY$ is the identity operator of the output data. Likewise, the bias of our estimator is determined by the regularization term $J(\widehat{G}_{n, \lambda}) = \|\widehat{G}_{n, \lambda}\|_{K}^2$ which implies that we need to study the spectrum of the reproducing kernel $K$ of $\cG$. These statements, formalized in Section~\ref{subsec:oracle_ineq}, suggest an approach to studying the bias-variance tradeoff of our estimator. Namely, we approach this problem by simultaneously diagonalizing the operators $\Sigma_F \otimes I$ and $K$. This allows us to write any $G \in \cG$ as a sum of basis functions where $\inprod{(\Sigma_F \otimes I)G, G}_{L^2(D_\cX \times D_\cY)}$ and $J(G) = \|G\|_{K}^2$ are expanded into series whose terms depend on the basis coefficients of $G$.\\

First, for any $G, H \in L^2(D_\cX \times D_\cY)$ define the semi-inner product
\begin{equation}
    \inprod{G, H}_{\Sigma_F} = \inprod{(\Sigma_F \otimes I)G, H}_{L^2(D_\cX \times D_\cY)} = \int_{D_\cX}\int_{D_\cX}\int_{D_\cY} G(x, y)\Sigma_F(x, z)H(z, y)\ud x \ud y \ud z
\end{equation}
Because $K$ is a Mercer kernel, its can be expanded as in~\eqref{eq:K1_mercer}, and then by~\citet[Chapter~3, Section~3]{cucker2002mathematical} we can write
\begin{equation}
    \inprod{G, F}_{\cG} = \inprod{G, F}_{K} = \sum_{k=1}^\infty \frac{\inprod{G, \Psi_k}_{L^2(D_\cX \times D_\cY)}\inprod{F, \Psi_k}_{L^2(D_\cX \times D_\cY)}}{\rho_k}.
\end{equation}

Define a sum of the two inner products
\begin{equation}
    \inprod{G, H}_{\overline{K}} = \inprod{G, H}_{\Sigma_F} + \inprod{G, H}_{K}
\end{equation}
with corresponding norm
\begin{equation}
    \|G\|_{\overline{K}}^2 = \|G\|_{\Sigma_F}^2 + \|G\|_{K}^2 = \|G\|_{\Sigma_F}^2 + J(G).
\end{equation}
Note that $\inprod{\cdot, \cdot}_{\overline{K}}$ is indeed an inner product over $\cG$ because it is clearly linear and conjugate symmetric, and it is positive definite since $\|G\|_{K}^2 \leq \|G\|_{\overline{K}}^2$.

Now we use this new norm $\|\cdot\|_{\overline{K}}$ to define a basis for $\cG$ that simultaneously diagonalizes the quadratic forms $\|G\|_{\Sigma_F}^2$ and $J(G)$. First in Appendix~\ref{app:rates} we prove the following proposition.
\begin{prop}\label{prop:equivnorms}
    The norms $\|\cdot\|_\cG$ and $\|\cdot\|_{\overline{K}}$ are equivalent. Hence, the Hilbert space $\cG$ with inner product $\inprod{\cdot, \cdot}_{\overline{K}}$ is also an RKHS.
\end{prop}
Let us denote $\overline{K}: (D_\cX \times D_\cY)^2 \to \mathbb{R}$ as the reproducing kernel associated with $\|\cdot\|_{\overline{K}}$. The kernel $\overline{K}$ can be viewed as a positive definite operator over the space $\cG$. Denoting the eigenvalues and eigenfunctions of $\overline{K}$ by $\{(\rho_k', \Psi_k')\}_{k=1}^\infty$ we can interpret $\overline{K}: L^2(D_\cX \times D_\cY) \to L^2(D_\cX \times D_\cY)$ as a positive operator defined as $\overline{K}(\Psi_k') := \int_{D_\cX}\int_{D_\cY}\overline{K}(\cdot, \cdot, \xi, \eta)\Psi_k'(\xi, \eta) \ud\xi\ud\eta = \rho_k'\Psi_k'$. Now we can define the square root of this positive operator as $\overline{K}^{\frac{1}{2}}: L^2(D_\cX \times D_\cY) \to L^2(D_\cX \times D_\cY)$ which satisfies $\overline{K}^{\frac{1}{2}}(\Psi_k') = (\rho_k')^{\frac{1}{2}}\Psi_k'$.

To summarize, we have defined above two norms $\|\cdot\|_{\Sigma_F}^2$ and $\|\cdot\|_{\overline{K}}^2$ where the former is the quadratic form of $\Sigma_F \otimes I$ and the latter is roughly the quadratic form of the inverse kernel $\overline{K}^{-1}$. For notational convenience, we have chosen to drop the tensor product with the identity and the kernel inverse from the norm subscripts.

Following the ideas of~\citet[Section~7.6, Theorem~7.6.4]{horn2012matrix} we simultaneously diagonalize the quadratic forms $\|\cdot\|_{\Sigma_F}^2$ and $\|\cdot\|_{\overline{K}}^2$. Defining the linear operator $\overline{K}^\frac{1}{2}(\Sigma_F \otimes I)\overline{K}^\frac{1}{2}$, denote its eigenvalues and $L^2$ orthonormal eigenfunctions by $\nu_1 \geq \nu_2 \geq \dots$ and $\{\Gamma_k\}_{k=1}^\infty \subset \cG$ respectively. Note that $\overline{K}$ is positive definite over $\cG$ and so is $\Sigma_F \otimes I$ by Assumption~\ref{assump:strictsubg} which implies that $\{\Gamma_k\}_{k=1}^\infty$ spans $\cG$ and that $\nu_k > 0$ for all $k \geq 1$. In fact, we have that
\begin{equation}
    \nu_k = \|\overline{K}^\frac{1}{2}\Gamma_k\|_{\Sigma_F}^2 \leq \|\overline{K}^\frac{1}{2}\Gamma_k\|_{\overline{K}}^2 = \|\Gamma_k\|_{L^2(D_\cX \times D_\cY)}^2 = 1.
\end{equation}
Now we define the basis functions
\begin{equation}
    \Omega_k = \nu_k^{-\frac{1}{2}}\overline{K}^\frac{1}{2}\Gamma_k \in \cG, \quad k \geq 1.
\end{equation}
Then using the induced inner product $\inprod{\cdot, \cdot}_{\overline{K}}$ we can write out
\begin{equation}
    \inprod{\Omega_i, \Omega_j}_{\overline{K}} = \inprod{\nu_i^{-\frac{1}{2}}\overline{K}^\frac{1}{2}\Gamma_i, \nu_j^{-\frac{1}{2}}\overline{K}^\frac{1}{2}\Gamma_j}_{\overline{K}} = \inprod{\nu_i^{-\frac{1}{2}}\Gamma_i, \nu_j^{-\frac{1}{2}}\Gamma_j}_{L^2(D_\cX \times D_\cY)} = \nu_i^{-1}\delta_{ij}
\end{equation}
and similarly
\begin{equation}
\begin{split}
    \inprod{(\Sigma_F \otimes I)\Omega_i, \Omega_j}_{L^2(D_\cX \times D_\cY)} &= \inprod{\nu_i^{-\frac{1}{2}}(\Sigma_F \otimes I)\overline{K}^\frac{1}{2}\Gamma_i, \nu_j^{-\frac{1}{2}}\overline{K}^\frac{1}{2}\Gamma_j}_{L^2(D_\cX \times D_\cY)}\\
    &= \nu_i^{-\frac{1}{2}}\nu_j^{-\frac{1}{2}}\inprod{\overline{K}^\frac{1}{2}(\Sigma_F \otimes I)\overline{K}^\frac{1}{2}\Gamma_i, \Gamma_j}_{L^2(D_\cX \times D_\cY)} = \delta_{ij}.
\end{split}
\end{equation}
We emphasize that the basis $\{\Omega_k\}_{k=1}^\infty$ defined above is not an orthogonal basis in $L^2(D_\cX \times D_\cY)$ but it does form an orthogonal basis for $\cG$ as we shall show next. Furthermore, we can simultaneously diagonalize the quadratic forms $\|\cdot\|_{\Sigma_F}^2$ and $\|\cdot\|_{\overline{K}}^2$ on the basis $\{\Omega_k\}_{k=1}^\infty$.
\begin{theorem}\label{thm:simdiag}
For any $G \in \cG$,
\begin{equation}
    G = \sum_{k=1}^\infty g_k\Omega_k
\end{equation}
which converges absolutely and where $g_k = \nu_k\inprod{G, \Omega_k}_{\overline{K}}$. Furthermore, setting $\gamma_k = (\nu_k^{-1} - 1)^{-1}$ then we can write
\begin{equation}
    \|G\|_{\overline{K}}^2 = \sum_{k=1}^\infty \nu_k^{-1}g_k^2 = \sum_{k=1}^\infty (1 + \gamma_k^{-1})g_k^2, \qquad \|G\|_{\Sigma_F}^2 = \sum_{k=1}^\infty g_k^2
\end{equation}
and similarly
\begin{equation}
    J(G) = \|G\|_{\overline{K}}^2 - \|G\|_{\Sigma_F}^2 = \sum_{k=1}^\infty \gamma_k^{-1}g_k^2.
\end{equation}
\end{theorem}
We do not necessarily assume that the $\gamma_k$ coefficients are ordered in decreasing order since the series in Theorem~\ref{thm:simdiag} converge absolutely. To acquire some intuition for this simultaneous diagonalization we discuss it in a useful setting where the operators $K$ and $\Sigma_F \otimes I$ commute.

\begin{prop}\label{prop:simdiagsimple}
Recall that the eigenvalues and eigenbasis of $\Sigma_F$ are $\mu_i, \phi_i \in L^2(D_\cX)$ for $i \geq 1$. Assume that the eigenbasis for $K$ is $\{\Psi_{ij} := \phi_i \otimes \varphi_j\}_{k = 1}^\infty \subseteq \cG$ with eigenvalues $\{\rho_{ij}\}_{i, j=1}^\infty$ where $\{\varphi_j\}_{j=1}^\infty$ is any orthonormal basis of $L^2(D_\cY)$. Then we know that $\{\gamma_{ij} := \mu_i\rho_{ij}\}_{i, j=1}^\infty$ and $\{\Omega_{ij} := \mu_i^{-\frac{1}{2}}\Psi_{ij}\}_{i, j=1}^\infty$ are the coefficients and basis functions given in Theorem~\ref{thm:simdiag}.\\

Sort the eigenvalues $\rho_{ij}$ of $K$ in decreasing order as $\rho_1 \geq \rho_2 \geq \hdots$ and assume that $\rho_k \lesssim k^{-r}$ for some $r > \frac{1}{2}$ as in Assumption~\ref{assump:spectradecay}. If we enumerate the coefficients $\gamma_k$ and basis functions $\Omega_k$ in the same order as the decreasing $\rho_k$, then it holds that $\gamma_k \lesssim \rho_k \lesssim k^{-r}$.
\end{prop}

We prove the proposition above in Appendix~\ref{app:rates}. As an example, the assumptions in this proposition hold when $D_\cX = D_\cY = [0, 1]$ and $F$ is a Brownian bridge with variance $\sigma^2$. In this case, $\Sigma_F(x, \xi) = \sigma^2[\min(x, \xi) - x\xi]$ for any $x, \xi \in [0, 1]$ with eigenvalues $\mu_i = \frac{\sigma^2}{\pi i}$ and eigenfunctions $\phi_i(x) = \sqrt{2}\sin(\pi i x)$. If we choose $\cG = W_1^2([0, 1]^2)$ to be the space of Sobolev-1 Green's functions on $D_\cX \times D_\cY = [0, 1]^2$ with Dirichlet boundary conditions at $\partial(D_\cX \times D_\cY)$, then as described in Section~\ref{subsec:kernel_examples} it has reproducing kernel
\begin{equation}
    K(x, y, \xi, \eta) = 2\sum_{i=1}^\infty\sum_{j=1}^\infty \frac{\sin(\pi ix)\sin(\pi jy)\sin(\pi i\xi)\sin(\pi j\eta)}{\pi^2(i^2 + j^2)}
\end{equation}
whose eigenvalues are $\rho_{ij} = \frac{1}{2\pi^2(i^2 + j^2)}$ with eigenfunctions $\Psi_{ij}(x, y) = \sqrt{2}\sin(\pi ix) \cdot \sqrt{2}\sin(\pi jy)$. Since we know that $\Psi_{ij} = \phi_i \otimes \phi_j$ then the assumptions of Proposition~\ref{prop:simdiagsimple} are satisfied. Because $K$ is a bounded Mercer kernel on $[0, 1]^2$ then from Example 1 of Appendix~\ref{app:eig_decay} we can check that $\rho_k \lesssim k^{-1}$. Finally, by applying Proposition~\ref{prop:simdiagsimple} this tells us that $\gamma_k \lesssim k^{-1}$.~\\

In order to prove an oracle inequality for $\widehat{G}_{n, \lambda}$, we must require the coefficients of the simultaneous diagonalization $\gamma_k$ to decrease at a certain rate. Here we take inspiration from the setting discussed in Proposition~\ref{prop:simdiagsimple} and make the following assumption on the coefficients $\gamma_k$.
\begin{assumption}\label{assump:gammadecay}
    When simultaneously diagonalizing $\|\cdot\|_{\overline{K}}$ and $\|\cdot\|_{\Sigma_F}$ in Theorem~\ref{thm:simdiag}, we assume that $\gamma_k \lesssim k^{-r}$ for some $r > \frac{1}{2}$.
\end{assumption}

\subsection{Oracle Inequality for $\widehat{G}_{n, \lambda}$}\label{subsec:oracle_ineq}
Now we are in the setting to prove the oracle inequality which bounds the difference of the expected risks $R(\widehat{G}_{n, \lambda}) - R(G_\cG)$ between our Green's function estimator and the oracle. Here we again use the notation
\begin{equation}
    G^T(F) = \int_{D_\cX}G(x, y)F(x)\ud x
\end{equation}
to denote a Green's function $G \in L^2(D_\cX \times D_\cY)$ integrated against an input $F \in L^2(D_\cX)$.

By definition of the oracle $G_\cG \in \argmin_{G \in \cG}R(G)$, we know that $\mathbb{E}\|U - G(F)\|_{L^2(D_\cY)}^2$ for all $G \in \cG$ is minimized at $G = G_\cG$. Hence, for any $G \in \cG$ by the optimality of $G_\cG$ and the Pythagorean theorem we see that
\begin{equation}
\begin{split}
    R(G) &= \mathbb{E}\|U - G^T(F)\|_{L^2(D_\cY)}^2\\
    &= \mathbb{E}\|U - G_\cG^T(F)\|_{L^2(D_\cY)}^2 + \mathbb{E}\|G^T(F) - G_\cG^T(F)\|_{L^2(D_\cY)}^2\\
    &= R(G_\cG) + \|G - G_\cG\|_{\Sigma_F}^2
\end{split}
\end{equation}
In particular, setting $G = \widehat{G}_{n, \lambda}$ in the equation above, we need to bound the norm $\|\widehat{G}_{n, \lambda} - G_\cG\|_{\Sigma_F}^2$. First we define the intermediate oracle
\begin{equation}
    \overline{G}_{\infty, \lambda} = \argmin_{G \in \cG}\Big\{R(G) + \lambda J(G)\Big\}
\end{equation}
which is unique since $\|G - G_\cG\|_{\Sigma_F}$ is strictly convex because $\Sigma_F \otimes I$ is positive definite over $\cG$. Then by the Cauchy--Schwarz inequality we can decompose
\begin{equation}\label{eq:triangleineq}
    \|\widehat{G}_{n, \lambda} - G_\cG\|_{\Sigma_F}^2 \lesssim \underbrace{\|\widehat{G}_{n, \lambda} - \overline{G}_{\infty, \lambda}\|_{\Sigma_F}^2}_\text{stochastic error} + \underbrace{\|\overline{G}_{\infty, \lambda} - G_\cG\|_{\Sigma_F}^2}_\text{deterministic error}.
\end{equation}

We bound the deterministic error simply by writing
\begin{equation}
    \|\overline{G}_{\infty, \lambda} - G_\cG\|_{\Sigma_F}^2 = R(\overline{G}_{\infty, \lambda}) - R(G_\cG)
\end{equation}
Since we know that $R(\overline{G}_{\infty, \lambda}) + \lambda J(\overline{G}_{\infty, \lambda}) \leq R(G_\cG) + \lambda J(G_\cG)$ then this proves that
\begin{equation}\label{eq:deterrbound}
    \|\overline{G}_{\infty, \lambda} - G_\cG\|_{\Sigma_F}^2 \leq \lambda\Big(J(G_\cG) - J(\overline{G}_{\infty, \lambda})\Big) \leq \lambda J(G_\cG).
\end{equation}

For the stochastic error term, it takes a bit more work to show the following bound.
\begin{lemma}\label{lem:stocherrbound}
    If $\frac{\log(1/\delta)}{n}\lambda^{-\frac{1}{r}} \lesssim 1$, then for the estimator $\widehat{G}_{n, \lambda}$ we have that
    \begin{equation}
        \|\widehat{G}_{n, \lambda} - \overline{G}_{\infty, \lambda}\|_{\Sigma_F}^2 \lesssim \max\Big(1, \|G_\cG\|_\text{op}, \lambda J(G_\cG)\Big)\frac{\log(1/\delta)}{n}\lambda^{-\frac{1}{r}}
    \end{equation}
    with probability at least $1 - \delta$.
\end{lemma}
Combining~\eqref{eq:triangleineq},~\eqref{eq:deterrbound} and Lemma~\ref{lem:stocherrbound} we get
\begin{equation}
\begin{split}
    \|\widehat{G}_{n, \lambda} - G_\cG\|_{\Sigma_F}^2 &\lesssim \max\Big(1, \|G_\cG\|_\text{op}, \lambda J(G_\cG)\Big)\frac{\log(1/\delta)}{n}\lambda^{-\frac{1}{r}} + \lambda J(G_\cG)\\
    &\lesssim \max\Big(1, \|G_\cG\|_\text{op}\Big)\frac{\log(1/\delta)}{n}\lambda^{-\frac{1}{r}} + \lambda J(G_\cG)
\end{split}
\end{equation}
since we assumed that $\frac{\log(1/\delta)}{n}\lambda^{-\frac{1}{r}} \lesssim 1$. Taking $\lambda \asymp (\frac{n}{\log(1/\delta)})^{-\frac{r}{r+1}}$ yields
\begin{equation}
\begin{split}
    \|\widehat{G}_{n, \lambda} - G_\cG\|_{\Sigma_F}^2 &\lesssim \max\Big(1, \|G_\cG\|_\text{op}, J(G_\cG)\Big)\Big(\frac{n}{\log(1/\delta)}\Big)^{-\frac{r}{r+1}}\\
    &\lesssim \max\Big(1, \|G_\cG\|_\text{op}, J(G_\cG)\Big)n^{-\frac{r}{r+1}}\log(1/\delta)
\end{split}
\end{equation}
with probability at least $1 - \delta$. Therefore, we have shown the following oracle inequality.
\begin{theorem}[Oracle Inequality]
\label{thm:oracle}
    The regularized RKHS estimator $\widehat{G}_{n, \lambda}$ from Theorem~\ref{thm:representer} with $\lambda \asymp \Big(\frac{n}{\log(1/\delta)}\Big)^{-\frac{r}{r+1}}$ satisfies the oracle inequality
    \begin{equation}
        R(\widehat{G}_{n, \lambda}) \leq R(G_\cG) +  C\max\Big(1, \|G_\cG\|_\text{op}, J(G_\cG)\Big)n^{-\frac{r}{r+1}}\log(1/\delta)
    \end{equation}
    with probability at least $1 - \delta$ for some numerical constant $C > 0$.
\end{theorem}

As a final note, if the reproducing kernel $K$ of $\cG$ is squared exponential or generally a smooth radial kernel, its eigenvalues decay exponentially~\citep[Section~3, Theorem~2]{belkin2018approximation}. This implies that $\gamma_k \lesssim \rho_k \lesssim k^{-r}$ for all $r > 0$. Hence, the optimal choice of regularizer is $\lambda \asymp \frac{\log(1/\delta)}{n}$ and the prediction error in Theorem~\ref{thm:oracle} above has rate $R(\widehat{G}_{n, \lambda}) - R(G_\cG) \lesssim \frac{\log(1/\delta)}{n}$.

\section{Examples}\label{sec:examples}
In this section, we show how the Green's functions several linear PDEs can be learned in an RKHS using the estimators defined in Section~\ref{sec:implementation}. On several examples, we show how RKHSs can be designed to enforce important physical constraints such as coordinate symmetries, time causality, and time invariance in these Green's function estimators.

\subsection{Poisson Equation}\label{subsec:poisson}
We begin with the one-dimensional Poisson equation 
\begin{equation}\label{eq:poisson}
    -\Delta u(x) = f(x) \ \text{on} \ D = [0, 1], \qquad u(0) = -0.1, \ u(1) = 0.1
\end{equation}
with Dirichlet boundary conditions. The input and output domains of $f(x)$ and $u(y)$ respectively are $D_\cX = D_\cY = [0, 1]$. The random input forcings $f(x)$ are generated from a squared exponential KLE with lengthscale $\ell = 0.01$ (see Appendix~\ref{app:dataexp} for details) and the corresponding solutions $u(y)$ are simulated with a finite difference solver. All input and output functions are discretized on a uniform $m_x = m_y = 100$ point grid on the unit interval. From this procedure we build up $n = 500$ input-output pairs $\{(f_i, u_i)\}_{i=1}^n$ on which we learn the Green's function $G(x, y)$ and bias term $\beta(y)$ of the Poisson equation. Here we do not add any noise to our data as we are interested in perfectly recovering the true Green's function and bias term of our PDE. In this example, we know the true form of the bias term is $\beta_\text{Poisson}(y) = 0.2x - 0.1$ and the Green's function is $G_\text{Poisson}(x, y)$ as given in~\eqref{eq:poisson_greens_function}.

In Figure~\ref{fig:poisson_kernels} we compute the relative error of our learned Green's function and bias term estimators from~\eqref{eq:estimator_re}, as well as their combined relative error on the train data from~\eqref{eq:forwardmap_re}, and we study how these errors behave as a function of training epochs in our optimization (top row). Each line color in the top row represents a choice of kernel $K(x, y, \xi, \eta)$ and $Q(y, \eta)$ which defines the RKHS $\cG$ for our Green's function and the RKHS $\cB$ for our bias term respectively. We take $K$ and $Q$ to be of the same type for each line plot (both exponential, both Mat\'{e}rn, etc.) and refer the reader to Section~\ref{subsec:kernel_examples} for the definitions of these kernels. The kernel lengthscales of $K$ and $Q$ in the $x$ and $y$ directions are set to $\sigma_x = \sigma_y = 0.2$ (see Appendix~\ref{app:kernel_lengthscale} for details). In the second and third rows, we plot the absolute difference of the learned estimators (after 500 training epochs) compared to the true Green's function and bias term of the Poisson equation.

From these results, we find that all estimators of the Green's function incur the most error around the diagonal $x = y$ where the function is not smooth. The exponential RKHS estimator gives the best approximation near the diagonal but suffers large approximation errors away from the diagonal where the Green's function is very smooth. The smoother squared exponential and Mat\'{e}rn 5/2 kernels give an improved fit away from the diagonal. Finally, the Mat\'{e}rn 3/2 gives the best estimator of the Green's function as it balances the degree of smoothness correctly and is able to nicely approximate the function both near and away from the diagonal $x = y$. All RKHS kernel estimators of the bias term provide a reasonable fit to the true linear bias term with expected ringing phenomena near the boundaries of the domain.

\begin{figure}[h]
\centering
\includegraphics[width=\textwidth]{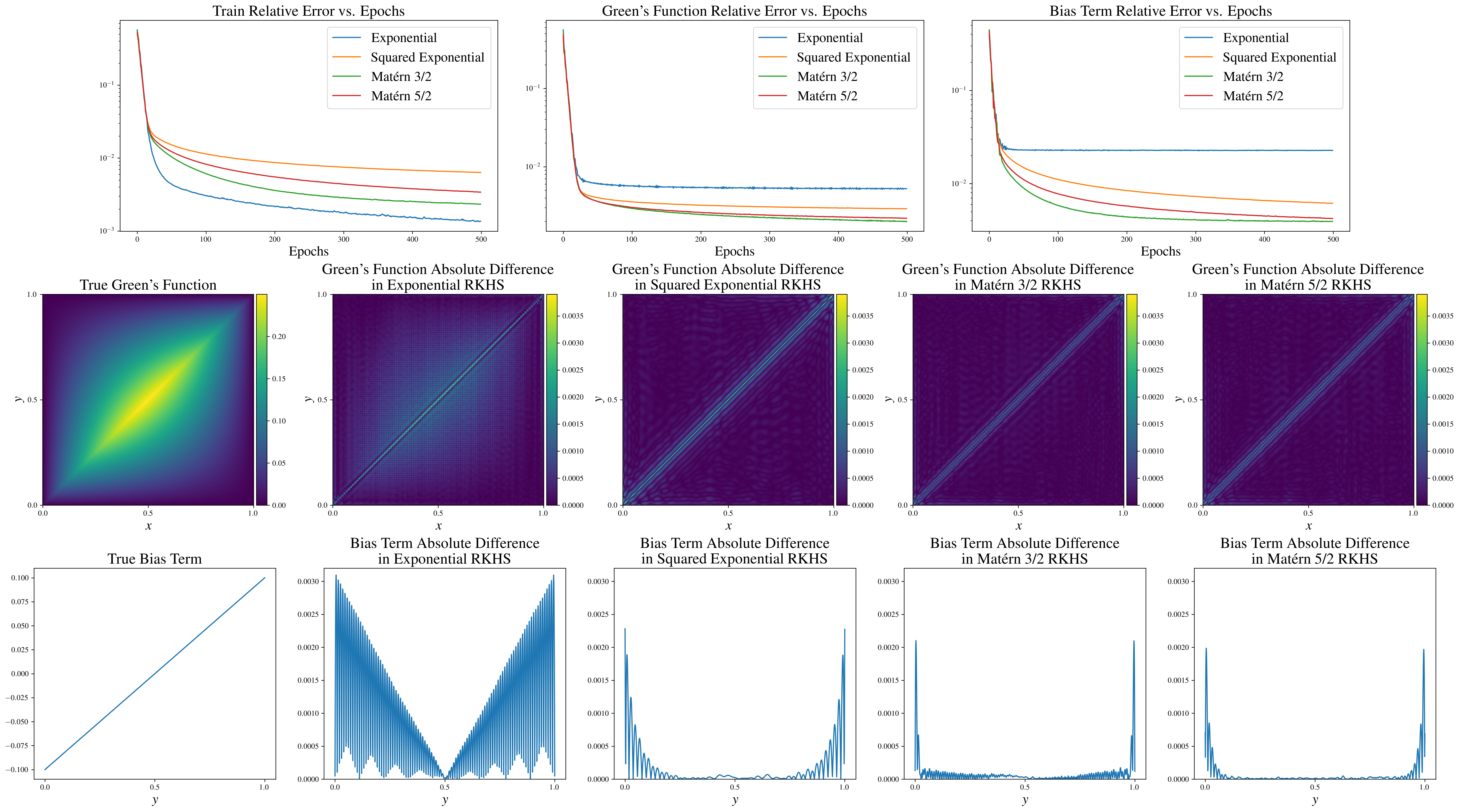}
\caption{Learning the Green's function and bias term of the Poisson equation in exponential, squared exponential, and Mat\'{e}rn RKHSs. Top row shows the relative error of the Green's function and bias term estimators as well as their combined relative error on the train data as a function of training epochs. Middle and bottom row show absolute differences of learned Green's functions and bias terms in each RKHS.}\label{fig:poisson_kernels}
\end{figure}

\subsection{Helmholtz Equation and Coordinate Symmetries}\label{subsec:helmholtz}
Now we study the one-dimensional Helmholtz equation 
\begin{equation}\label{eq:helmholtz}
    -\Delta u(x) - \omega^2u(x) = f(x) \ \text{on} \ D = [0, 1], \qquad u(0) = -0.1, \ u(1) = 0.1
\end{equation}
with Dirichlet boundary conditions and high wavenumber $\omega = 20$. Here the input and output domains for $f(x)$ and $u(y)$ are simply the unit interval $D_\cX = D_\cY = [0, 1]$. The input forcings $f(x)$ to the Helmholtz equation are generated from a squared exponential KLE with lengthscale $\ell = 0.01$ and the corresponding solutions $u(y)$ are simulated with a finite difference solver and corrupted with a fixed amount of i.i.d additive Gaussian noise as in the previous example. We are interested in learning the Green's function and bias term of this PDE, where we know their closed form expressions to be
\begin{equation}
\begin{gathered}
    G_\text{Helmholtz}(x, y) = 2\sum_{k=1}^\infty\frac{\sin(\pi kx)\sin(\pi ky)}{\pi^2k^2 - \omega^2}\\ \beta_\text{Helmholtz}(y) = 0.1\frac{(1+\cos(\omega))}{\sin(\omega)}\sin(\omega y) - 0.1\cos(\omega y).
\end{gathered}
\end{equation}
Note that the true Green's function of the Helmholtz equation is coordinate symmetric $G(x, y) = G(y, x)$ due to the self-adjointness of the differential operator $-(\Delta + \omega^2)$.

In Figure~\ref{fig:helmholtz} we study the robustness of our learning method to the number of samples, mesh discretization, and noise corruption in the Helmholtz equation. In these experiments, we choose as a reference $n = 100$ input-output samples, set a uniform mesh discretization at $m = m_x = m_y = 100$ points, and begin with no noise corruption. Then fixing two of the three parameters to their reference values, we vary one-at-a-time the number of input samples from $n = 1$ to $100$, the mesh discretization from $m = 1$ to $100$, and the noise corruption from $p = 0\%$ to $50\%$. We set the regularization parameter to $\lambda = 10^{-5}$ except for the noise experiments where we set $\lambda = 10^{-3}$ to achieve better noise robustness. For all experiments, our estimators for the Green's function $G(x, y)$ and bias term $\beta(y)$ are learned in a Mat\'{e}rn 5/2 RKHS with kernels $K(x, y, \xi, \eta)$ and $Q(y, \eta)$ defined as in Section~\ref{subsec:kernel_examples} with lengthscales $\sigma_x = \sigma_y = 0.02$.

In Figure~\ref{fig:helmholtz} we compute the relative errors of our Green's function and bias term estimators compared to their true functional form as defined in~\eqref{eq:estimator_re}. We see that the relative errors (shown in blue) decrease exponentially as a function of the number of input-ouput samples (n), and their mesh discretization/measurements (m). Most importantly, in the third column we show that our estimators (blue lines) are noise robust and scale linearly with the amount of noise present in the train data.

We perform an identical set of experiments where we explicitly enforce the coordinate symmetry $G(x, y) = G(y, x)$ in our Green's function estimator due to the self-adjointness of the differential operator in the Helmholtz equation. As described in Appendix~\ref{app:coord_symm}, this is done by transforming the Mat\'{e}rn 5/2 reproducing kernel $K(x, y, \xi, \eta)$ of our Green's function estimator into
\begin{equation}
    K_\text{symm}(x, y, \xi, \eta) = \frac{1}{4}\Big[K(x, y, \xi, \eta) + K(x, y, \eta, \xi) + K(y, x, \xi, \eta) + K(y, x, \eta, \xi)\Big].
\end{equation}
The Green's function $G$ of the RKHS defined by this symmetrized kernel will necessarily be coordinate symmetric. As shown in the first column of Figure~\ref{fig:helmholtz}, enforcing this coordinate symmetry into our Green's function leads to faster convergence of both the Green's function and bias term estimators when the number of samples are increased (orange lines). Furthermore, symmetrizing the Green's function significantly improves its robustness to noise as shown in the top rightmost plot.

\begin{figure}[h]
\centering
\centering
\includegraphics[width=\textwidth]{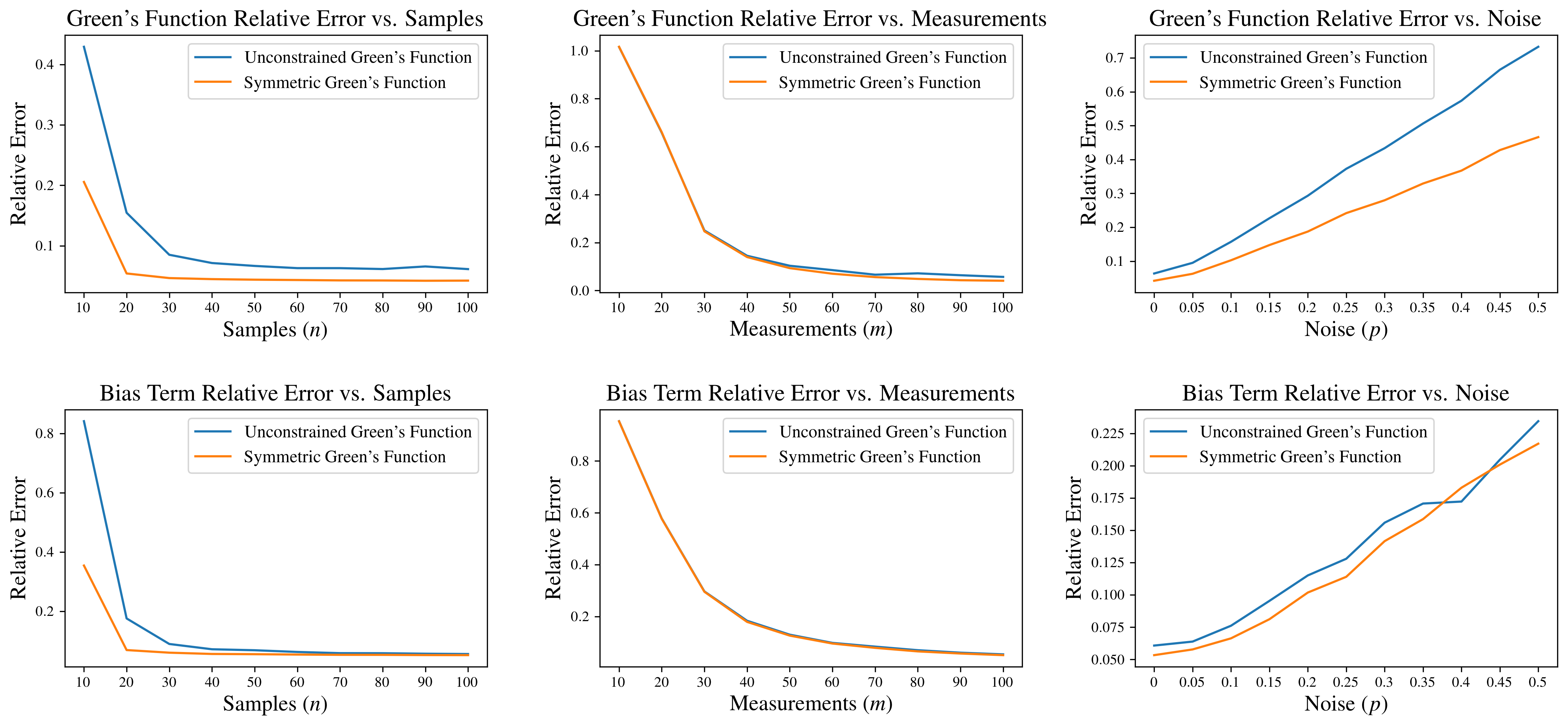}
\caption{Plots of the relative errors of the Green's function and bias term estimators on the Helmholtz equation as the number of sample (n), measurements (m), and noise corruption (p) are increased. Two sets of identical experiments are performed where in the first experiment the Green's function is learned in a Mat\'{e}rn 5/2 RKHS with no constraints, and in a second experiment in a symmetrized Mat\'{e}rn 5/2 RKHS such that the coordinate symmetries $G(x, y) = G(y, x)$ are enforced.}\label{fig:helmholtz}
\end{figure}

\subsection{Schr\"{o}dinger Equation}\label{subsec:schrodinger}
We now study the time-independent Schr\"{o}dinger equation in two spatial dimensions
\begin{equation}
	\Delta u(x_1, x_2) - V(x_1, x_2)u(x_1, x_2) = 0 \ \text{on} \ D = [0, 1]^2, \qquad u(x_1, x_2) = b(x_1, x_2) \ \text{on} \ \partial D
\end{equation}
where we are interested in learning the map from the boundary condition $b(x_1, x_2)$ on domain $D_\cX = \partial D$ to the solution $u(y_1, y_2)$ on domain $D_\cY = D$. Instead of defining the boundary conditions $b(x_1, x_2)$ as a function on the unit square $D$, we parametrize them as a function of arc length $b(x)$ for $x \in [0, 4]$ clockwise along the boundary of the unit square $\partial D$. Random input boundary conditions are generated by a KLE with squared exponential \textit{periodic} kernel of lengthscale $\ell = 0.01$ and the corresponding solutions $u(y)$ are simulated with a finite difference solver and corrupted with $20\%$ i.i.d additive Gaussian noise (see Appendix~\ref{app:dataexp}). The output solutions $u$ are discretized at $m_y = m = 50$ uniform grid points in both the $y_1$ and $y_2$ dimensions. Since the boundary conditions $b(x)$ wrap around the unit square, we discretize them correspondingly at $m_x = 4m-3 = 197$ uniform grid points.

The potential function $V(x_1, x_2)$ for this example depicted in Figure~\ref{fig:schrodinger} (top left) is chosen as a step function which is positive in a hexagonal-shaped well and zero outside. In Figure~\ref{fig:schrodinger} (top center and right) we learn the Green's function $\widehat{G}(x, y_1, y_2): \partial D \to D$ mapping the boundary condition $b(x)$ of the PDE to the solution $u(y_1, y_2)$ in a Mat\'{e}rn 5/2 RKHS with kernel
\begin{equation}
	K(x, y_1, y_2, \xi, \eta_1, \eta_2) = C_\nu\Big(\sqrt{\frac{(x - \xi)^2}{\sigma_x^2} + \frac{(y_1 - \eta_1)^2}{\sigma_y^2} + \frac{(y_2 - \eta_2)^2}{\sigma_y^2}}\Big).
\end{equation}
for lengthscales $\sigma_x = \sigma_y = 0.02$. Here $\nu = 5/2$ where $C_\nu$ is the Mat\'{e}rn covariance function defined in~\eqref{eq:matern} of Section~\ref{subsec:kernel_examples}. We train on 500 noisy samples $(b_i, u_i)$ and regularize our estimator with $\lambda = 10^{-4}$ penalty which allows us to learn a smooth Green's function even with $20\%$ noise in our output samples.

In the top center plot of Figure~\ref{fig:schrodinger} we see that the estimator has learned an impulse response from perturbing the boundary condition at a given point. By integrating the learned Green's function along the boundary $\partial D$ of the unit square (top right), we clearly see the hexagonal shape of the potential $V(x_1, x_2)$ implying that solutions of the PDE have smaller magnitude in this hexagonal region as expected from the form of the PDE. In the bottom of Figure~\ref{fig:schrodinger} we study how our learned Green's function estimator performs on 500 new test samples when we vary the lengthscale $\ell$ of the boundary condition from $0.01$ to $10.0$ and the mesh discretization from $m = 50$ to $150$. We remind the reader that the mesh discretization in the $y_1$ and $y_2$ directions both scale as $m_y = m$ and the mesh discretization of the boundary condition scales as $m_x = 4m-3$. From the botom plot of Figure~\ref{fig:schrodinger} we see that our learned Green's function is independent of mesh discretization as the relative test error plateaus quickly as we increase $m$. Furthermore, the relative error of our estimator only decreases as we raise the boundary condition lengthscale from $\ell = 0.1$ on which it was trained to $\ell = 10.0$ (e.g. very smooth boundary conditions). When we evaluate our Green's function estimator on test boundary conditions with lengthscale  below $\ell = 0.1$ (on which it was trained), then the predictive ability of our estimator gradually worsens. This behavior is expected as the Green's function estimator cannot make perfect predictions on inputs which exceed the lengthscale of the boundary conditions it was trained. We include example predictions of our Green's function estimator for input boundary conditions of several lengthscales in Appendix~\ref{app:test_pred}

\begin{figure}[h]
\centering
\includegraphics[width=\textwidth]{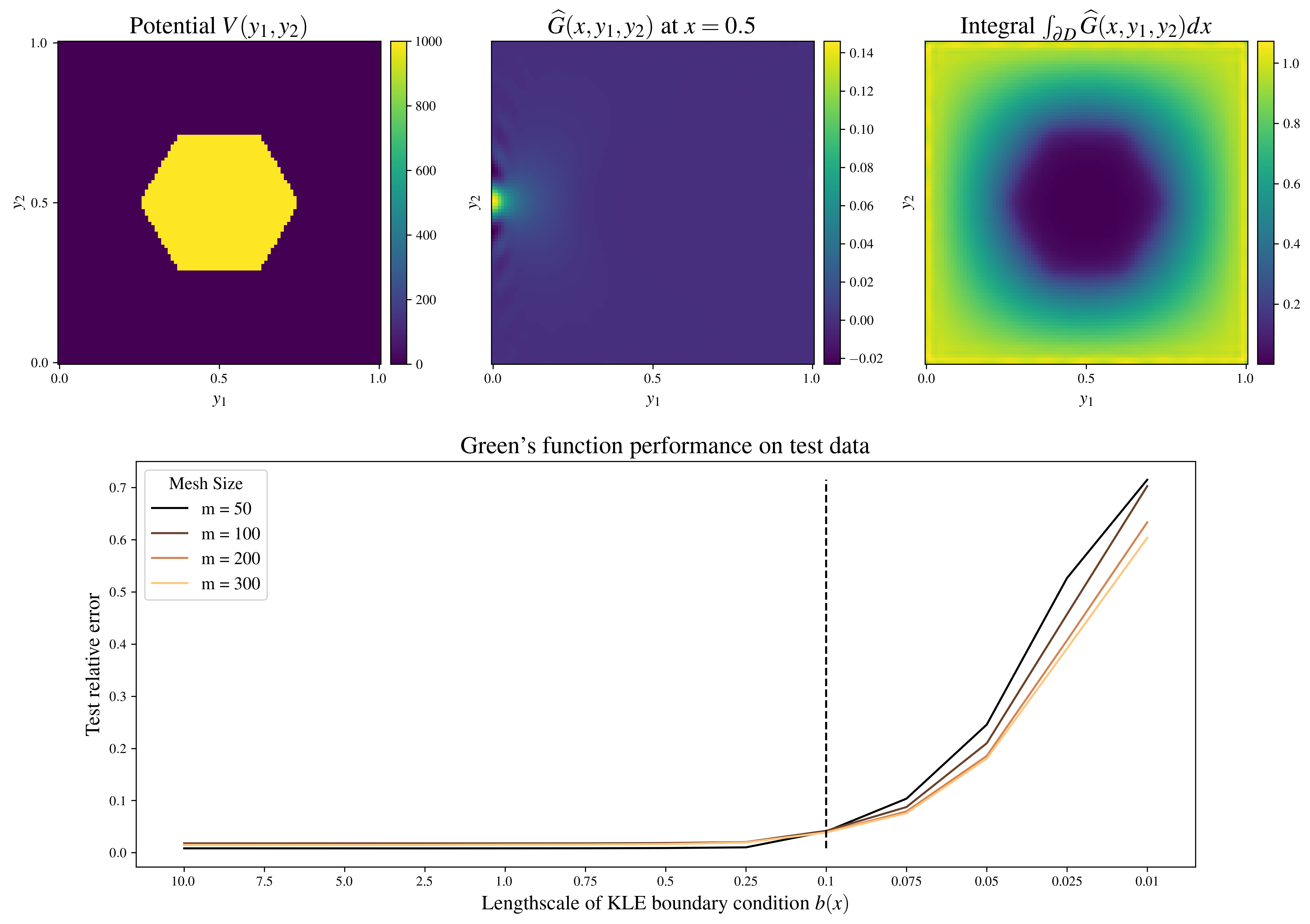}
\caption{Learning the Schr\"{o}dinger equation Green's function $\widehat{G}: b \mapsto u$ from Dirichlet boundary conditions on the boundary of the unit square to the solution in the interior. The Schr\"{o}dinger equation has a hexagonal step-function potential shown in the top left plot. Here we estimate $\widehat{G}$ in a Mat\'{e}rn 5/2 RKHS and several views of the learned Green's function are shown in the top center and right plots. In the bottom row we plot the test error of our estimator and show that with no additional training, it can simulate the PDE on finer grids and generalizes to input boundary conditions with varying lengthscales. Black dashed line indicates the lengthscale of the boundary conditions $b$ which $\widehat{G}$ was trained on.}\label{fig:schrodinger}
\end{figure}

\subsection{Fokker--Planck Equation}\label{subsec:FPE}
Now we show how to estimate the Green's function or fundamental solution of the \textit{Fokker--Planck equation} (FPE) which is a generalization of the diffusion (heat) equation and describes the evolution of a distribution of Brownian particles in a potential $V(x)$. Each particle's position $X_t$ is governed by a stochastic differential equation (SDE) of the form
\begin{equation}
    dX_t = -\frac{dV}{dx}(X_t)dt + \sqrt{2\alpha}dW_t
\end{equation}
where the \textit{diffusivity} $\alpha > 0$ is a constant and $W_t$ is a Wiener process. Assuming that the particles at time $t = 0$ are distributed by $X_0 \sim u_0$ then the Fokker--Planck equation for the probability density $u(x, t)$ is
\begin{equation}\label{eq:FPE}
    \frac{\partial}{\partial t}u(x, t) = \frac{\partial}{\partial x}\Big[\frac{dV}{dx}(x)u(x, t)\Big] + \alpha\frac{\partial^2}{\partial x^2}u(x, t), \quad u(x, 0) = u_0(x).
\end{equation}
Defining the \textit{probability flux}
\begin{equation}
    j(x, t) = \frac{dV}{dx}(x)u(x, t) + \alpha\frac{\partial}{\partial x}u(x, t)
\end{equation}
we can rewrite the FPE above as
\begin{equation}
    \frac{\partial}{\partial t}u(x, t) = \frac{\partial}{\partial x}j(x, t), \quad u(x, 0) = u_0(x).
\end{equation}
Now take the space-time domain $(x, t) \in D = [a, b] \times \mathbb{R}_+$ where the spatial domain is a finite interval. In order for the solution of~\eqref{eq:FPE} to be well-specified we impose \textit{reflecting} boundary conditions $j(a, t) = j(b, t) = 0$. This enforces that particles which reach the boundary are reflected back into the domain such that no mass leaves the domain (i.e. probability flux is zero).\\

In the following example, we simulate \eqref{eq:FPE} on the domain $D = [-2, 2] \times [0, 1]$ with the potential function
\begin{equation}\label{eq:FPEpotential}
    V(x) = x^4 - 3x^2.
\end{equation}
We learn a \textit{fundamental solution} $\widehat{G}(x, y, t)$ that maps the initial distribution $u_0(x)$ on the domain $D_\cX = [-2, 2]$ to the distribution at all future times $u(y, t)$ on the domain $D_\cY = D$. Here we learn our Green's function in a Mat\'{e}rn 3/2 RKHS with the kernel
\begin{equation}\label{eq:FPEkernel}
    K(x, y, t, \xi, \eta,\tau) = C_\nu\Big(\sqrt{\frac{(x - \xi)^2}{\sigma_x^2} + \frac{(y - \eta)^2}{\sigma_y^2} + \frac{(t - \tau)^2}{\sigma_t^2}}\Big)
\end{equation}
for $\nu = 3/2$ where $\sigma_x = \sigma_y = 0.16$ and $\sigma_t = 0.04$. Here $C_\nu$ is the Mat\'{e}rn covariance function defined in~\eqref{eq:matern} of Section~\ref{subsec:kernel_examples}.

Our Green's function estimator is trained on 500 samples where the inputs $(u_0)_i$ are generated from a Gaussian process KLE with a squared exponential kernel of lengthscale $\ell = 0.1$. For each initial condition $(u_0)_i$, we simulate the output solutions $u_i$ by a matrix numerical method~\citep{holubec2019physically} and corrupt our outputs with $20\%$ additive Gaussian noise. In the training data, the input initial conditions are discretized on $m_x = m = 50$ uniform grid points and the output solutions are discretized on $m_y \times m_t$ grid points where $m_y = m_t = m = 50$. The estimator is trained to convergence for 100 epochs with a $\lambda = 10^{-5}$ penalty on its RKHS norm.

Figure~\ref{fig:fokker_planck} shows cross-sections of our learned Green's function $\widehat{G}(x, y, t)$ at several timepoints $t$. From the learned Green's function we extract important features of the Fokker--Planck dynamics. At $t = 0$, the Green's function learns a map from $u_0(x) \mapsto u(y, 0)$ which is as close as possible to a delta function, limited only by the fixed lengthscale $\sigma_x, \sigma_y$ of our RKHS kernel. As time $t$ increases, the Green's function maps all the mass in $u_0(x)$ for $x > 0$ and $x < 0$ near the points $\pm 1.225$ respectively which correspond to the basins of the potential $V(x)$. As expected, our Green's function has learned that movement of mass for the FPE tends to the basins of the potential function.

The second row of Figure~\ref{fig:fokker_planck} shows the relative test error of our estimator when it is evaluated on 500 new test samples generated from a different distribution than the train data. For each test data set, we generate the initial conditions from a KLE with squared exponential kernel and vary the lengthscale of this process from $\ell = 0.01$ to 10.0. We also vary the mesh discretization of the input and output samples jointly from $m_x = m_y = m_t = m =  50$ to 150. We observe that the relative test error drops as we raise the lengthscale of the initial condition $u_0$ from $\ell = 0.1$ to $10.0$ but, as expected, increases if the lengthscale becomes smaller than $\ell = 0.1$ on which our estimator was trained. Additionally, we observe that our Green's function is insensitive to the mesh discretization and quickly plateaus as the mesh size $m$ is increased. Example predictions of our Green's function estimator on initial conditions with varied lengthscales are shown in Appendix~\ref{app:test_pred}.

\begin{figure}[h]
\centering
\includegraphics[width=\textwidth]{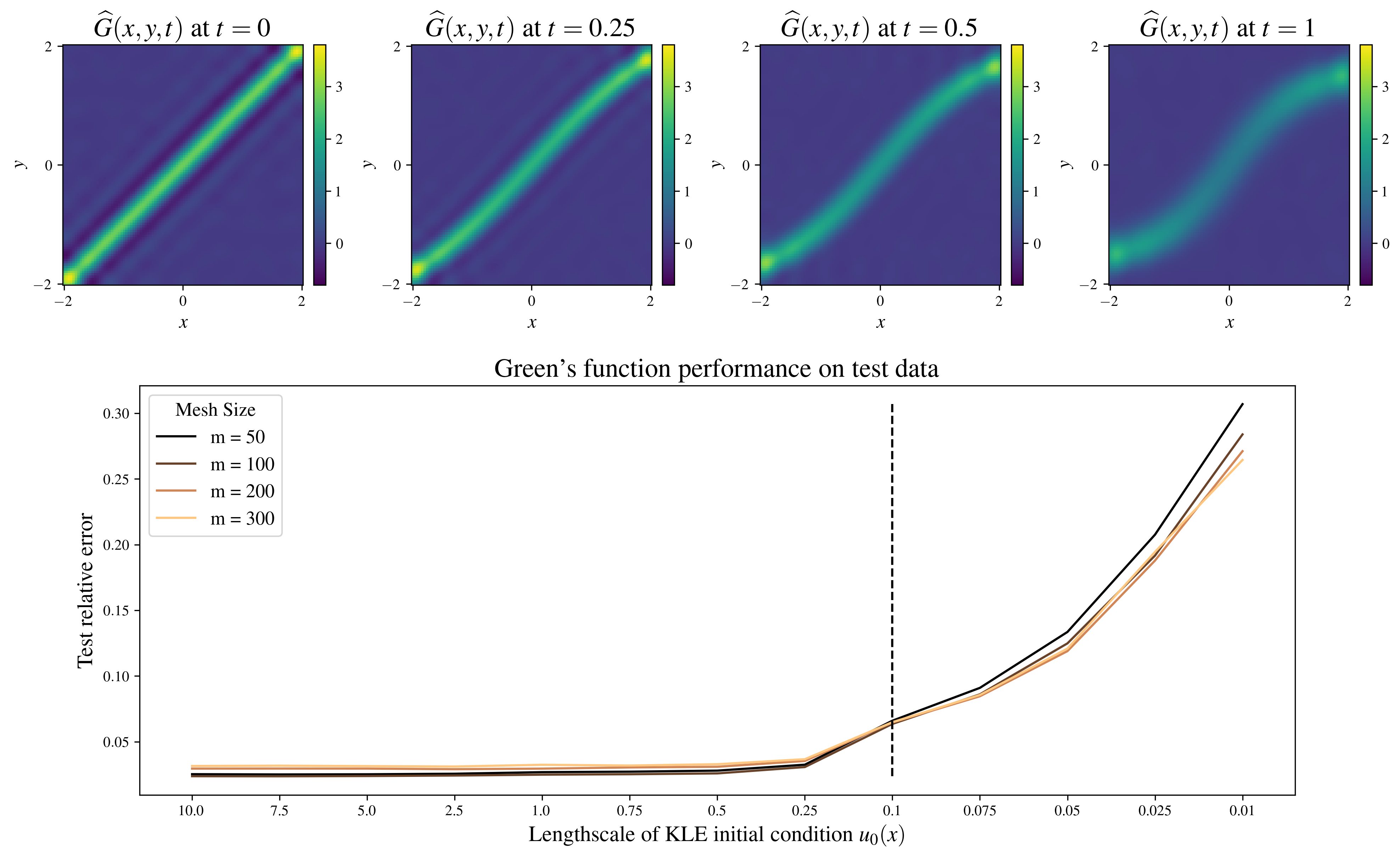}
\caption{Learning the Green's function of the Fokker--Planck equation in a Mat\'{e}rn 3/2 RKHS. Top row shows the Green's function estimator $\widehat{G}(x, y, t)$ at different time slices. Bottom row shows that our estimator generalizes to test data with inputs discretized on finer grids and with varying lengthscales. Black dashed line indicates the lengthscale of the initial conditions $u_0$ which $\widehat{G}$ was trained on.}\label{fig:fokker_planck}
\end{figure}

\subsection{Heat Equation with Time Invariance \& Causality Constraints}\label{subsec:heateq}
Many physical and biological systems are time-dependent where the inputs to the system $f(x, t)$ and output solutions $u(y, t)$ can be functions of time (here denoted by $t$). For example, the heat equation on the space-time domain $D = [0, 1] \times [0, \infty)$ with Dirichlet boundary conditions is
\begin{equation}\label{eq:heatpde}
    \begin{cases}
    \frac{\partial u}{\partial t} - \alpha\frac{\partial^2u}{\partial x^2} = f(x, t), \quad (x, t) \in [0, 1] \times [0, \infty)\\
    u(0, t) = u(1, t) = 0\\
    u(x, 0) = 0
    \end{cases}
\end{equation}
for some constant $\alpha > 0$. The function $u(x, t)$ is the distribution of heat in space at time $t$ and $f(x, t)$ is a heat source. Its fundamental solution has the form
\begin{equation}
\begin{gathered}
    G_\text{Heat}(x, s, y, t) = \mathbf{1}_{\{t \geq s\}}\sum_{k=1}^\infty 2\sin(\pi kx)\sin(\pi ky)e^{-\alpha k^2\pi^2(t-s)}\\
    u(y, t) = \int_0^\infty\int_0^\infty G_\text{Heat}(x, s, y, t)f(x, s)\ud x \ud s.
\end{gathered}
\end{equation}

Another example is the damped harmonic oscillator for $t \in \mathbb{R}$,
\begin{equation}
    \frac{\partial^2u}{\partial t^2} + c\frac{\partial u}{\partial t} + \omega^2u = f(t)
\end{equation}
where $u(t)$ is the position of the oscillator at time $t$, $f(t)$ is the forcing, $c$ is the damping coefficient, and $\omega$ is the frequency of the oscillator. Letting $r_1, r_2$ be the roots of the quadratic equation $x^2 + cx + \omega^2$ we can write the solution as
\begin{equation}
    u(t) = \int_{-\infty}^\infty G_\text{Harmonic}(s, t)f(s)ds, \quad G_\text{Harmonic}(s, t) = \mathbf{1}_{\{t \geq s\}}\frac{1}{r_1 - r_2}[e^{r_1(t - s)} - e^{r_2(t - s)}].
\end{equation}
Note that in both these examples, the Green's functions $G(\cdot, t, \cdot, s)$ satisfied the conditions
\begin{equation}
\begin{gathered}
    G(\cdot, s, \cdot, t) = G(\cdot, \cdot, t - s) \qquad \qquad \qquad G(\cdot, s, \cdot, t) = 0 \ \ \text{for} \ \ t < s.
\end{gathered}
\end{equation}
The first condition is called \textit{time invariance} and it enforces that the Green's function performs a convolution in the time variable. The Green's function is time invariant because it only depends on the difference $t - s$ between the time a perturbation was applied and the time of the response. This is a common feature of memoryless systems where a perturbation at time $s$ influences a response at time $t$ in precisely the same way that a perturbation at time $s + \Delta t$ influences a response at time $t + \Delta t$ for any $\Delta t > 0$. The second condition above forces the Green's function to be zero for all $t < s$ which ensures that the Green's function is \textit{causal} in time. This means that an impulse from $f$ applied at time $s$ cannot affect the solution $u$ at an earlier time $t < s$.

If such prior knowledge about the system is available, as is the case for many diffusive processes, it is advantageous to optimize over Green's functions that satisfy these constraints. Here we use these constraints to learn the Green's function of the heat equation~\eqref{eq:heatpde} in one spatial dimension with diffusivity constant $\alpha = 0.01$ and zero initial and boundary conditions. We aim to learn the map from the function of heat sources $f(x, s)$ to the distribution of heat $u(y, t)$ which both live on domains $D_\cX = D_\cY = [0, 1] \times [0, 1]$. As discussed above, we study two Green's functions estimators
\begin{subequations}
\begin{gather}
    \widehat{G}_1(x, s, y, t) = \widehat{g}(x, y, t-s)\label{eq:g_invariant}\\
	\widehat{G}_2(x, s, y, t) = \mathbf{1}_{\{t \geq s\}}\widehat{g}(x, y, t-s)\label{eq:g_causal}
\end{gather}
\end{subequations}
where the first estimator is time invariant while the second estimator is time invariant as well as time causal. As detailed in Appendices~\ref{app:flip_symm}-~\ref{app:time_invariance}, we can learn Green's function in these two forms by designing RKHSs with reproducing kernels
\begin{subequations}
\begin{gather}
    K(x, s, y, t, \xi, \sigma, \eta, \tau) = k(x, y, t-s, \xi, \eta, \tau-\sigma)\\
	K(x, s, y, t, \xi, \sigma, \eta, \tau) = \mathbf{1}_{t \geq s}\mathbf{1}_{\tau \geq \sigma} \cdot k_\text{symm}(x, y, t-s, \xi, \eta, \tau-\sigma)
\end{gather}
\end{subequations}
respectively where $k:([0, 1]^2 \times \mathbb{R})^2 \to \mathbb{R}$ is some Mercer kernel function and the kernel $k_\text{symm}: ([0, 1]^2 \times \mathbb{R})^2 \to \mathbb{R}$ is its flip-symmetrized version given by
\begin{equation}
\begin{aligned}
	k_\text{symm}(x, y, t, \xi, \eta, \tau) = \frac{1}{4}\Big(&k(x, y, t, \xi, \eta, \tau) + k(x, y, t, \xi, \eta, -\tau)\\
	&\qquad+ k(x, y, -t, \xi, \eta, \tau) + k(x, y, -t, \xi, \eta, -\tau)\Big).
\end{aligned}
\end{equation}
In the experiments for this section, we take the base kernel
\begin{equation}
k(x, y, t, \xi, \eta, \tau) = \exp\Big(\sqrt{\frac{(x - \xi)^2}{\sigma_x^2} + \frac{(y - \eta)^2}{\sigma_y^2} + \frac{(t - \tau)^2}{\sigma_t^2}}\Big)
\end{equation}
to be exponential with lengthscales $\sigma_x = \sigma_y = \sigma_t = 0.04$.

To train our Green's function estimators, we use a small train set of 100 samples with inputs $f(x, s)$ generated by an exponential KLE with lengthscale $\ell = 0.1$ on the $[0, 1]^2$ unit square. Outputs $u(y, t)$ are simulated using a backward Euler scheme and corrupted with $20\%$ additive i.i.d. Gaussian noise. Input functions $f$ are discretized on a grid of $m_x \times m_t$ points and output functions $u$ on a grid of $m_y \times m_t$ points where $m_x = m_y = m_t = m = 50$. During training, we regularize our estimators with $\lambda = 2\times 10^{-7}$ penalty which is chosen to produce smooth Green's functions $\widehat{G}_1, \widehat{G}_2$ while still small enough to approximate the exponential growth of the true Green's function around $s = t$.

In the top of Figure~\ref{fig:heat_eq} we display cross-sections of the true and estimated Green's functions, $G_\text{Heat}$ and $\widehat{G}_1, \widehat{G}_2$ respectively, for different time intervals $t = 0, 0.25$ and $0.5$. As expected, for $t = 0$ both of our Green's function estimators attempt to approximate the continuous delta function $\delta(x - y)$ and are solely limited by the small but fixed bandwidth of their RKHS kernel $K$. For $t > 0$ we observe a good agreement with the true Green's function and correctly identify the key features of the heat equation. Namely, both estimators $\widehat{G}_1, \widehat{G}_2$ smooth the input forcings $f(x, s)$ for positive times $t > 0$, fit the zero Dirichlet boundary conditions at $y = 0$ and $y = 1$ by tending to zero at the edges, and correctly learn that the Green's function is symmetric in its spatial (x, y) coordinates due to the the positive definite property of the Laplacian operator in~\eqref{eq:heatpde}.

In the bottom of Figure~\ref{fig:heat_eq} we show how our estimators perform on 500 new test samples when we vary the lengthscale $\ell$ of the boundary condition from $0.01$ to $10.0$ and the mesh discretization $m = m_x = m_y = m_t$ from 50 to 150. Our estimators are indeed independent of mesh discretization as their relative test errors plateau quickly as we increase $m$. As with previous examples, we see that the relative errors decrease as we raise the forcing input lengthscales from $\ell = 0.1$ to $10.0$ (e.g. very smooth boundary conditions) but gradually increase if the input lengthscales becomes smaller than $\ell = 0.1$ on which our estimators were trained. Once again, this behavior is expected since the true Green's function of the heat equation grows exponentially near the line $t = s$, and hence cannot be estimated perfectly at resolutions which exceed the mesh discretization and the lengthscale of the forcing inputs in the train data. Importantly, we see in the bottom plot of Figure~\ref{fig:heat_eq} that the Green's function $\widehat{G}_2$ from~\eqref{eq:g_causal} which is causal as well as time invariant is able to generalize better on inputs of finer lengthscales compared to the Green's function $\widehat{G}_1$ from~\eqref{eq:g_invariant} which is simply time invariant but not causal. This again highlights the benefits of learning Green's functions of PDEs in RKHSs that encode physical constraints. Example test predictions of our time-invariant and time-causal Green's function estimator $\widehat{G}_2$ are shown in Appendix~\ref{app:test_pred} for input forcings of varied lengthscales.

\begin{figure}[hbp]
\centering
\includegraphics[width=\textwidth]{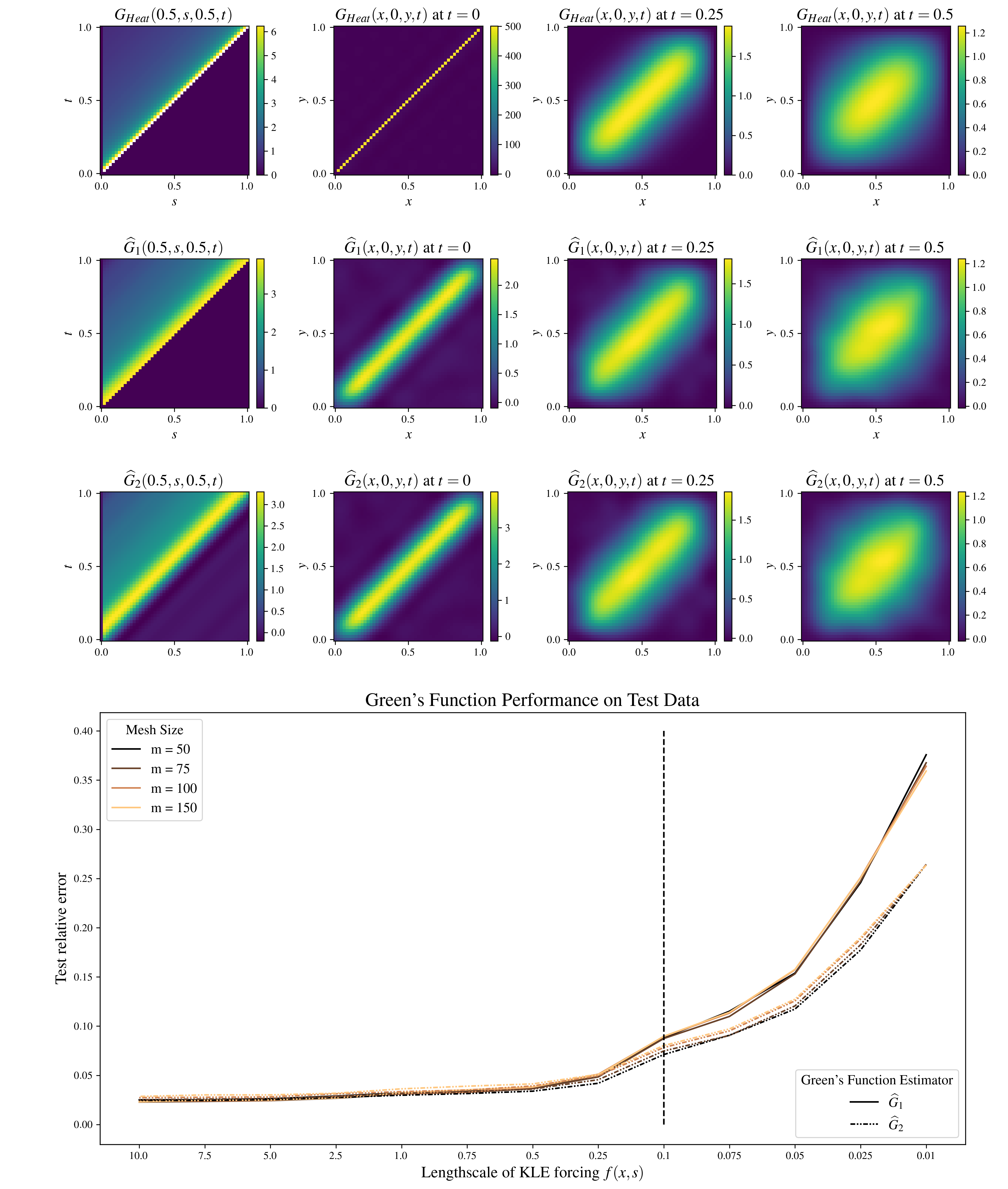}
\caption{Estimating the Green's function of the heat equation which maps a function of heat sources $f(x, s)$ to the distribution of heat $u(y, t)$. Top row shows multiple views of the true Green's function. Next two rows show the learned Green's function estimators $\widehat{G}_1$ from~\eqref{eq:g_invariant} and $\widehat{G}_2$ from~\eqref{eq:g_causal} which satisfy time invariant and causal constraints respectively. Bottom plot shows how both Green's function estimators (shown in solid and dash-dotted lines) generalize to test forcing inputs discretized on finer grids with varying levels of smoothness. The Green's function $\widehat{G}_2$ with causal constraints generalizes better to input forcings with smaller lengthscales. Black dashed line indicates the lengthscale of the training inputs $f$.}\label{fig:heat_eq}
\end{figure}

\newpage

\section{Conclusion}
In this paper, we developed a data-driven method for estimating the fundamental solution operator of a linear PDE which maps an input of the PDE (initial condition, boundary condition, or forcing) to its solution. Our method estimates the Green's function and bias term of the fundamental solution operator in an RKHS solely from input-output samples and with no detailed knowledge of the PDE. Through RKHS theory, we showed that our estimator is provably optimal on the training data and can be learned by minimizing a convex objective over a finite set of weights. We extended prior results on functional linear regression to bound the prediction error of our Green's function estimator trained on a finite number of samples. On the Poisson, Helmholtz, Schr\"{o}dinger, Fokker--Planck, and heat equation, we showed numerically that our RKHS estimators are significantly robust to noise in the training data and generalize to out-of-distribution test inputs with varying degrees of smoothness and mesh discretization.

An important direction of research is to extend our RKHS framework to learn nonlinear operators. This is necessary to model solution operators of nonlinear PDEs but also arises in linear PDEs when learning a map from a coefficient function to a solution. Recent advances in neural operators~\citep{lu2019deeponet, li2020fourier} and operator RKHS theory~\citep{kadri2016operator, nelsen2020random} have been instrumental in modeling such nonlinear maps. In light of these works, a natural extension of our framework would be the following \textit{Reproducing Kernel Network} (RKN)
\begin{equation}
    u_l(y) = \sigma_l\Big(\int_{D_{l-1}}G_l(x, y)u_{l-1}(x) \ud x + \beta_l(y)\Big) \ \ \text{for all} \ \ l = 1, \hdots, L
\end{equation}
where the input to the PDE is $f(x) = u_0(x)$ at layer $l = 0$ and the solution is modeled by $u(y) = u_L(y)$ at layer $l = L$. Here the $l$th layer takes a function from $L_2(D_{l-1})$ to $L_2(D_l)$ where it is composed of an integral operator $G_l$ in an RKHS $\cG_l \subset L^2(D_{l-1} \times D_l)$ and a bias term $\beta_l$ in an RKHS $\cB_l \subset L^2(D_l)$ as well as a pointwise activation function $\sigma_l: \mathbb{R} \to \mathbb{R}$. This general architecture is truly a continuous operator map between function spaces that can be implemented on any PDE domain geometry and irregular mesh. Furthermore, it promises various benefits including robustness to noise and mesh discretization as seen in our linear Green's function setting. Developing optimization methods and statistical guarantees for such nonlinear operators pose an interesting direction for future research.

\acks{We thank Philippe Rigollet for encouraging us to study the statistical aspects of surrogate modeling and for helpful advice and comments on our manuscript. We also thank Alasdair Hastewell, Adam Block, and J\"{orn} Dunkel for helpful discussions and feedback on our final draft as well as Nicholas Nelsen and Houman Owhadi for useful comments and for pointing out relevant related work. Lastly, we greatly appreciate the many insightful suggestions of the anonymous referees. George Stepaniants was supported by the National Science Foundation under Grant No. 1745302. We acknowledge the MIT SuperCloud and Lincoln Laboratory Supercomputing Center~\citep{reuther2018interactive} for providing HPC resources that have contributed to the numerical experiments reported within this paper. All code and experiments can be found at\\
\url{https://github.com/sgstepaniants/OperatorLearning}.}

\newpage

\appendix
\section{Experimental Details and Supplementary Results}\label{app:dataexp}
Here we give further details on our experimental setup for all experiments studied in Section~\ref{sec:examples}. For all numerical examples of the Poisson, Helmholtz, and heat equations we studied forced linear PDEs of the form
\begin{equation}
\begin{split}
    \cL u = f\\
    \cB u = 0
\end{split}
\end{equation}
where $\cL$ is a linear differential operator and $\cB$ is a linear operator which enforces the zero Dirichlet or Neumann boundary conditions of the PDE. On the examples of the Schr\"{o}dinger and Fokker--Planck equations we studied linear boundary value problems of the form
\begin{equation}
\begin{split}
    \cL u = 0\\
    \cB u = f
\end{split}
\end{equation}
where $\cB$ enforces either the initial conditions or the boundary conditions of the PDE. In both settings, we assume that $f$ and $u$ are square integrable functions in possibly different domains $D_\cX$ and $D_\cY$ respectively and are interested in learning the linear operator $\cT$ which maps $f$ to $u$.

\subsection{Data Generation}\label{app:data_gen}
For every experiment, the input functions $f(x)$ are discretized on some mesh $\{x_j\}_{j=1}^{m_x}$ of $m_x$ points. These inputs are generated using a Karhunen--Loeve expansion up to order $m_x$ given by
\begin{equation}\label{eq:KLE}
    f(x_j) = \sum_{k=1}^{m_x} Z_k\phi_k(x_j), \quad Z_k \sim \cN(0, \lambda_k)
\end{equation}
for i.i.d. normal random variables $Z_k$ where $\lambda_k, \phi_k$ are the eigenvalues and eigenvectors of a continuous, symmetric, positive definite kernel function $K(x, x')$. Note that the kernel $K$ here is used to define the distribution of our random inputs $f$, it is not to be confused with the reproducing kernel of the Green's function and bias term RKHSs.

By Mercer's theorem~\cite[Section~2, Theorem~1]{cucker2002mathematical} we know that $K$ has the spectral decomposition
\begin{equation}
    K(x, x') = \sum_{k=1}^\infty \lambda_k\phi_k(x)\phi_k(y).
\end{equation}
To compute the expansion in~\eqref{eq:KLE}, the eigenvalues and eigenfunctions of $K$ are computed numerically from the discrete matrix $\mathbf{K} \in \mathbb{R}^{m_x \times m_x}$ where $K_{ij} = K(x_i, x_j)$.

In Sections~\ref{subsec:poisson} and~\ref{subsec:helmholtz}, for the Poisson and Helmholtz equations we use the squared exponential kernel $K(x, x') = \exp(\|x - x'\|^2/2\ell^2)$ with a predefined lengthscale $\ell$ to generate all input forcings $f(x)$. Likewise, for the Fokker--Planck equation in Section~\ref{subsec:FPE} we use the squared exponential kernel to generate all initial conditions $u_0(x)$. On the example of Schr\"{o}dinger's equation in Section~\ref{subsec:schrodinger}, the random boundary conditions $b(x)$ are generated using the squared exponential \textit{periodic} kernel $K(x, x') = \exp(-2\sin(\pi\|x - x'\|/4)^2/\ell^2)$ to ensure that the boundary conditions are indeed periodic when wrapped around the boundary of the unit square. Lastly, for the heat equation in Section~\ref{subsec:heateq} we use the exponential kernel $K(x, x') = \exp(\|x - x'\|/\ell)$ to generate the input forcings $f(x)$.

\subsection{Noise Model}
Throughout all experiments, we corrupt our output samples $u(y)$ with a fraction $p$ of i.i.d. Gaussian noise. Specifically, given $n$ output samples $\mathbf{u}_1, \hdots, \mathbf{u}_n \in \mathbb{R}^{m_y}$ discretized on a mesh $\{y_k\}_{k=1}^{m_y}$ we compute the average standard deviation of all the samples
\begin{equation}
    \sigma_u = \sqrt{\frac{1}{nm_y}\sum_{i=1}^n\sum_{k=1}^{m_y}(u_{ik} - \overline{u}_k)^2}, \quad \overline{u}_k = \frac{1}{n}\sum_{i=1}^n u_{ik}
\end{equation}
and then corrupt our outputs with a fraction $p$ of i.i.d. Gaussian noise with standard deviation $\sigma_u$ by
\begin{equation}
    \tilde{u}_{ik} = u_{ik} + p\varepsilon_{ik}, \quad \varepsilon_{ik} \sim \cN(0, \sigma_u^2)
\end{equation}
for $i = 1, \hdots, n$ and $k = 1, \hdots, m_y$. Here $\varepsilon_{ik}$ are i.i.d. Gaussian random variables. This noise setup assumes that all locations in the domain $D_\cY$ of the functional outputs $u$ receive the same level of noise corruption which is a realistic model of sensor recordings where the sensor has a fixed level of noise in its measurements.

\subsection{Setting the RKHS Kernel Lengthscale Parameter}\label{app:kernel_lengthscale}
Recall from Section~\ref{sec:implementation} that given input functions $f_i(x)$ discretized on the mesh $\{x_j\}_{j=1}^{m_x}$ and output functions $u_i(y)$ discretized on the mesh $\{y_k\}_{k=1}^{m_y}$, our Green's function and bias term estimators take the form
\begin{equation}
    G_\mathbf{W}(x, y) = \sum_{j=1}^{m_x}\sum_{k=1}^{m_y} K(x, y, x_j, y_k)W_{jk}\Delta_j^x\Delta_k^y, \quad \beta_\mathbf{w}(y) = \sum_{k=1}^{m_y}Q(y, y_k)w_k \Delta_k^y
\end{equation}
where $\Delta_j^x, \Delta_k^y$ are fixed numerical quadrature weights of the input-output function grids respectively. During training, we optimize the weights $\mathbf{w}, \mathbf{W}$ of our estimators such that they minimize~\eqref{eq:opt_cost} which consists of a mean squared error term on the training data plus RKHS norm penalties on the estimators. In all experiments, the input-output grids $\{x_j\}_{j=1}^{m_x}, \{y_k\}_{k=1}^{m_y}$ are chosen to be uniform with equal spacing in all dimensions.

Here we study how the relative errors of the learned Green's function and bias term depend on the number of grid points $m$ and the lengthscale $\sigma$ of the kernels $K$ and $Q$ that define these estimators. We perform this experiment on the 1D Poisson equation from~\eqref{eq:poisson} and on the 1D Helmholtz equation with $\omega = 20$ from~\eqref{eq:helmholtz}. Our estimators are trained on 500 input-output samples with regularization $\lambda = 10^{-6}$ because we do not add noise to our data. In this experiment, $m_x = m_y = m$ and $\{x_j\}_{j=1}^{m_x}, \{y_k\}_{k=1}^{m_y} \subset [0, 1]$ since we are solving these PDEs on the unit interval. From Section~\ref{subsec:kernel_examples}, $K$ and $Q$ are both taken to be either exponential, squared exponential, Mat\'{e}rn 3/2 or Mat\'{e}rn 5/2 kernels and their lengthscales in the $x$ and $y$ directions are set equal to each other $\sigma_x = \sigma_y = \sigma$.

In Figure~\ref{fig:poisson_helmholtz_weight_sweep}, we show the relative error of the learned Green's function and bias term on the Poisson and Helmholtz equations across several different kernels, mesh sizes $m$, and kernel lengthscales $\sigma$. We see that the relation $\sigma = \frac{2}{m}$ (red dashed line) is a consistently good choice for the kernel lengthscale across all kernels types as it is quite small while still achieving a low relative error. In all our numerical examples in Section~\ref{sec:examples} we have found this to be a reliable choice of the kernel lengthscale. 

\begin{figure}[htbp]
\centering
\includegraphics[width=\textwidth]{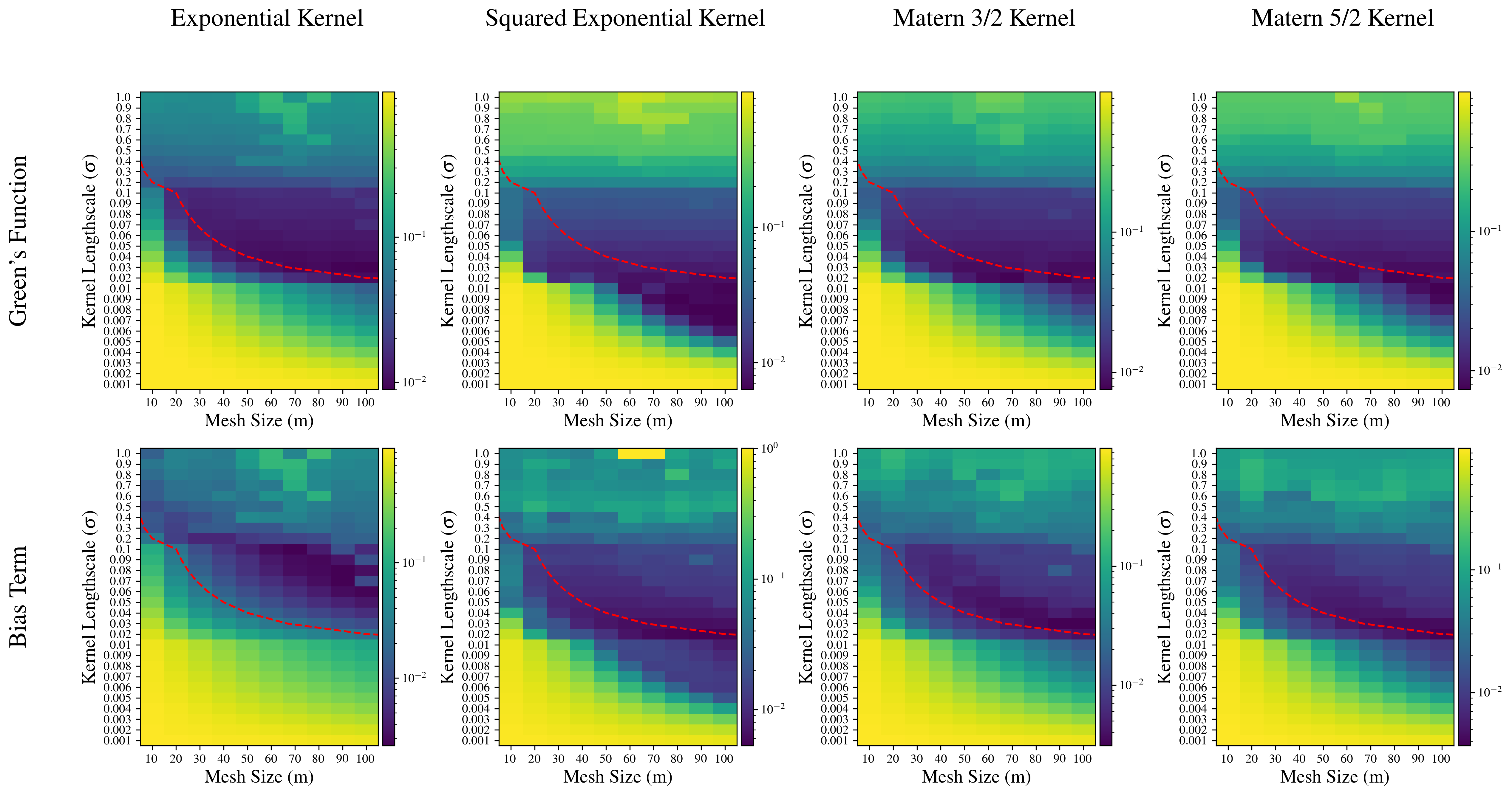}
\caption*{Sweep over kernel lengthscales for Poisson equation}
\vspace{\floatsep}
\includegraphics[width=\textwidth]{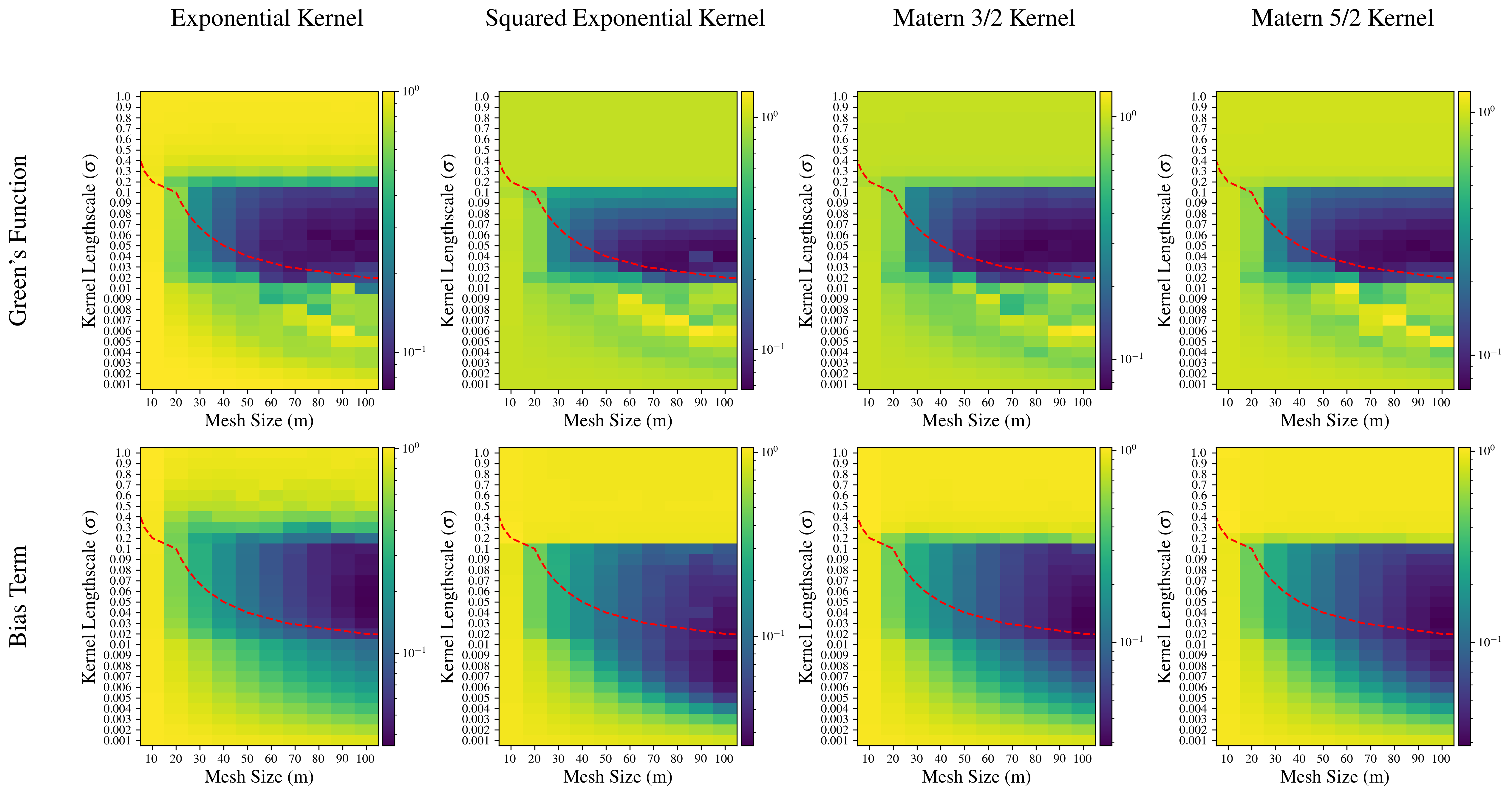}
\caption*{Sweep over kernel lengthscales for Helmholtz equation}
\caption{Each heatmap in this figure shows for a particular RKHS kernel, the relative error of the Green's function and bias term of the Poisson or Helmholtz equation learned in that RKHS. In each heatmap plot, we show how the relative error of the learned Green's function and bias term depend on the number of grid discretization points $m_x = m_y = m$ and the lengthscale $\sigma_x = \sigma_y = \sigma$ of the reproducing kernel used to construct the estimator on that grid. We find that the scaling relation $\sigma = \frac{2}{m}$ (red dashed line) is a consistently good choice for the lengthscale across all RKHS kernels as it is as small as possible while still achieving a low relative error.}\label{fig:poisson_helmholtz_weight_sweep}
\end{figure}

In general, for multidimensional problems we may be given input functions $f(x)$ with $x \in \mathbb{R}^{d_\cX}$ discretized on a uniform grid of size $m_x^1 \times \hdots m_x^{d_\cX}$ defined on a product domain $D_\cX = \Pi_{i=1}^{d_\cX} [a_i^x, b_i^x]$. Likewise, the output functions $u(y)$ with $y \in \mathbb{R}^{d_\cY}$ can be discretized on a uniform grid of size $m_y^1 \times \hdots m_y^{d_\cY}$ defined on a product domain $D_\cY = \Pi_{i=1}^{d_\cY} [a_i^y, b_i^y]$. For multidimensional problems, the kernels $K$ and $Q$ have vector-valued lengthscales which specify the width in each dimension as $[\boldsymbol{\sigma}_x, \boldsymbol{\sigma}_y] \in \mathbb{R}^{d_\cX + d_\cY}$ and $\boldsymbol{\sigma}_y \in \mathbb{R}^{d_\cY}$ respectively. Generalizing our rule above to multiple dimensions, we set the kernel lengthscales to $(\boldsymbol{\sigma}_x)_i \approx 2\frac{b_i^x-a_i^x}{m_x^i}$ for $i = 1, \hdots, d_\cX$ as well as $(\boldsymbol{\sigma}_y)_i \approx 2\frac{b_i^y-a_i^y}{m_y^i}$ for $i = 1, \hdots, d_\cY$.

\subsection{Test Predictions of Learned Green's Functions}\label{app:test_pred}
In this section, we show how Green's functions learned on the Schr\"{o}dinger, Fokker--Planck, and heat equations predict the solutions of these PDEs given finely discretized test inputs of varying lengthscales. For each lengthscale $\ell$, we generate 500 test inputs $f(x)$ at that lengthscale and simulate the corresponding solutions $u(y)$. We then show the best (lowest relative error) and worst (highest relative error) predictions of PDE solutions that are observed across those test samples at that lengthscale.

\begin{figure}[htbp]
\centering
\includegraphics[width=\textwidth]{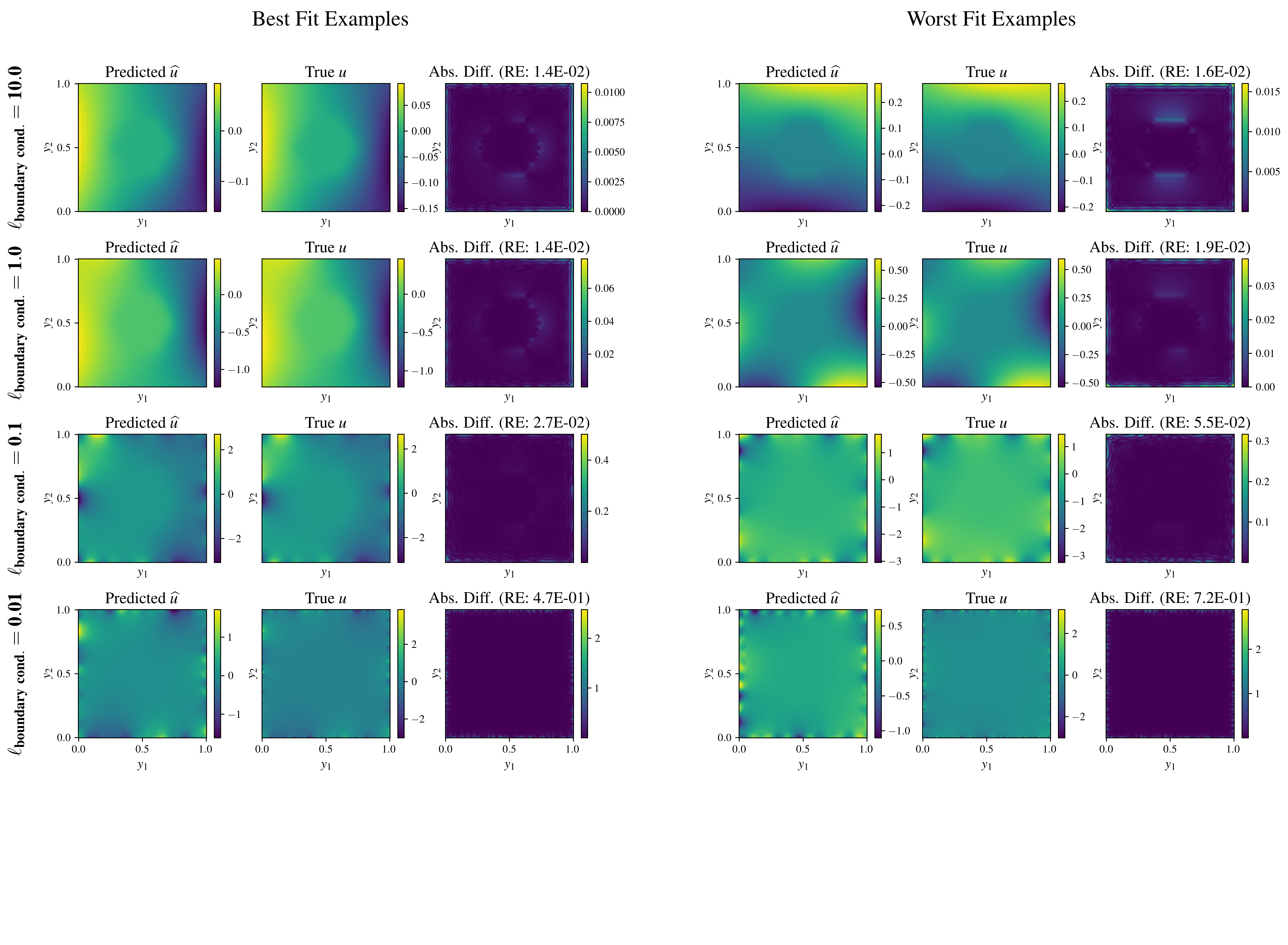}
\vspace{-2.5cm}
\caption{Example test predictions from Section~\ref{subsec:schrodinger} for the Schr\"{o}dinger equation $\Delta u - Vu = 0$ on $D = [0, 1]^2$ with $u = b$ on $\partial D$ for the linear map $G: b\mapsto u$ from boundary condition to solution. All input boundary conditions $b$ and solutions $u$ in the plots above are discretized on a grid of size $300 \times 300$.}\label{fig:schrodinger_examples}
\end{figure}

\begin{figure}[htbp]
\centering
\includegraphics[width=\textwidth]{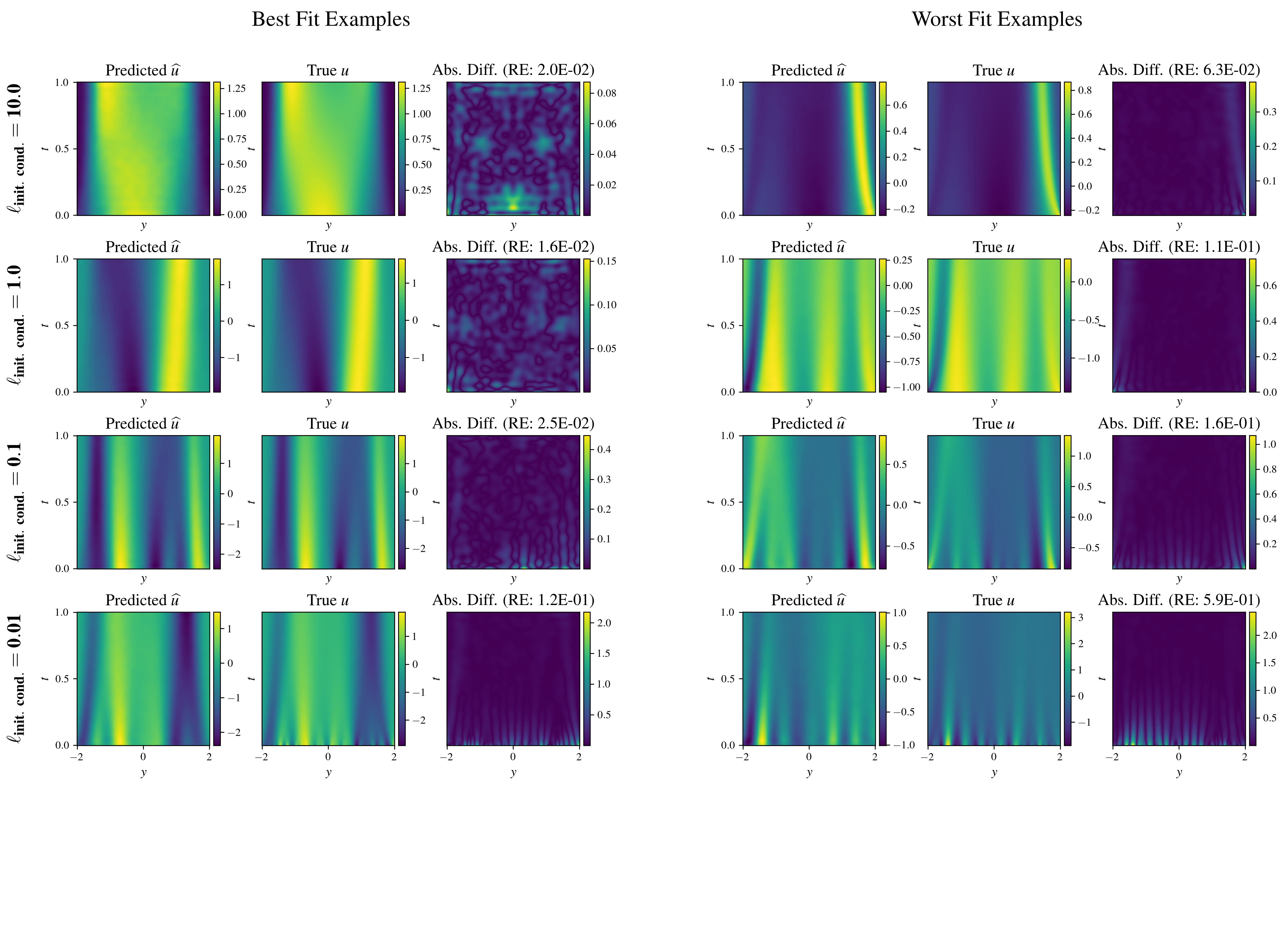}
\vspace{-2.5cm}
\caption{Example test predictions from Section~\ref{subsec:FPE} of $\widehat{G}: u_0 \mapsto u$ for the Fokker--Planck equation $\partial_t u = \nabla \cdot (u\nabla V) + \alpha\Delta u$ in one space and time dimension on the domain $(x, t) \in D = [-2, 2] \times [0, 1]$. The input initial conditions $u_0(x)$ are discretized on 150 points in [-2, 2] while the solutions plotted above on the domain $D$ are discretized on a space-time grid of $300 \times 300$ points.}\label{fig:fplanck_examples}
\end{figure}

\begin{figure}[htbp]
\centering
\includegraphics[width=\textwidth]{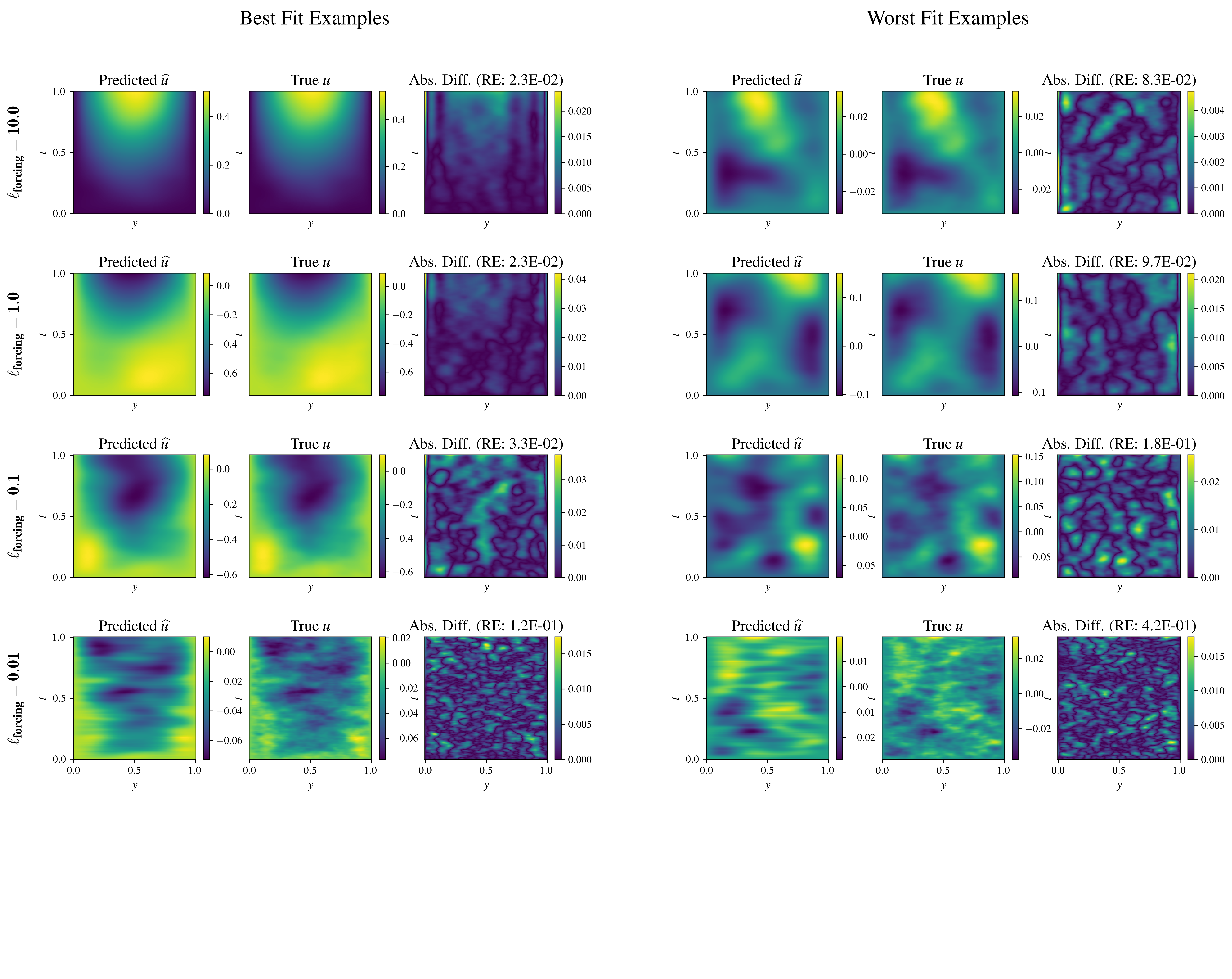}
\vspace{-2.5cm}
\caption{Example test predictions from Section~\ref{subsec:heateq} of $\widehat{G}_2: f \mapsto u$ for the heat equation $\partial_t u = \alpha\Delta u$ with $\alpha = 0.01$ in one space and time dimension on the domain $(x, t) \in D = [0, 1] \times [0, 1]$. The input forcings $f(x, s)$ and output solutions $u(y, t)$ on the domain $D$ plotted above are discretized on a space-time grid of $150 \times 150$ points.}\label{fig:heateq_examples}
\end{figure}

\newpage

\section{Representer Theorems}\label{app:repthm}
At the start of Section~\ref{sec:repthm}, we began by proving a simple representer theorem for our Green's function estimator when the functional data given to us was discretized on a set of grid points. Deriving the closed-form of the Green's function estimator relied on a key result that we restate and prove here.
\begin{theorem}
    Suppose we are given an RKHS $\cH \subset L^2(D)$ on a domain $D$ with continuous, symmetric, and strictly positive definite kernel $K: D^2 \to \mathbb{R}$. We denote the RKHS Hilbert space norm of any function $f \in \cH$ by $\|f\|_\cH$. Take any finite set of $m$ points $\{x_k\}_{k=1}^m \subset D$ in the domain and $n$ weight vectors $\{\mathbf{a}_i\}_{i=1}^n \subset \mathbb{R}^m$. We denote $\mathbf{f} = (f(x_1), \hdots, f(x_m))^T \in \mathbb{R}^m$ as the function $f$ evaluated on the set of grid points. Finally, let $\psi: \mathbb{R}_+ \to \mathbb{R}$ be a strictly increasing real-valued function and let $\cL: \mathbb{R}^n \to \mathbb{R}$ be an arbitrary error function. Then any minimizer
    \begin{equation}
        \widehat{f} \in \argmin_{f \in \cH} L(\{\inprod{\mathbf{f}, \mathbf{a}_i}_2\}_{i=1}^n) + \psi(\|f\|_\cH)
    \end{equation}
    must be of the form
    \begin{equation}
        \widehat{f}(x) = \sum_{k=1}^m K(x, x_k)w_k, \quad \mathbf{w} \in \spn(\{\mathbf{a}_i\}_{i=1}^n) \subset \mathbb{R}^m.
    \end{equation}
\end{theorem}

We assume here for simplicity that $K$ is a strictly positive definite kernel although the proof of this theorem can be easily extended to degenerate RKHSs. 
\begin{proof}
    We begin by defining two vector spaces
    \begin{equation}
        V = \spn(\{\mathbf{a}_i\}_{i=1}^n), \quad V^\perp = \{\mathbf{v} \in \mathbb{R}^m: \mathbf{v}^T\mathbf{K}\mathbf{w} = 0, \ \forall \mathbf{w} \in V\}
    \end{equation}
    where $\mathbf{K} = \{K(x_i, x_j)\}_{i, j = 1}^m \in \mathbb{R}^{m \times m}$. Because $\mathbf{K}$ is strictly positive definite, we know that $V \oplus V^\perp = \mathbb{R}^m$.
    
    Since our objective function depends only on $f \in \cH$ evaluated at the points $x_1, \hdots, x_m$ then we know by the standard representer theorem that $\widehat{f}$ must take the form
    \begin{equation}
        \widehat{f}(x) = \sum_{k=1}^m K(x, x_k)w_k
    \end{equation}
    where $\mathbf{w} = (w_1, \hdots, w_m)^T \in \mathbb{R}^m$. Now we can decompose our weight vector as $\mathbf{w} = \mathbf{w}^\parallel + \mathbf{w}^\perp$ for $\mathbf{w}^\parallel \in V$ and $\mathbf{w}^\perp \in V^\perp$ to write
    \begin{equation}
        \widehat{f}(x) = \sum_{k=1}^m K(x, x_k)w_k^\parallel + \sum_{k=1}^m K(x, x_k)w_k^\perp.
    \end{equation}
    Using the definition of the RKHS norm we can write out
    \begin{equation}
    \begin{aligned}
        \|\widehat{f}\|_\cH^2 &= (\mathbf{w}^\parallel)^T\mathbf{K}\mathbf{w}^\parallel + (\mathbf{w}^\perp)^T\mathbf{K}\mathbf{w}^\perp + 2(\mathbf{w}^\perp)^T\mathbf{K}\mathbf{w}^\parallel\\
        &= (\mathbf{w}^\parallel)^T\mathbf{K}\mathbf{w}^\parallel + (\mathbf{w}^\perp)^T\mathbf{K}\mathbf{w}^\perp
    \end{aligned}
    \end{equation}
    where the final term $(\mathbf{w}^\perp)^T\mathbf{K}\mathbf{w}^\parallel$ is zero since $\mathbf{w}^\perp \in V^\perp$.
    
    Additionally, writing as shorthand $\widehat{\mathbf{f}} = (\widehat{f}(x_1), \hdots, \widehat{f}(x_m))^T \in \mathbb{R}^m$ we see for each $i = 1, \hdots, n$ that
    \begin{equation}
        \inprod{\widehat{\mathbf{f}}, \mathbf{a}_i}_2 = \mathbf{a}_i^T\widehat{\mathbf{f}} = \mathbf{a}_i^T\mathbf{K}\mathbf{w}^\parallel + \mathbf{a}_i^T\mathbf{K}\mathbf{w}^\perp = \mathbf{a}_i^T\mathbf{K}\mathbf{w}^\parallel
    \end{equation}
    where again the last equality follows since $\mathbf{a}_i \in V$ and $\mathbf{w}^\perp \in V^\perp$. Finally, we can write our objective function as
    \begin{equation}
        L(\{\inprod{\mathbf{f}, \mathbf{a}_i}_2\}_{i=1}^n) + \psi(\|f\|_\cH) = L(\{\mathbf{a}_i^T\mathbf{K}\mathbf{w}^\parallel\}_{i=1}^n) + \psi(\sqrt{(\mathbf{w}^\parallel)^T\mathbf{K}\mathbf{w}^\parallel + (\mathbf{w}^\perp)^T\mathbf{K}\mathbf{w}^\perp}).
    \end{equation}
    Since $\psi$ is strictly increasing this proves that the optimal $\mathbf{w} \in \mathbb{R}^m$ must have $\mathbf{w}^\perp = 0$ implying that
    \begin{equation}
        \mathbf{w} = \mathbf{w}^\parallel \in V = \spn(\{\mathbf{a}_i\}_{i=1}^n).
    \end{equation}
\end{proof}
As an important note, from the theorem above we can recover the traditional representer theorem by setting $n = m$ and $\mathbf{a}_i = \mathbf{e}_i$ for all $i = 1, \hdots, n$.~\\

Now we lay the groundwork and derive the Green's function representer theorem stated in Theorem~\ref{thm:representer} for functional (non-discrete) input-output samples.
\subsection*{Proof of Operator Representer Theorem}
We state a more general representer theorem than that which is given in Theorem~\ref{thm:representer}. It holds for any operator (function-valued) RKHS with arbitrary loss function and regularization term. It is a generalization of the representer theorem proven in~\citet[Appendix~B, Theorem~9]{kadri2016operator}. The original theory of operator RKHSs and the corresponding representer theorem was first presented in the seminal work of~\citet[Section~4, Theorem~4.1]{micchelli2005learning}. We follow an alternative derivation of their result using a similar analysis to that of~\citet[Section~1.3, Theorem~1.3.1]{wahba1990spline}.

To establish notation, we let $\cX$ be a separable Hilbert space of functions from $D_\cX \to \mathbb{R}$ and we let $\cY$ be a separable Hilbert space of functions from $D_\cY \to \mathbb{R}$. We denote $\cL(\cY)$ as the space of bounded linear operators from $\cY$ to $\cY$. First we define an RKHS over operators as posed by~\citet[Section~4, Definitions~3, 5]{kadri2016operator}.

\begin{definition}[Operator-valued kernel]\label{def:operatorkernel}
An operator-valued kernel is a function $K: \cX \times \cX \to \cL(\cY)$ satisfying
\begin{enumerate}[(i)]
    \item $K$ is Hermitian if $\forall f, g \in \cX$ we have that $K(f, g) = K(g, f)^*$ where $*$ denotes the adjoint operator.
    \item $K$ is nonnegative (positive semidefinite) on $\cX$ if it is Hermitian and for all $r \in \mathbb{N}$ and any $\{(f_i, u_i)\}_{i=1}^r \subset \cX \times \cY$ we have that the matrix $M \in \mathbb{R}^{r \times r}$ with $M_{ij} = \inprod{K(f_i, f_j)u_i, u_j}_\cY$ is positive semidefinite.
\end{enumerate}
\end{definition}

\begin{definition}[Operator RKHS]
Let $\cO$ be a Hilbert space of operators $O: \cX \to \cY$ with inner product $\inprod{\cdot, \cdot}_\cO$. We call $\cO$ an operator RKHS if there exists an operator-valued kernel $K: \cX \times \cX \to \cL(\cY)$ such that
\begin{enumerate}[(i)]
    \item The function $g \mapsto K(f, g)u$ for $g \in \cX$ belongs to $\cO$ for all $f \in \cX, u \in \cY$.
    \item $K$ satisfies the \underline{reproducing property}
    \begin{equation}
        \inprod{O, K(f, \cdot)u}_\cO = \inprod{O(f), u}_\cY
    \end{equation}
    for all $O \in \cO$ and $f \in \cX, u \in \cY$.
\end{enumerate}
\end{definition}

We assume now that the operator RKHS $\cO$ can be decomposed orthogonally into $\cO = \cO_0 \oplus \cO_1$ where $\cO_0$ is a finite-dimensional Hilbert space spanned by the operators $\{E_k\}_{k=1}^r$ and $\cO_1$ is its orthogonal complement under the inner product $\inprod{\cdot, \cdot}_\cO$. We denote the inner product $\inprod{\cdot, \cdot}_\cO$ restricted to $\cO_0, \cO_1$ as $\inprod{\cdot, \cdot}_{\cO_0}, \inprod{\cdot, \cdot}_{\cO_1}$ respectively.

\begin{theorem}\label{thm:genrepresenter}
    Let $\psi: \mathbb{R}_+ \to \mathbb{R}$ be a strictly increasing real-valued function and let $\cL: (\cX \times \cY \times \cY) \to \mathbb{R}$ be an arbitrary error function. Then any minimizer
    \begin{equation}\label{eq:genOmin}
        \widehat{O}_{n\psi} \in \argmin_{O \in \cO}\Big[\cL\Big(\{(f_i, u_i, O(f_i))\}_{i=1}^n\Big) + \psi(\|\proj_{\cO_1}O\|_{\cO_1})\Big]
    \end{equation}
    must be of the form
    \begin{equation}\label{eq:genrepresenter}
        \widehat{O}_{n\psi}(x, y) = \sum_{k=1}^r d_kE_k + \sum_{i=1}^n K(f_i, \cdot)c_i
    \end{equation}
    for some $\mathbf{d} \in \mathbb{R}^r$ and $c_i \in \cY, i \in [n]$.
\end{theorem}

\begin{proof}[Theorem~\ref{thm:genrepresenter}]\\
Define the linear functional $L_i^v: \cO \to \mathbb{R}$ for any $v \in \cY, i \in [n]$ where $L_i^vO = \inprod{O(f_i), v}_\cY$. Note that for each $v \in \cY$ we have that $L_i^v$ is continuous (bounded) because
\begin{equation}
\begin{aligned}
|L_i^vO| &= |\inprod{O(f_i), v}_\cY| = |\inprod{O, K(f_i, \cdot)v}_\cO| \leq \|O\|_\cO\|K(f_i, \cdot)v\|_\cO\\
&= \|O\|_\cO\sqrt{\inprod{K(f_i, \cdot)v, K(f_i, \cdot)v}_\cO} = \|O\|_\cO\sqrt{\inprod{K(f_i, f_i)v, v}_\cY}\\
&\leq \|v\|_\cY\|K(f_i, f_i)\|_\text{op}\|O\|_\cO
\end{aligned}
\end{equation}
Note that the operator norm $\|K(f_i, f_i)\|_\text{op} < \infty$ because by definition $K(f_i, f_i) \in \cL(\cY)$ is bounded and $\|v\|_\cY < \infty$ since $v \in \cY$.

Therefore, by the Riesz representation theorem we know that for all $i \in [n], v \in \cY$ there exists a representer $N_i^v \in \cG$ for $L_i^v$ such that
\begin{equation}
\langle N_i^v, O\rangle_\cO = L_i^vO.
\end{equation}
for every $O \in \cO$.

Define $\Xi_i^v = \proj_{\cO_1}N_i^v \in \cO_1$ for all $i \in [n]$ and $v \in \cY$. Then for every $f \in \cX, u \in \cY$ we can write
\begin{equation}
\begin{aligned}
    \inprod{\Xi_i^v(f), u}_\cY &= \inprod{\Xi_i^v, K(f, \cdot)u}_{\cO_1} = \inprod{\Xi_i^v, K(f, \cdot)u}_\cO = \inprod{\proj_{\cO_1}N_i^v, K(f, \cdot)u}_\cO\\
    &= \inprod{N_i^v, \proj_{\cO_1}K(f, \cdot)u}_\cO = \inprod{N_i^v, K(f, \cdot)u}_\cO = L_i^v\Big(K(f, \cdot)u\Big)\\
    &= \inprod{K(f, f_i)u, v}_\cY = \inprod{K(f_i, f)v, u}_\cY.
\end{aligned}
\end{equation}
We have used the fact above that the projection operator $\proj_{\cO_1}$ is self-adjoint and that $K(f, f_i)^* = K(f_i, f)$ by definition of an operator reproducing kernel. Since the equality above holds for all $f \in \cX, u \in \cY$ this proves that $\Xi_i^v = K(f_i, \cdot)v$ for all $v \in \cY$.

Now let $Q: \cX \to \cY \in \cO$ be any operator perpendicular in the norm of $\cO$ to the subspace of $\cO$ spanned by $\{E_k\}_{k=1}^r$ and $\{\Xi_i^v: i \in [n], v \in \cY\}$. Note that $\spn\{\Xi_i^v: i \in [n], v \in \cY\} = \{K(f_i, \cdot)v: v \in \cY\}$. Hence, for any $\mathbf{d} \in \mathbb{R}^r$ and $c_i \in \cY, i \in [n]$ we can decompose any $O \in \cO$ as
\begin{equation}\label{eq:first_representer}
    O = \sum_{k=1}^r d_kE_k + \sum_{i=1}^n K(f_i, \cdot)c_i + Q.
\end{equation}

Since $Q$ is orthogonal to the span of $\{E_k\}_{k=1}^r$ this immediately implies that $Q \in \cO_1$. Hence, for any $i \in [n], v \in \cY$,
\begin{equation}
    \inprod{Q(f_i), v}_\cY = L_i^vQ = \inprod{N_i^v, Q}_\cO = \inprod{\Xi_i^v, Q}_{\cO_1} = 0.
\end{equation}
This proves that $Q(f_i) = 0$ so
\begin{equation}
    O(f_i) = \sum_{k=1}^r d_kE_k(f_i) + \sum_{i=1}^n K(f_i, f_i)c_i
\end{equation}
does not depend on $Q$. Lastly, we can see that $\proj_{\cO_1}O = \sum_{i=1}^n K(f_i, \cdot)c_i + Q$ where $\sum_{i=1}^n K(f_i, \cdot)c_i$ and $Q$ are orthogonal under the $\cO_1$ inner product.

Then, the objective of~\eqref{eq:genOmin} with the representation of $O$ in~\eqref{eq:first_representer} becomes
\begin{equation}
\begin{aligned}
    &\cL\Big(\{(f_i, u_i, O(f_i))\}_{i=1}^n\Big) + \psi\Big(\|\proj_{\cO_1}O\|_{\cO_1}\Big)\\
    &= \cL\Big(\{(f_i, u_i, O(f_i))\}_{i=1}^n\Big) + \psi\Big(\sqrt{\Big\|\sum_{i=1}^n K(f_i, \cdot)c_i\Big\|_{\cO_1}^2 + \|Q\|_{\cO_1}^2}\Big).
\end{aligned}
\end{equation}

Clearly, the optimal choice of $Q$ is zero since $\Psi$ is strictly increasing so this proves that the minimizer of the objective in~\eqref{eq:genOmin} must be of the form
\begin{equation}\label{eq:optO}
    \widehat{O}_{n\psi} = \sum_{k=1}^r d_kE_k + \sum_{i=1}^n K(f_i, \cdot)c_i.
\end{equation}
This completes the proof of the representer theorem for operator RKHSs $\cO = \cO_0 \oplus \cO_1$.
\end{proof}~\\

\begin{remark}
    Note that the space of Green's function integral operators
    \begin{equation}
        \cO = \Big\{O_G(f) := \int_{D_{\cX}} G(x,\cdot )f(x) \ud x \Big| G \in \cG\Big\}
    \end{equation}
    for some RKHS $\cG \subset L^2(D_\cX \times D_\cY)$ actually defines an operator RKHS. This is because we can make $\cO$ a Hilbert space by defining on it the induced inner product
    \begin{equation}
        \inprod{O_F, O_G}_\cO = \inprod{F, G}_\cG.
    \end{equation}
    Furthermore, given that $\cG$ has the continuous reproducing kernel $K: (D_\cX \times D_\cY)^2 \to \mathbb{R}$, it is easy to check that the operator-valued reproducing kernel for $\cO$ is $\cK: L^2(D_\cX) \times L^2(D_\cX) \to \cL(L^2(D_\cY))$ given by
    \begin{equation}
        [\cK(f, g)u](y) = \int_{D_\cX}\int_{D_\cX}\int_{D_\cY}K(x, y, \xi, \eta)g(x)f(\xi)u(\eta)
    \end{equation}
    for all $f, g \in L^2(D_\cX)$ and $u \in L^2(D_\cY)$. Note that $\cK(f, g)$ is always a bounded linear operator since $K$ is continuous on the bounded set $(D_\cX \times D_\cY)^2$. Using the definition of $\cK$ we can write for any $O_F \in \cO$,
    \begin{equation}
    \begin{aligned}
        \inprod{O_F, \cK(f, \cdot)u} &= \biginprod{F, \int_{D_\cX}\int_{D_\cY}K(\cdot, \cdot, \xi, \eta)f(\xi)u(\eta)\ud\xi\ud\eta}_\cG\\
        &= \int_{D_\cX}\int_{D_\cY}\inprod{F, K_{(\xi, \eta)}}_\cG f(\xi)u(\eta) \ud\xi\ud\eta\\
        &= \int_{D_\cX}\int_{D_\cY}F(\xi, \eta) f(\xi)u(\eta) \ud\xi\ud\eta\\
        &= \inprod{O_F(f), u}_{L^2(D_\cY)}.
    \end{aligned}
    \end{equation}
    Hence, the operator representer theorem derived above is immediately applicable to the setting of Theorem~\ref{thm:representer}.
\end{remark}

In the following corollary, we derive an expression for the closed form of these weights $d_k, c_i$ in the case of ridge regression. Throughout, we denote $\cY^n$ as the Hilbert space of vector-valued functions where every $\mathbf{u} \in \cY^n$ is a vector of functions $(u_1, \hdots, u_n)^T$ with each $u_i \in \cY^n$. For all $h, g \in \cY^n$ we define the inner product $\inprod{\mathbf{h}, \mathbf{g}}_{\cY^n} = \sum_{i=1}^n \inprod{h_i, g_i}_\cY$. Similar to before, $\cL(\cY^n)$ denotes the space of bounded linear operators from $\cY^n$ to $\cY^n$.
\begin{cor}\label{cor:closed_form}
    Define the positive semidefinite self-adjoint linear operators $\cM, \cM_\lambda \in \cL(\cY^n)$ for all $\mathbf{h} \in \cY^n$ by
    \begin{equation}
    \begin{split}
        [\cM_\lambda(\mathbf{h})]_i &= [\cM(\mathbf{h})]_i + n\lambda h_i\\
        [\cM(\mathbf{h})]_i &= \sum_{j=1}^n K(f_j, f_i)h_j.
    \end{split}
    \end{equation}
    Any minimizer of
    \begin{equation}\label{eq:Omin}
        \widehat{O}_{n\lambda} \in \argmin_{O \in \cO}\Big[\frac{1}{n}\sum_{i=1}^n\Big\|u_i - O(f_i))\Big\|_\cY^2 + \lambda\|\proj_{\cO_1}O\|_{\cO_1}^2\Big]
    \end{equation}
    must be of the form
    \begin{equation}\label{eq:representer}
        \widehat{O}_{n\lambda}(x, y) = \sum_{k=1}^r d_kE_k + \sum_{i=1}^n K(f_i, \cdot)c_i
    \end{equation}
    Assuming that the matrix $\mathbf{A} \in \mathbb{R}^{r \times r}$ defined by
    \begin{equation}
        \mathbf{A} = \begin{bmatrix}
        \inprod{\mathbf{T}^1, \cM_\lambda^{-1}(\mathbf{T}^1)}_{\cY^n} & \hdots & \inprod{\mathbf{T}^1, \cM_\lambda^{-1}(\mathbf{T}^r)}_{\cY^n}\\
        \vdots & & \vdots\\
        \inprod{\mathbf{T}^r, \cM_\lambda^{-1}(\mathbf{T}^1)}_{\cY^n} & \hdots & \inprod{\mathbf{T}^r, \cM_\lambda^{-1}(\mathbf{T}^r)}_{\cY^n}
        \end{bmatrix}.
    \end{equation}
    is invertible, an optimal choice of the weights $\mathbf{d} \in \mathbb{R}^n, \mathbf{c} \in \cY^n$ is
    \begin{equation}
        \mathbf{c} = \cM_\lambda^{-1}(\mathbf{u}) - \sum_{k=1}^rd_k\cM_\lambda^{-1}(\mathbf{T}^k), \quad \mathbf{d} = \mathbf{A}^{-1}\begin{bmatrix}
        \inprod{\mathbf{T}^1, \cM_\lambda^{-1}(\mathbf{u})}_{\cY^n}\\
        \vdots\\
        \inprod{\mathbf{T}^r, \cM_\lambda^{-1}(\mathbf{u})}_{\cY^n}
        \end{bmatrix}
    \end{equation}
\end{cor}

\begin{proof}[Corollary~\ref{cor:closed_form}]\\
As shorthand, let us define the functions $T_i^k \in \cY$ for $i \in [n], k \in [r]$ where
\begin{equation}
    T_i^k = E_k(f_i)
\end{equation}
and the bounded linear operators $\Sigma_{ij} \in \cL(\cY)$ given by
\begin{equation}
    \Sigma_{ij} = K(f_j, f_i).
\end{equation}
Note that in fact,
\begin{equation}
    \inprod{K(f_i, \cdot)c_i, K(f_j, \cdot)c_j}_{\cO_1} = \inprod{c_i, K(f_j, f_i)c_j}_\cY = \inprod{c_i, \Sigma_{ij}c_j}_\cY
\end{equation}
by the reproducing property of $K$ and $\Sigma_{ij}^* = \Sigma_{ji}$. Also, all of the $\Sigma_{ij}$ are positive semidefinite linear operators by property (ii) of Definition~\ref{def:operatorkernel}. Hence, from the derivations in Theorem~\ref{thm:genrepresenter} we know that the optimal $\widehat{O}_{n\lambda}$ in~\eqref{eq:optO} satisfies
\begin{equation}
    \widehat{O}_{n\lambda}(f_i) = \sum_{k=1}^r d_kT_i^k + \sum_{j=1}^n \Sigma_{ij}c_j, \qquad \|\proj_{\cO_1}\widehat{O}_{n\lambda}\|_{\cO_1}^2 = \sum_{i=1}^n\sum_{j=1}^n \inprod{c_i, \Sigma_{ij}c_j}_\cY.
\end{equation}

Now fixing $\mathbf{d} \in \mathbb{R}^r$ we need to find the functions $c_i \in \cY$ that minimize
\begin{equation}
\begin{aligned}
    \sum_{i=1}^n\Big\|&u_i - \sum_{k=1}^r d_kT_i^k - \sum_{j=1}^n \Sigma_{ij}c_j\Big\|_\cY^2 + n\lambda\sum_{i=1}^n\sum_{j=1}^n \inprod{c_i, \Sigma_{ij}c_j}_\cY\\
    &= -2\sum_{i=1}^n\sum_{j=1}^n \biginprod{u_i - \sum_{k=1}^r d_kT_i^k, \Sigma_{ij}c_j}_\cY\\
    &\qquad+ \sum_{i=1}^n\sum_{j=1}^n\sum_{k=1}^n \biginprod{\Sigma_{ij}c_j, \Sigma_{ik}c_k}_\cY\\
    &\qquad+ n\lambda\sum_{i=1}^n\sum_{j=1}^n \inprod{c_i, \Sigma_{ij}c_j}_\cY + \text{const}.
\end{aligned}
\end{equation}
Taking the Hilbert space variational derivative of the expression above in $c_i \in \cY$ and setting it to zero we get that
\begin{equation}
    -\sum_{j=1}^n \Sigma_{ij}\Big(u_j - \sum_{k=1}^r d_kT_j^k\Big) + \sum_{j=1}^n\sum_{k=1}^n\Sigma_{ij}\Sigma_{jk}c_k + n\lambda\sum_{j=1}^n\Sigma_{ij}c_j = 0
\end{equation}
where we have repeatedly used the fact that $\Sigma_{ij}^* = \Sigma_{ji}$. We can rewrite this as
\begin{equation}\label{eq:cvarderiv}
\sum_{j=1}^n\sum_{k=1}^n\Sigma_{ij}\Big(\Sigma_{jk} + n\lambda\delta_{jk}\Big)c_k = \sum_{j=1}^n \Sigma_{ij}\Big(u_j - \sum_{k=1}^r d_kT_j^k\Big)
\end{equation}
where $\delta_{jk}$ is the Kronecker delta function.

Now define the bounded linear operators $\cM, \cM_\lambda \in \cL(\cY^n)$ for all $\mathbf{h} \in \cY^n$ by
\begin{equation}
\begin{split}
    [\cM_\lambda(\mathbf{h})]_i &= [\cM(\mathbf{h})]_i + n\lambda h_i\\
    [\cM(\mathbf{h})]_i &= \sum_{j=1}^n \Sigma_{ij}h_j.
\end{split}
\end{equation}
We can use the definition of $\cM_\lambda$ to write~\eqref{eq:cvarderiv} as
\begin{equation}
    \sum_{j=1}^n\Sigma_{ij}[\cM_\lambda(\mathbf{c})]_j = \sum_{j=1}^n \Sigma_{ij}\Big(u_j - \sum_{k=1}^r d_kT_j^k\Big).
\end{equation}
and using the definition of $\cM$ this can further be written as
\begin{equation}
    \cM\Big(\cM_\lambda(\mathbf{c})\Big) = \cM\Big(\mathbf{u} - \sum_{k=1}^r d_k\mathbf{T}^k\Big).
\end{equation}
It is clear from the expression above that if the linear operator $\cM_\lambda \in \cL(\cY)$ is invertible then
\begin{equation}\label{eq:copt}
    \mathbf{c} = \cM_\lambda^{-1}(\mathbf{u}) - \sum_{k=1}^rd_k\cM_\lambda^{-1}(\mathbf{T}^k).
\end{equation}
is a solution. Note that $\cM_\lambda$ is a sum of positive semidefinite operators $\Sigma_{ij}$ plus $\lambda$ times the identity operator. Hence $\cM_\lambda$ is indeed strictly positive definite and invertible for $\lambda > 0$. It is also self-adjoint under the inner product $\inprod{\mathbf{h}, \mathbf{g}}_{\cY^n} = \sum_{i=1}^n \inprod{h_i, g_i}_\cY$ since $\Sigma_{ij}^* = \Sigma_{ji}$. This also immediately implies that $\cM_\lambda^{-1}$ is strictly positive definite and self-adjoint.

Taking this value of $\mathbf{c}$ in~\eqref{eq:copt} we need to find the coefficients $\mathbf{d} \in \mathbb{R}^r$ that minimize
\begin{equation}
\begin{aligned}
    \sum_{i=1}^n\Big\|&u_i - \sum_{k=1}^r d_kT_i^k - \sum_{j=1}^n \Sigma_{ij}c_j\Big\|_\cY^2 + n\lambda\sum_{i=1}^n\sum_{j=1}^n \inprod{c_i, \Sigma_{ij}c_j}_\cY\\
    &= \sum_{i=1}^n\Big\|\Big(u_i - \sum_{j=1}^n \Sigma_{ij}[\cM_\lambda^{-1}(\mathbf{u})]_j\Big) - \sum_{k=1}^r d_k\Big(T_i^k - \sum_{j=1}^n \Sigma_{ij}[\cM_\lambda^{-1}(\mathbf{T}^k)]_j\Big)\Big\|_\cY^2\\
    &\qquad- 2n\lambda\sum_{i=1}^n\sum_{j=1}^n\sum_{k=1}^r d_k\biginprod{[\cM_\lambda^{-1}(\mathbf{u})]_i, \Sigma_{ij}[\cM_\lambda^{-1}(\mathbf{T}^k)]_j}_\cY\\
    &\qquad + n\lambda\sum_{i=1}^n\sum_{j=1}^n\sum_{k=1}^r\sum_{l=1}^rd_kd_l\biginprod{[\cM_\lambda^{-1}(\mathbf{T}^k)]_i, \Sigma_{ij}[\cM_\lambda^{-1}(\mathbf{T}^l)]_j}_\cY + \text{const}.
\end{aligned}
\end{equation}
Noting that $u_i = \sum_{j=1}^n \Sigma_{ij}[\cM_\lambda^{-1}(\mathbf{u})]_j + n\lambda[\cM_\lambda^{-1}(\mathbf{u})]_i$ and $T_i^k = \sum_{j=1}^n \Sigma_{ij}[\cM_\lambda^{-1}(\mathbf{T}^k)]_j + n\lambda[\cM_\lambda^{-1}(\mathbf{T}^k)]_i$ we can rewrite the above expression as
\begin{equation}
\begin{aligned}
    &n^2\lambda^2\sum_{i=1}^n\Big\|\cM_\lambda^{-1}(\mathbf{u})]_i - \sum_{k=1}^r d_k\cM_\lambda^{-1}(\mathbf{T}^k)]_i\Big\|_{L^2(D_\cY)}^2\\
    &\qquad- 2n\lambda\sum_{i=1}^n\sum_{j=1}^n\sum_{k=1}^r d_k\biginprod{[\cM_\lambda^{-1}(\mathbf{u})]_i, \Sigma_{ij}[\cM_\lambda^{-1}(\mathbf{T}^k)]_j}_\cY\\
    &\qquad + n\lambda\sum_{i=1}^n\sum_{j=1}^n\sum_{k=1}^r\sum_{l=1}^rd_kd_l\biginprod{[\cM_\lambda^{-1}(\mathbf{T}^k)]_i, \Sigma_{ij}[\cM_\lambda^{-1}(\mathbf{T}^l)]_j}_\cY + \text{const}.
\end{aligned}
\end{equation}
Dividing through by $n\lambda$ and setting the derivative in $d_k$ to zero we get
\begin{equation}
\begin{aligned}
    &-n\lambda\sum_{i=1}^n\biginprod{\cM_\lambda^{-1}(\mathbf{u})]_i, \cM_\lambda^{-1}(\mathbf{T}^k)]_i}_\cY + n\lambda\sum_{i=1}^n\sum_{l=1}^rd_l\biginprod{\cM_\lambda^{-1}(\mathbf{T}^k)]_i, \cM_\lambda^{-1}(\mathbf{T}^l)]_i}_\cY\\
    &\qquad- \sum_{i=1}^n\sum_{j=1}^n\biginprod{[\cM_\lambda^{-1}(\mathbf{u})]_i, \Sigma_{ij}[\cM_\lambda^{-1}(\mathbf{T}^k)]_j}_\cY\\
    &\qquad+ \sum_{i=1}^n\sum_{j=1}^n\sum_{l=1}^rd_l\biginprod{[\cM_\lambda^{-1}(\mathbf{T}^k)]_i, \Sigma_{ij}[\cM_\lambda^{-1}(\mathbf{T}^l)]_j}_\cY = 0
\end{aligned}
\end{equation}
which we can further rewrite as
\begin{equation}
\begin{aligned}
    &\sum_{i=1}^n\sum_{l=1}^rd_l\biginprod{[\cM_\lambda^{-1}(\mathbf{T}^k)]_i, \Big(\sum_{j=1}^n\Sigma_{ij}[\cM_\lambda^{-1}(\mathbf{T}^l)]_j + n\lambda[\cM_\lambda^{-1}(\mathbf{T}^l)]_i\Big)}_\cY\\
    &\qquad= \sum_{i=1}^n\biginprod{[\cM_\lambda^{-1}(\mathbf{u})]_i, \Big(\sum_{j=1}^n\Sigma_{ij}[\cM_\lambda^{-1}(\mathbf{T}^k)]_j + n\lambda[\cM_\lambda^{-1}(\mathbf{T}^k)]_i\Big)}_\cY.
\end{aligned}
\end{equation}
Finally, noting that the terms in the parentheses are $T_i^l, T_i^k$ respectively we can write
\begin{equation}
    \sum_{l=1}^rd_l\sum_{i=1}^n\inprod{T_i^l, [\cM_\lambda^{-1}(\mathbf{T}^k)]_i}_\cY = \sum_{i=1}^n \inprod{T_i^k, [\cM_\lambda^{-1}(\mathbf{u})]_i}_\cY.
\end{equation}
Using the shorthand notation $\inprod{\mathbf{h}, \mathbf{g}}_{\cY^n} = \sum_{i=1}^n \inprod{h_i, g_i}_\cY$ and the fact that $\cM_\lambda^{-1}$ is self-adjoint this becomes
\begin{equation}
    \sum_{l=1}^r\inprod{\mathbf{T}^k, \cM_\lambda^{-1}(\mathbf{T}^l)}_{\cY^n} d_l = \inprod{\mathbf{T}^k, \cM_\lambda^{-1}(\mathbf{u})}_{\cY^n}
\end{equation}
for $k = 1, \hdots, n$. Hence, assuming that
\begin{equation}
    \mathbf{A} = \begin{bmatrix}
    \inprod{\mathbf{T}^1, \cM_\lambda^{-1}(\mathbf{T}^1)}_{\cY^n} & \hdots & \inprod{\mathbf{T}^1, \cM_\lambda^{-1}(\mathbf{T}^r)}_{\cY^n}\\
    \vdots & & \vdots\\
    \inprod{\mathbf{T}^r, \cM_\lambda^{-1}(\mathbf{T}^1)}_{\cY^n} & \hdots & \inprod{\mathbf{T}^r, \cM_\lambda^{-1}(\mathbf{T}^r)}_{\cY^n}
    \end{bmatrix}.
\end{equation}
is invertible we find that
\begin{equation}
    \mathbf{d} = \mathbf{A}^{-1}\begin{bmatrix}
    \inprod{\mathbf{T}^1, \cM_\lambda^{-1}(\mathbf{u})}_{\cY^n}\\
    \vdots\\
    \inprod{\mathbf{T}^r, \cM_\lambda^{-1}(\mathbf{u})}_{\cY^n}
    \end{bmatrix}
\end{equation}
which is what we desired to prove.
\end{proof}

\newpage

\section{Rearrangements of Eigenvalues}\label{app:eig_rearrangement}

In Proposition~\ref{prop:simdiagsimple} we encountered the problem where it was necessary to bound the rate of decay of eigenvalues $\rho_{ij}$ of a multidimensional kernel. For a simple example, if we have a kernel $K = k_1 \otimes k_2$ where the eigenvalues of $k_1$ decay with rate $i^{-r_1}$ and the eigenvalues of $k_2$ decay with rate $j^{-r_2}$ then the eigenvalues of $K$ decay as $\rho_{ij} \asymp i^{-r_1}j^{-r_2}$. Sorting the eigenvalues $\rho_{ij}$ and determining their rate of decay exactly is a difficult problem. However, as in Example 2 of Section~\ref{subsec:mercer_kernels} it is possible to enumerate these eigenvalues using a bijection $\pi: \mathbb{N}^2 \to \mathbb{N}$ such as the Cantor pairing function. This allows us to easily obtain bounds on the rate of decay of the eigenvalues $\rho_k := \rho_{\pi^{-1}(k)}$ under this ordering. However, these error bounds cannot be immediately applied to bound the rate of decrease of the sorted eigenvalues $\rho_1 \geq \rho_2 \geq \hdots$

Here we show that any monotonic upper bound for a sequence is still an upper bound for the sequence sorted in decreasing order. First we prove the following lemma to show that any sequence of eigenvalues tending to zero can be reordered into sorted decreasing order by some bijection.
\begin{lemma}
    A sequence $\{a_k\}_{k=1}^\infty$ satisfies $\lim_{k \to \infty} a_k = \inf_{k \geq 1} a_k$ if and only if there exists a bijection $\pi: \mathbb{N} \to \mathbb{N}$ such that $a_{\pi(1)} \geq a_\pi(2) \geq \hdots$
\end{lemma}
\begin{proof}
    First let's assume that $a_k$ satisfies $\lim_{k \to \infty} a_k = \inf_{k \geq 1} a_k$. Then we can construct the following bijection $\pi$. Define $U = \sup_{k \geq 1} a_k$ and $L = \inf_{k \geq 1} a_k$. Take all the $a_k \in [U, \frac{U}{2} + \frac{L}{2})$ and sort them. Continue this process by appending all of the $a_k \in [\frac{U}{2^n} + \frac{(2^n - 1)L}{2^n}, \frac{U}{2^{n+1}} + \frac{(2^{n+1} - 1)L}{2^{n+1}})$ in sorted order for all $n \geq 1$. This defines the bijection $\pi$ from the original sequence of $a_k$'s to the sorted decreasing sequence $a_1 \geq a_2 \geq \hdots$
    
    Now in the opposite direction, let's assume there exists a bijection $\pi$ such that $a_{\pi(1)} \geq a_\pi(2) \geq \hdots$ Then by the monotone convergence theorem we have that $\lim_{k \to \infty} a_{\pi(k)} = \inf_{k \geq 1} a_k$. Since a sequence's limit is invariant under rearrangements $\pi$, this proves that $\lim_{k \to \infty} a_k = \inf_{k \geq 1} a_k$.
\end{proof}

\begin{lemma}\label{lemma:rearrangement}
    Assume that a sequence $\{a_k\}_{k=1}^\infty$ is nonincreasing $a_1 \geq a_2 \geq \hdots$ and bounded from below. Take any bijection $\pi: \mathbb{N} \to \mathbb{N}$ and assume that under this reordering we know that $a_{\pi(k)} \leq b_k$ for some nonincreasing sequence $b_1 \geq b_2 \geq \hdots$ (i.e. an upper bound). Then this implies that $a_k \leq b_k$.
\end{lemma}
\begin{proof}
    Note again that the $a_k$ have a limit by the monotone convergence theorem and are bounded from above. Without loss of generality, the tightest nonincreasing upper bound for the sequence $a_{\pi(k)}$ is $b_k = \sup_{m \geq k} a_{\pi(m)} < \infty$ so it suffices to prove the lemma for this upper bound. We proceed to prove this result by contradiction. Assume there exists a $k \geq 1$ such that $a_k > b_k = \sup_{m \geq k} a_{\pi(m)}$. Then by the monotonicity of the $a_k$ sequence this implies that $k < \pi(m)$ for all $m \geq k$. This can be rewritten as $\{\pi(m): m \geq k\} \subseteq \{k+1, k+2, \hdots\}$ which implies that $\{\pi^{-1}(m): 1 \leq m \leq k\} \subseteq \{1, \hdots, k-1\}$. But since $\pi^{-1}$ is also a bijection then we have reached a contradication since the cardinality of $\{1, \hdots, k\}$ is larger than $\{1, \hdots, k-1\}$ so $\pi^{-1}$ is not injective.
\end{proof}

\section{Decay of Kernel Eigenvalues}\label{app:eig_decay}
In this section, we describe two important cases in which eigenvalues of kernel functions decay at polynomial or exponential rates.\\

\noindent\underline{Example 1:} Assume we have an RKHS with a reproducing kernel $K: D \times D \to \mathbb{R}$ on a compact domain $D$ that is a Mercer kernel (e.g. continuous, square integrable, and nonnegative definite). Then we know that $K$ is a bounded function on $D$ so
\begin{equation}
    \Tr(K) = \int_D K(x, x)\ud x < \infty
\end{equation}
Furthermore, by Mercer's theorem we know that $K$ has the eigenexpansion
\begin{equation}
    K(x, y) = \sum_{n=1}^\infty \lambda_n \phi_n(x)\phi_n(y)
\end{equation}
for nonnegative eigenvalues $\lambda_1 \geq \lambda_2 \geq \hdots$ and $L^2(D)$ orthonormal eigenfunctions $\phi_k$. This implies that
\begin{equation}
    \Tr(K) = \int_D K(x, x)\ud x = \sum_{n=1}^\infty \lambda_n.
\end{equation}
Since the series $\sum_{n=1}^\infty \lambda_n$ converges and the sequence $\{\lambda_n\}_{n=1}^\infty$ is nonnegative and decreasing, this implies that $n\lambda_n \to 0$ as $n \to \infty$. Hence, this proves that
\begin{equation}
    \lambda_n \lesssim \frac{1}{n}.
\end{equation}
where again the notation $a \lesssim b$ means that $a \leq Cb$ for some $C \geq 0$. In general, the spectrum of a reproducing kernel (assumed to be positive semidefinite) decreases at rate at least $\frac{1}{n}$ if it is a trace class linear operator. Also, this statement is sharp over the space of reproducing kernels as we can construct for any $\epsilon > 0$ a symmetric, positive semidefinite, square integrable reproducing kernel of the form
\begin{equation}
    K(x, y) = \sum_{n=1}^\infty n^{-(1+\epsilon)}\phi_n(x)\phi_n(y).
\end{equation}
for any $\epsilon > 0$. For an arbitrary domain $D$, as long as $L^2(D)$ is a separable Hilbert space we can always construct an infinitely large countable basis. Note that $\Tr(K) = \sum_{n=1}^\infty n^{-(1+\epsilon)} < \infty$ by the integral test.\\

\newpage
\noindent\underline{Example 2:} A large class of reproducing kernels can be built by taking tensor products of simpler ones. For example, taking the product of two squared exponential kernels $\frac{1}{\sqrt{2\pi}\sigma_1}e^{-\frac{(x-\xi)^2}{2\sigma_1^2}}$ and $\frac{1}{\sqrt{2\pi}\sigma_2}e^{-\frac{(y-\eta)^2}{2\sigma_2^2}}$ gives us another squared exponential kernel. In general, take $m$ RKHSs with reproducing kernels $K_i: D_i \times D_i \to \mathbb{R}$ where $D_i \subset \mathbb{R}^{d_i}$ for $i = 1, \hdots, m$. Defining the product kernel $K: (D_1 \times \hdots \times D_m)^2 \to \mathbb{R}$ where
\begin{equation}
    K(x_1, \hdots, x_m, y_1, \hdots, y_m) = \prod_{i=1}^m K_i(x_i, y_i)
\end{equation}
then by~\citet[Section~5, Theorem~5.24]{paulsen2016introduction} we know that $K$ is also a reproducing kernel of some RKHS. If we have information about the rate of decay of the eigenvalues of each kernel $K_i$, then we can bound the rate of decay of the tensor product kernel $K$.

Let us assume that the sorted eigenvalues of each $K_i$ are $\{\lambda_k^{(i)}\}_{k=1}^\infty$ and that they decay polynmially at the rate $\lambda_k^{(i)} \lesssim k^{-r_i}$ for some $r_i > 0$. Note that the eigenvalues of the tensor product kernel $K$ are $\rho_{k_1 \hdots k_m} = \prod_{i=1}^m \lambda_{k_i}^{(i)}$ for all $k_1 \hdots k_m \geq 1$. To quantify the rate of decrease of the eigenvalues of $K$, we first order them as $\{\rho_n\}_{n=1}^\infty$ using the $m$-tupling function
\begin{equation}
    \pi^{(m)}(k_1, \hdots, k_m) = \pi(\pi^{(m-1)}(k_1, \hdots, k_{m-1}), k_m), \quad m > 2
\end{equation}
defined recursively where
\begin{equation}
    \pi^{(2)}(k_1, k_2) = \pi(k_1, k_2) = \frac{1}{2}(k_1 + k_2 - 2)(k_1 + k_2 - 1) + k_2
\end{equation}
is the Cantor pairing function. For any $\rho_n = \rho_{k_1 \hdots k_m}$ in our ordered sequence where $n = \pi^{(m)}(k_1, \hdots, k_m)$, define the sum $s(n) = \sum_{i=1}^m k_i$. It is not hard to check that $\prod_{i=1}^m x_i^{-r_i}$ for $(x_1, \hdots, x_m) \in \mathbb{R}^m$ is a convex function on the convex constraint set $x_i \geq 1$ with $\sum_{i=1}^m x_i = s(n)$. Therefore, the global maximum of this function is located at one of the extremal points $x_i = s(n) - m + 1$ with $x_j = 1$ for all $j \neq i$. As shorthand, write $[m] = \{1, \hdots, m\}$. Defining $i_{\min} = \argmin_{i \in [m]} r_i$ and $r_{\min} = r_{i_{\min}}$ we know that this function over the convex constraint set is maximized at the extremal point $x_{i_{\min}} = s(n) - m + 1$ with $x_j = 1$ for all $j \neq i_{\min}$. The maximal value it takes there is $(s(n) - m + 1)^{-r_{\min}}$. This implies that
\begin{equation}
    \rho_n = \rho_{k_1, \hdots, k_m} = \prod_{i=1}^m\lambda_{k_i}^{(i)} \lesssim \prod_{i=1}^m k_i^{-r_i} \lesssim s(n)^{-r_{\min}}
\end{equation}
where again the notation $a \lesssim b$ means that $a \leq Cb$ for some $C \geq 0$. If we iterate over $m$-tuples $(k_1, \hdots, k_m)$ in the order prescribed by the pairing function $\pi^{(m)}$, then by the time we have reached a tuple where $n = \pi^{(m)}(k_1, \hdots, k_m)$ we must have iterated over \textit{at most} all of the positive $m$-tuples with sums $m, \hdots, s(n)$. The number of such $m$-tuples is exactly $\sum_{k=m}^{s(n)}{k-1 \choose m-1}$ because ${k-1 \choose m-1}$ is the number of tuples whose sum is exactly $k$ through a stars and bars argument. Now by the Hockey-stick identity we see that $\sum_{k=m}^{s(n)}{k-1 \choose m-1} = {s(n) \choose m} \lesssim s(n)^m$. Finally, this implies that $n \lesssim s(n)^m$ so $s(n) \gtrsim n^\frac{1}{m}$ and hence
\begin{equation}
    \rho_n \lesssim n^{-\frac{r_{\min}}{m}}.
\end{equation}
Since the $\rho_n$ enumerated by the $m$-tupling function $\pi^{(m)}$ tend to zero, then by Lemma~\ref{lemma:rearrangement} the sorted eigenvalues $\rho_1 \geq \rho_2 \geq \hdots$ of $K$ also decrease at the same rate as above.

Following a similar analysis, if all of the sorted eigenvalues $\{\lambda_k^{(i)}\}_{k=1}^\infty$ of $K_i$ decay exponentially at the rate $\lambda_k^{(i)} \lesssim \exp(-a_ik^{r_i})$ for all $i \in [d]$ then under the constraint $s(n) = \sum_{i=1}^m k_i$ we have that
\begin{equation}
    \rho_n = \rho_{k_1, \hdots, k_m} = \prod_{i=1}^m\lambda_{k_i}^{(i)} = e^{-\sum_{i=1}^m r_i k_i} \lesssim e^{-a_{i_{\min}}s(n)^{r_{i_{\min}}}}.
\end{equation}
where $i_{\min} = \argmin_{i \in [m]} r_i$. Since we know that $s(n) \gtrsim n^\frac{1}{m}$ then finally
\begin{equation}
    \rho_n \lesssim \exp\Big(-a_{i_{\min}} \cdot n^\frac{r_{i_{\min}}}{m}\Big).
\end{equation}
Again by Lemma~\ref{lemma:rearrangement} the sorted eigenvalues $\rho_1 \geq \rho_2 \geq \hdots$ of $K$ also decrease at the same rate as above.

\newpage

\section{SubGaussian Functions}\label{app:subG}
In this appendix we define what it means for a random function to be subgaussian with respect to a variance proxy operator, similar to the definition of subgaussianity for real random variables.
\begin{definition}
 Let $\Gamma: L^2(D) \to L^2(D)$ be a positive semidefinite trace class linear operator. A random variable $F \in L^2(D)$ is subgaussian with respect to variance proxy $\Gamma$ (written as $F \sim \subG(\Gamma)$) if there exists an $\alpha \geq 0$ such that for all $f \in L^2(D)$,
    \begin{equation}
        \mathbb{E}\Big[e^{\inprod{f, F - \mathbb{E}[F]}_{L^2(D)}}\Big] \leq e^{\alpha^2\inprod{\Gamma f, f}_{L^2(D)}/2}
    \end{equation}
    Furthermore, if $F \sim \subG(\Gamma)$ then the $\psi_2$-norm of $F$ with respect to $\Gamma$ is defined as
    \begin{equation}
        \|F\|_{\psi_2, \Gamma} = \inf\Big\{\alpha \geq 0: \mathbb{E}\Big[e^{\inprod{f, F - \mathbb{E}[F]}_{L^2(D)}}\Big] \leq e^{\alpha^2\inprod{\Gamma f, f}_{L^2(D)}/2}, \ \forall f \in L^2(D)\Big\}.
    \end{equation}
\end{definition}

We denote the covariance operator of $F$ as the positive semidefinite linear operator $\Sigma_F: L^2(D_\cY) \to L^2(D_\cY)$ which is identified with the function
\begin{equation}
    \Sigma_F = \mathbb{E}[(F - \E[F]) \otimes (F - \E[F])].
\end{equation}
By~\citet[Section~2.1, Lemma~2.4]{chen2018hanson} we know that 
\begin{equation}
\Sigma_F \preceq 4\|F\|_{\psi_2, \Gamma_F}^2\Gamma_F
\end{equation}
and we also make the following key assumption.
\begin{definition}[Strict subgaussianity]
    The random variable $F \sim \subG(\Gamma_F)$ is called strictly subgaussian if its covariance operator and covariance proxy satisfy
    \begin{equation}
        \Gamma_F \preceq C\Sigma_F
    \end{equation}
    for some fixed constant $C > 0$.
\end{definition}
The definition above is a natural extension of \textit{strict subgaussianity} to random variables in Hilbert spaces. It is trivially satisfied for Gaussian random variables $Z \in L^2(D)$ for which it is easy to check that $Z \sim \subG(\Sigma)$ where $\Sigma = \mathbb{E}[(Z - \mathbb{E}[Z]) \otimes (Z - \mathbb{E}[Z])]$ and $\|Z\|_{\psi_2, \Sigma} = 1$.~\\

\begin{lemma}\label{lem:subgcomp}
    Take any subgaussian random variable $F \sim \subG(\Gamma)$ in $L^2(D_\cX)$ where $\Gamma \in L^2(D_\cX \times D_\cX)$ is a trace class linear operator. Then for any operator $H \in L^2(D_\cX \times D_\cY)$ we have that $H(F) \sim \subG(H\Gamma H^*)$ and $\|H(F)\|_{\psi_2, H\Gamma H^*} \leq \|F\|_{\psi_2, \Gamma}$.
\end{lemma}
\begin{proof}
    First since $F \sim \subG(\Gamma)$ then there exists an $\alpha \geq 0$ such that for all $f \in L^2(D_\cX)$,
    \begin{equation}
        \mathbb{E}\Big[e^{\inprod{f, F - \mathbb{E}[F]}_{L^2(D_\cX)}}\Big] \leq e^{\alpha^2\inprod{\Gamma f, f}_{L^2(D_\cX)}/2}.
    \end{equation}
    So for the same $\alpha$ we can write for all $g \in L^2(D_\cY)$,
    \begin{equation}
        \mathbb{E}\Big[e^{\inprod{g, H(F) - \mathbb{E}[H(F)]}_{L^2(D_\cY)}}\Big] = \mathbb{E}\Big[e^{\inprod{H^*(g), F - \mathbb{E}[F]}_{L^2(D_\cX)}}\Big] \leq e^{\alpha^2\inprod{H\Gamma H^*(g), g}_{L^2(D_\cY)}/2}
    \end{equation}
    which proves that $H(F) \sim \subG(H\Gamma H^*)$. Furthermore, we know that the $\psi_2$-norm of $F$ with respect to $\Gamma$ is defined as
    \begin{equation}
        \|F\|_{\psi_2, \Gamma} = \inf\Big\{\alpha \geq 0: \mathbb{E}\Big[e^{\inprod{g, F - \mathbb{E}[F]}_{L^2(D_\cX)}}\Big] \leq e^{\alpha^2\inprod{\Gamma g, g}_{L^2(D_\cX)}/2}, \ \forall g \in L^2(D_\cX)\Big\}.
    \end{equation}
    which implies by the derivations above that
    \begin{equation}
    \begin{aligned}
        \|F\|_{\psi_2, \Gamma} &\geq \inf\Big\{\alpha \geq 0: \mathbb{E}\Big[e^{\inprod{g, H(F) - \mathbb{E}[H(F)]}_{L^2(D_\cY)}}\Big] \leq e^{\alpha^2\inprod{H\Gamma H^*(g), g}_{L^2(D_\cY)}/2}, \ \forall g \in L^2(D_\cY)\Big\}\\
        &= \|H(F)\|_{\psi_2, H\Gamma H^*}.
    \end{aligned}
    \end{equation}
\end{proof}

\begin{remark}\label{rem:psi1}
    Define the usual $\psi_1$-norm for a (possibly noncentered) subexponential real-valued random variable $R$,
    \begin{equation}
        \|R\|_{\psi_1} = \inf\Big\{t > 0: \mathbb{E}\Big[\exp\Big(\frac{|R|}{t}\Big)\Big] \leq 2\Big\}.
    \end{equation}
    Then for any centered subgaussian vectors in $L^2(D)$ denoted by $X \sim \subG(\Sigma)$ and $Y \sim \subG(\Gamma)$ we have that
    \begin{equation}
    \begin{aligned}
        \|\inprod{X, &Y}_{L^2(D)}\|_{\psi_1} := \inf\Big\{t > 0: \mathbb{E}\Big[\exp\Big(\frac{|\inprod{X, Y}_{L^2(D)}|}{t}\Big)\Big] \leq 2\Big\}\\
        &\leq \inf\Big\{t > 0: \mathbb{E}\Big[\exp\Big(\frac{\|X\|_{L^2(D)}\|Y\|_{L^2(D)}}{t}\Big)\Big] \leq 2\Big\}\\
        &\leq \Big\|\|X\|_{L^2(D)}\Big\|_{\psi_2}\Big\|\|Y\|_{L^2(D)}\Big\|_{\psi_2}\\
        &= \Big\|\|X\|_{L^2(D)}^2\Big\|_{\psi_1}^\frac{1}{2}\Big\|\|Y\|_{L^2(D)}^2\Big\|_{\psi_1}^\frac{1}{2}
    \end{aligned}
    \end{equation}
    By~\citet[Appendix~A.2, Lemma~A.4]{chen2018hanson} we know there exists a universal constant $c > 0$ such that $\|\|X\|_{L^2(D)}^2\Big\|_{\psi_1} \leq c\|X\|_{\psi_2, \Sigma}^2\Tr(\Sigma)$ which implies that
    \begin{equation}
        \|\inprod{X, Y}_{L^2(D)}\|_{\psi_1} \leq c\|X\|_{\psi_2, \Sigma}\|Y\|_{\psi_2, \Gamma}\Tr(\Sigma)^\frac{1}{2}\Tr(\Gamma)^\frac{1}{2}.
    \end{equation}
\end{remark}

\newpage
\section{Results for Error Analysis}\label{app:rates}
In this appendix we derive all of the results for the error analysis of our Green's function estimator. The following results closely follow the derivations in~\citet{yuan2010reproducing} for bounding the error of scalar-output functional linear regression and extend them to the case of functional-output functional linear regression.~\\
\begin{proof}[Proposition~\ref{prop:equivnorms}]
First note that
\begin{equation}
    \|G\|_{\Sigma_F}^2 \leq \mu_1\int_{D_\cY} \|G(x, \cdot)\|_{L^2(D_\cX)}^2\ud x = \mu_1\|G\|_{L^2(D_\cX \times D_\cY)}^2 \leq c_1\|G\|_\cG^2
\end{equation}
where the last inequality follows by Cauchy--Schwarz since
\begin{equation}
\begin{aligned}
    \|G\|_{L^2(D_\cX \times D_\cY)}^2 &= \int_{D_\cY}\int_{D_\cX}\inprod{K_{(x, y)}, G}_\cG^2\ud x \ud y \leq \int_{D_\cY}\int_{D_\cX}\|K_{(x, y)}\|_\cG^2 \|G\|_\cG^2\ud x \ud y\\
    &= \Big(\int_{D_\cY}\int_{D_\cX} K(x, y, x, y)^2\ud x \ud y\Big)\|G\|_\cG^2
\end{aligned}
\end{equation}
where the integral on the right hand side is finite since $D$ is closed and bounded and $K$ is continuous. From this, since $J(G) = \|G\|_\cG^2$ we immediately have that
\begin{equation}
    \|G\|_{\overline{K}}^2 = \|G\|_{\Sigma_F}^2 + J(G) \leq (c_1 + 1)\|G\|_\cG^2.
\end{equation}
The reverse inequality follows immediately since by definition
\begin{equation}
    \|G\|_\cG^2 = J(G) \leq \|G\|_{\overline{K}}^2.
\end{equation}
for any $G \in \cG$.
This proves that $\|\cdot\|_\cG$ and $\|\cdot\|_{\overline{K}}$ are equivalent norms on $\cG$.

From this we immediately see that $\|\cdot\|_{\overline{K}}$ is a valid norm on $\cG$ since it is zero only at $0 \in \cG$ and finite for all $G \in \cG$. Furthermore, $\cG$ equipped with $\|\cdot\|_\cG$ is an RKHS iff all the linear evaluation functionals $L_{(x, y)}: G \to G(x, y)$ are bounded
\begin{equation}
    |L_{(x, y)}(G)|:= |G(x, y)| \leq M\|G\|_\cG, \ \forall G \in \cG.
\end{equation}
This proves that $\cG$ equipped with $\|\cdot\|_{\overline{K}}$ is also an RKHS because all of the linear functionals are once again bounded
\begin{equation}
    |L_{(x, y)}(G)|:= |G(x, y)| \leq M\|G\|_\cG \leq M\|G\|_{\overline{K}}, \ \forall G \in \cG.
\end{equation}
\end{proof}

\begin{proof}[Theorem~\ref{thm:simdiag}]\\
First we write out for any $G \in \cG$,
\begin{equation}
\begin{aligned}
    \overline{K}^{-\frac{1}{2}}G &= \sum_{k=1}^\infty \inprod{\overline{K}^{-\frac{1}{2}}G, \Gamma_k}_{L^2(D_\cX \times D_\cY)}\Gamma_k = \sum_{k=1}^\infty \inprod{\overline{K}^{-\frac{1}{2}}G, \nu_k^\frac{1}{2}\overline{K}^{-\frac{1}{2}}\Omega_k}_{L^2(D_\cX \times D_\cY)}\nu_k^\frac{1}{2}\overline{K}^{-\frac{1}{2}}\Omega_k\\
    &= \sum_{k=1}^\infty \nu_k\inprod{\overline{K}^{-1}G, \Omega_k}_{L^2(D_\cX \times D_\cY)}\overline{K}^{-\frac{1}{2}}\Omega_k = \overline{K}^{-\frac{1}{2}}\Big(\sum_{k=1}^\infty\nu_k\inprod{G, \Omega_k}_{\overline{K}}\Omega_k\Big).
\end{aligned}
\end{equation}
Applying $\overline{K}^\frac{1}{2}$ to both sides we see that
\begin{equation}
    G = \sum_{k=1}^\infty g_k\Omega_k, \quad g_k = \nu_k\inprod{G, \Omega_k}_{\overline{K}}
\end{equation}
which converges absolutely. Now let $\gamma_k = (\nu_k^{-1} - 1)^{-1}$. Then we can use the fact that $\inprod{\Omega_k, \Omega_j}_{\overline{K}} = \nu_k^{-1}\delta_{kj}$ to write
\begin{equation}
    \|G\|_{\overline{K}}^2 = \Big\langle\sum_{k=1}^\infty g_k\Omega_k, \sum_{j=1}^\infty g_j\Omega_j\Big\rangle_{\overline{K}} = \sum_{k=1}^\infty \nu_k^{-1}g_k^2 = \sum_{k=1}^\infty (1 + \gamma_k^{-1})g_k^2.
\end{equation}
Similarly, since $\inprod{(\Sigma_F \otimes I)\Omega_k, \Omega_j}_{L^2(D_\cX \times D_\cY)} = \delta_{kj}$ we have
\begin{equation}
    \inprod{(\Sigma_F \otimes I)G, G}_{L^2(D_\cX \times D_\cY)} = \Big\langle\sum_{k=1}^\infty g_k(\Sigma_F \otimes I)\Omega_k, \sum_{j=1}^\infty g_j\Omega_j\Big\rangle_{L^2(D_\cX \times D_\cY)} = \sum_{k=1}^\infty g_k^2.
\end{equation}
\end{proof}

\begin{proof}[Proposition~\ref{prop:simdiagsimple}]\\
For any $\{\varphi_j\}_{j=1}^\infty$ which is an orthonormal basis of $L^2(D_\cX)$, it is not hard to see that $\{\phi_i \otimes \varphi_j\}_{i, j = 1}^\infty$ is an eigenbasis of $\Sigma_F \otimes I$ where
\begin{equation}
    (\Sigma_F \otimes I)(\phi_i \otimes \varphi_j) = \mu_i\phi_i \otimes \varphi_j.
\end{equation}
By definition of the reproducing kernel $K$ for $\cG$, since $\spn\{\phi_i \otimes \varphi_j\}_{i, j = 1}^\infty$ are the eigenfunctions of $K$ then we must have that $\cG = \spn\{\phi_i \otimes \varphi_j\}_{i, j = 1}^\infty$. By the definition of the induced inner product, we know that for all $F, G \in \cG$,
\begin{equation}
    \inprod{F, G}_{\overline{K}} = \inprod{(\Sigma_F \otimes I)F, G}_{L^2(D_\cX \times D_\cY)} + \inprod{F, G}_{K}.
\end{equation}
Therefore, we can write out for all $1 \leq i, i', j, j' < \infty$,
\begin{equation}
    \inprod{\phi_i \otimes \varphi_j, \phi_{i'} \otimes \varphi_{j'}}_{\overline{K}} = \mu_i\delta_{ii'}\delta_{jj'} + \rho_{ij}^{-1}\delta_{ii'}\delta_{jj'} = (\mu_i + \rho_{ij}^{-1})\delta_{ii'}\delta_{jj'}.
\end{equation}
Since $\overline{K}$ is invertible over $\cG$ and $\inprod{\phi_i \otimes \varphi_j, \phi_{i'} \otimes \varphi_{j'}}_{\overline{K}} = \inprod{\overline{K}^{-1}(\phi_i \otimes \varphi_j), \phi_{i'} \otimes \varphi_{j'}}_{L^2(D_\cX \times D_\cY)}$ this implies that $\overline{K}$ has the eigendecomposition
\begin{equation}
    \overline{K}(x, y, \xi, \eta) = \sum_{i=1}^\infty\sum_{j=1}^\infty (\mu_{j} + \rho_{ij}^{-1})^{-1}\phi_i(x)\phi_i(\xi)\varphi_j(y)\varphi_j(\eta)
\end{equation}
and since the $\{\phi_i\}_{i=1}^\infty$ and $\{\varphi_j\}_{j=1}^\infty$ are orthonormal in $L^2(D_\cX)$ we have that
\begin{equation}
    \overline{K}(\phi_i \otimes \varphi_j) = (\mu_i + \rho_{ij}^{-1})^{-1}\phi_i \otimes \varphi_j
\end{equation}
which implies that
\begin{equation}
\begin{aligned}
    \overline{K}^\frac{1}{2}(\Sigma_F \otimes I)\overline{K}^\frac{1}{2}(\phi_i \otimes \varphi_j) &= (\mu_i + \rho_{ij}^{-1})^{-\frac{1}{2}}\overline{K}^\frac{1}{2}(\Sigma_F \otimes I)(\phi_i \otimes \varphi_j)\\
    &= \mu_i(\mu_i + \rho_{ij}^{-1})^{-\frac{1}{2}}\overline{K}^\frac{1}{2}(\phi_i \otimes \varphi_j)\\
    &= (1 + \mu_i^{-1}\rho_{ij}^{-1})^{-1}\phi_i \otimes \varphi_j.
\end{aligned}
\end{equation}
This proves that the eigenfunctions of $\overline{K}^\frac{1}{2}(\Sigma_F \otimes I)\overline{K}^\frac{1}{2}$ are $\Gamma_{ij} = \Psi_{ij} = \phi_i \otimes \varphi_j$ with eigenvalues $\nu_{ij} = (1 + \mu_i^{-1}\rho_{ij}^{-1})^{-1}$. Therefore, we have that $\gamma_{ij} = (\nu_{ij}^{-1} - 1)^{-1} = \mu_i\rho_{ij}$ and $\Omega_{ij} = \nu_{ij}^{-\frac{1}{2}}\overline{K}^\frac{1}{2}\Gamma_{ij} = \mu_i^{-\frac{1}{2}}\phi_i \otimes \varphi_j = \mu_i^{-\frac{1}{2}}\Psi_{ij}$.\\

Now reorder the $\rho_{ij}$ in decreasing order and enumerate them as $\rho_k$. Assume that these sorted decreasing eigenvalues satisfy $\rho_k \lesssim k^{-r}$ for some $r > \frac{1}{2}$ as in Assumption~\ref{assump:spectradecay}.

Let us reorder the coefficients $\gamma_{ij}$ and enumerate them as $\gamma_k$ in the same way as the $\rho_k$ eigenvalues. Since the covariance $\Sigma_F$ is a bounded operator then all of the eigenvalues $\mu_i$ are bounded. Finally, this proves that $\gamma_k \lesssim \rho_k \lesssim k^{-r}$ for some $r > \frac{1}{2}$.
\end{proof}

\newpage
\begin{proof}[Lemma~\ref{lem:stocherrbound}]
We would like to bound the term $\|\widehat{G}_{n, \lambda} - \overline{G}_{\infty, \lambda}\|_{\Sigma_F}$. Define the variational derivatives for any $G, P, H \in \cG$,
\begin{equation}
\begin{aligned}
    \delta\widehat{R}(G)[H] &= -\frac{2}{n}\sum_{i=1}^n \inprod{U_i - G(F_i), H(F_i)}_{L^2(D_\cY)}\\
    \delta R(G)[H] &= -2\mathbb{E}\Big[\inprod{U - G(F), H(F)}_{L^2(D_\cY)}\Big]\\
    \delta^2\widehat{R}[P, H] &= \frac{2}{n}\sum_{i=1}^n \inprod{P(F_i), H(F_i)}_{L^2(D_\cY)} = 2\inprod{(\widehat{\Sigma}_F \otimes I)P, H}_{L^2(D_\cX \times D_\cY)}\\
    \delta^2R[P, H] &= 2\mathbb{E}\Big[\inprod{P(F), H(F)}_{L^2(D_\cY)}\Big] = 2\inprod{(\Sigma_F \otimes I)P, H}_{L^2(D_\cX \times D_\cY)}.
\end{aligned}
\end{equation}
where $\widehat{\Sigma}_F = \frac{1}{n}\sum_{i=1}^n F_i \otimes F_i$ and $\Sigma_F = \mathbb{E}[F \otimes F]$. Remember that since $\cG$ equipped with $\|\cdot\|_{\overline{K}}$ is an RKHS, then $\|H\|_{L^2(D_\cX \times D_\cY)} \leq c\|H\|_{\overline{K}}$ for some universal constant $c > 0$. Viewing the variational derivatives above as functionals of $H \in \cG$, it is not hard to see that they are bounded in the $\|\cdot\|_{\overline{K}}$ norm. Hence, by the Riesz representation theorem there exist $\nabla\widehat{R}(G), \nabla R(G), \nabla^2\widehat{R}(P),$ and $\nabla^2R(P) \in \cG$ such that
\begin{equation}
\begin{aligned}
    \delta\widehat{R}(G)[H] &= \inprod{\nabla\widehat{R}(G), H}_{\overline{K}}, \quad \delta R(G)[H] = \inprod{\nabla R(G), H}_{\overline{K}}\\
    \delta^2\widehat{R}[P, H] &= \inprod{\nabla^2\widehat{R}(P), H}_{\overline{K}}, \quad \delta^2R[P, H] = \inprod{\nabla^2R(P), H}_{\overline{K}}.
\end{aligned}
\end{equation}
Denoting $\widehat{R}_\lambda(G) = \widehat{R}(G) + \lambda J(G)$ and $R_\lambda(G) = R(G) + \lambda J(G)$ we can similarly compute the first and second variational derivatives $\delta\widehat{R}_\lambda(G)[H], \delta R_\lambda(G)[H], \delta^2\widehat{R}_\lambda[P, H], \delta^2R_\lambda[P, H]$ and show that these functionals of $H \in \cG$ are bounded in norm $\|\cdot\|_{\overline{K}}$. Hence, there exist $\nabla\widehat{R}_\lambda(G), \delta R_\lambda(G), \nabla^2\widehat{R}_\lambda(P),$ and $\nabla^2R_\lambda(P) \in \cG$ such that
\begin{equation}
\begin{aligned}
    \delta\widehat{R}_\lambda(G)[H] &= \inprod{\nabla\widehat{R}_\lambda(G), H}_{\overline{K}}, \quad \delta R_\lambda(G)[H] = \inprod{\nabla R_\lambda(G), H}_{\overline{K}}\\
    \delta^2\widehat{R}_\lambda[P, H] &= \inprod{\nabla^2\widehat{R}_\lambda(P), H}_{\overline{K}}, \quad \delta^2R_\lambda[P, H] = \inprod{\nabla^2R_\lambda(P), H}_{\overline{K}}.
\end{aligned}
\end{equation}
We interpret the Hessians $\nabla^2\widehat{R}, \nabla^2R, \nabla^2\widehat{R}_\lambda,$ and $\nabla^2R_\lambda$ as maps from $\cG$ to $\cG$. Since $\inprod{\Omega_k, \Omega_j}_{\overline{K}} = \nu_k^{-1}\delta_{kj}$, this immediately implies for all $k \geq 1$ and $G \in \cG$ that
\begin{equation}
    \inprod{\nabla R(G), \Omega_k}_{\overline{K}} = \delta R(G)[\Omega_k] \implies \nabla R(G) = \sum_{k=1}^\infty \nu_k\delta R(G)[\Omega_k]\Omega_k
\end{equation}
and likewise for $\nabla\widehat{R}, \nabla\widehat{R}_\lambda, \nabla R_\lambda, \nabla^2\widehat{R}, \nabla^2R, \nabla^2\widehat{R}_\lambda, \nabla^2R_\lambda$. For all $G = \sum_{k=1}^\infty g_k\Omega_k \in \cG$ we can write out a series expansion for $\nabla^2R_\lambda$ as
\begin{equation}
    \nabla^2R_\lambda(G) = 2\sum_{k=1}^\infty \nu_k(1 + \lambda\gamma_k^{-1})g_k\Omega_k.
\end{equation}
It is not hard to check using the series expansions from Theorem~\ref{thm:simdiag} that indeed
\begin{equation}
    \inprod{\nabla^2R_\lambda(G), H}_{\overline{K}} = 2\inprod{G, H}_{\Sigma_F}^2 + 2\lambda\inprod{G, H}_{K} = \delta^2R_\lambda[G, H]
\end{equation}
for all $G, H \in \cG$. Finally, we can define the ``linearization" of $\widehat{G}_{n, \lambda}$ as $\tilde{G} \in \cG$ where
\begin{equation}
    \tilde{G} = \overline{G}_{\infty, \lambda} - \nabla^2R_\lambda^{-1}\Big(\nabla\widehat{R}_\lambda(\overline{G}_{\infty, \lambda})\Big), \quad \nabla^2R_\lambda^{-1}(G) = \frac{1}{2}\sum_{k=1}^\infty \nu_k^{-1}(1+\lambda\gamma_k^{-1})^{-1}g_k\Omega_k
\end{equation}
following the analysis of the rate of convergence of penalized likelihood estimators in~\citet[Section~1, Equation~1.13]{cox1990asymptotic}. Now decomposing
\begin{equation}
    \widehat{G}_{n, \lambda} - \overline{G}_{\infty, \lambda} = (\widehat{G}_{n, \lambda} - \tilde{G}) + (\tilde{G} - \overline{G}_{\infty, \lambda})
\end{equation}
we bound both terms on the right-hand side.\\

\noindent\underline{1. Bounding $\|\tilde{G} - \overline{G}_{\infty, \lambda}\|_{\Sigma_F}$}~\\

First noting that $\delta R_\lambda(\overline{G}_{\infty, \lambda}) = 0$ by definition then
\begin{equation}
    \delta\widehat{R}_\lambda(\overline{G}_{\infty, \lambda}) = \delta\widehat{R}_\lambda(\overline{G}_{\infty, \lambda}) - \delta R_\lambda(\overline{G}_{\infty, \lambda}) = \delta\widehat{R}(\overline{G}_{\infty, \lambda}) - \delta R(\overline{G}_{\infty, \lambda}).
\end{equation}
So for any $H \in \cG$ we can write
\begin{equation}
\begin{aligned}
    \Big(\delta\widehat{R}_\lambda(\overline{G}_{\infty, \lambda})[H]\Big)^2 &= \Big(\delta\widehat{R}(\overline{G}_{\infty, \lambda})[H] - \delta R(\overline{G}_{\infty, \lambda})[H]\Big)^2\\
    &= 4\Big(\frac{1}{n}\sum_{i=1}^n \inprod{U_i - \overline{G}_{\infty, \lambda}(F_i), H(F_i)}_{L^2(D_\cY)} - \mathbb{E}\Big[\inprod{U - \overline{G}_{\infty, \lambda}(F), H(F)}_{L^2(D_\cY)}\Big]\Big)^2.
\end{aligned}
\end{equation}
Using the fact that $U_i = \cT^*(F_i) + \epsilon_i$ we get by Jensen's and Young's inequality
\begin{equation}
\begin{aligned}
    &\Big(\delta\widehat{R}_\lambda(\overline{G}_{\infty, \lambda})[H]\Big)^2\\
    &\leq 4\Big(\frac{1}{n}\sum_{i=1}^n \inprod{(\overline{G}_{\infty, \lambda} - G_\cG)(F_i), H(F_i)}_{L^2(D_\cY)} - \mathbb{E}\Big[\inprod{(\overline{G}_{\infty, \lambda} - G_\cG)(F), H(F)}_{L^2(D_\cY)}\Big]\Big)^2\\
    &\qquad+ 4\Big(\frac{1}{n}\sum_{i=1}^n \inprod{G_\cG(F_i) - \cT^*(F_i), H(F_i)}_{L^2(D_\cY)} - \mathbb{E}\Big[\inprod{G_\cG(F) - \cT^*(F), H(F)}_{L^2(D_\cY)}\Big]\Big)^2\\
    &\qquad+ 4\Big(\frac{1}{n}\sum_{i=1}^n \inprod{\varepsilon_i, H(F_i)}_{L^2(D_\cY)} - \mathbb{E}[\inprod{\varepsilon, H(F)}_{L^2(D_\cY)}]\Big)^2\\
    &= 4\Big(\frac{1}{n}\sum_{i=1}^n \inprod{(\overline{G}_{\infty, \lambda} - G_\cG)(F_i), H(F_i)}_{L^2(D_\cY)} - \mathbb{E}\Big[\inprod{(\overline{G}_{\infty, \lambda} - G_\cG)(F), H(F)}_{L^2(D_\cY)}\Big]\Big)^2\\
    &\qquad+ 4\Big(\frac{1}{n}\sum_{i=1}^n \inprod{G_\cG(F_i) - \cT^*(F_i), H(F_i)}_{L^2(D_\cY)}\Big)^2 + 4\Big(\frac{1}{n}\sum_{i=1}^n \inprod{\varepsilon_i, H(F_i)}_{L^2(D_\cY)}\Big)^2
\end{aligned}
\end{equation}
where the second line uses the fact that $\mathbb{E}[\inprod{\varepsilon, H(F)}_{L^2(D_\cX)}] = 0$ by independence and $\mathbb{E}[\inprod{G_\cG(F) - \cT^*(F), H(F)}_{L^2(D_\cX)}] = \mathbb{E}[\inprod{G_\cG(F) - U, H(F)}_{L^2(D_\cX)}] + \mathbb{E}[\inprod{\varepsilon, H(F)}_{L^2(D_\cX)}] = 0$. Combining Lemma~\ref{lem:subgcomp} with Remark~\ref{rem:psi1}, the $\psi_1$-norm for the first term above is
\begin{equation}
\begin{aligned}
    \|\inprod{(\overline{G}_{\infty, \lambda} - G_\cG)(F), H(F)}_{L^2(D_\cY)}\|_{\psi_1} &\lesssim \|(\overline{G}_{\infty, \lambda} - G_\cG)(F)\|_{\psi_2, (\overline{G}_{\infty, \lambda} - G_\cG)\Gamma_F(\overline{G}_{\infty, \lambda} - G_\cG)^T}\|H(F)\|_{\psi_2, H\Gamma_FH^T}\\
    &\quad \cdot\Tr((\overline{G}_{\infty, \lambda} - G_\cG)\Gamma_F(\overline{G}_{\infty, \lambda} - G_\cG)^T)^\frac{1}{2}\Tr(H\Gamma_FH^T)^\frac{1}{2}\\
    &\leq \|F\|_{\psi_2, \Gamma_F}^2\|\overline{G}_{\infty, \lambda} - G_\cG\|_{\Gamma_F}\|H\|_{\Gamma_F} \lesssim \|\overline{G}_{\infty, \lambda} - G_\cG\|_{\Gamma_F}\|H\|_{\Gamma_F}
\end{aligned}
\end{equation}
and similarly for the third term above
\begin{equation}
\begin{aligned}
    \|\inprod{\epsilon, H(F)}_{L^2(D_\cX)}\|_{\psi_1} &\lesssim \|\varepsilon\|_{\psi_2, \Gamma_\varepsilon}\|H(F)\|_{\psi_2, H\Gamma_FH^T}\Tr(\Gamma_\varepsilon)^\frac{1}{2}\Tr(H\Gamma_FH^T)^\frac{1}{2}\\
    &\leq \|\varepsilon\|_{\psi_2, \Gamma_\varepsilon}\|F\|_{\psi_2, \Gamma_F}\Tr(\Gamma_\varepsilon)^\frac{1}{2}\|H\|_{\Gamma_F} \lesssim \|H\|_{\Gamma_F}.
\end{aligned}
\end{equation}
For the second term above we apply Remark~\ref{rem:psi1}. Using Assumption~\ref{assump:linear_growth} we can let $M' = \|G_\cG\|_\text{op} + M$ such that,
\begin{equation}
\begin{aligned}
    \|\inprod{G_\cG(F) - \cT^*(F), H(F)}_{L^2(D_\cY)}\|_{\psi_1} &\leq \Big\|\|G_\cG(F) - \cT^*(F)\|_{L^2(D_\cY)}^2\Big\|_{\psi_1}^\frac{1}{2}\Big\|\|H(F)\|_{L^2(D_\cY)}^2\Big\|_{\psi_1}^\frac{1}{2}\\
    &\leq \Big\|M'\|F\|_{L^2(D_\cX)} + c\Big\|_{\psi_2}\Big\|\|H(F)\|_{L^2(D_\cY)}^2\Big\|_{\psi_1}^\frac{1}{2}\\
    &\leq \Big(M'\Big\|\|F\|_{L^2(D_\cX)}^2\Big\|_{\psi_1}^\frac{1}{2} + c'\Big)\Big\|\|H(F)\|_{L^2(D_\cY)}^2\Big\|_{\psi_1}^\frac{1}{2}
\end{aligned}
\end{equation}
and by an application of~\citet[Appendix~A.2, Lemma~A.4]{chen2018hanson} we get that
\begin{equation}
\begin{aligned}
    \|\inprod{G_\cG(F) - \cT^*(F), H(F)}_{L^2(D_\cX)}\|_{\psi_1} &\leq \Big(M'\|F\|_{\psi_2, \Gamma_F}\Tr(\Gamma_F)^\frac{1}{2} + c'\Big)\|H(F)\|_{\psi_2, H\Gamma_FH^T}\Tr(H\Gamma_FH^T)^\frac{1}{2}\\
    &\leq \Big(M'\|F\|_{\psi_2, \Gamma_F}\Tr(\Gamma_F)^\frac{1}{2} + c'\Big)\|F\|_{\psi_2, \Gamma_F}\|H\|_{\Gamma_F}\\
    &\lesssim \max(1, \|G_\cG\|_\text{op} + M)\|H\|_{\Gamma_F}\\
    &\lesssim \max(1, \|G_\cG\|_\text{op})\|H\|_{\Gamma_F}.
\end{aligned}
\end{equation}
Hence by Bernstein's inequality,
\begin{equation}\label{eq:bernstein}
\begin{split}
    &\Big|\frac{1}{n}\sum_{i=1}^n \inprod{(\overline{G}_{\infty, \lambda} - G_\cG)(F_i), H(F_i)}_{L^2(D_\cY)} - \mathbb{E}\Big[\inprod{(\overline{G}_{\infty, \lambda} - G_\cG)(F), H(F)}_{L^2(D_\cY)}\Big]\Big|\\
    &\qquad\qquad\lesssim \|\overline{G}_{\infty, \lambda} - G_\cG\|_{\Gamma_F}\|H\|_{\Gamma_F}\Big(\sqrt{\frac{\log(1/\delta)}{n}} \vee \frac{\log(1/\delta)}{n}\Big)\\
    &\Big|\frac{1}{n}\sum_{i=1}^n \inprod{G_\cG(F_i) - \cT^*(F_i), H(F_i)}_{L^2(D_\cY)}\Big| \lesssim \max(1, \|G_\cG\|_\text{op})\|H\|_{\Gamma_F}\Big(\sqrt{\frac{\log(1/\delta)}{n}} \vee \frac{\log(1/\delta)}{n}\Big)\\
    &\Big|\frac{1}{n}\sum_{i=1}^n \inprod{\varepsilon_i, H(F_i)}_{L^2(D_\cY)}\Big| \lesssim \|H\|_{\Gamma_F}\Big(\sqrt{\frac{\log(1/\delta)}{n}} \vee \frac{\log(1/\delta)}{n}\Big)
\end{split}
\end{equation}
with probability at least $1 - \delta$. Note that we can write the first bound above more generally for all $G, H \in \cG$ as
\begin{equation}\label{eq:covconc}
\begin{split}
    &\Big|\inprod{((\widehat{\Sigma}_F - \Sigma_F) \otimes I)G, H}_{L^2(D_\cX \times D_\cY)}\Big| = \Big|\frac{1}{n}\sum_{i=1}^n \inprod{G(F_i), H(F_i)}_{L^2(D_\cY)} - \mathbb{E}\Big[\inprod{G(F), H(F)}_{L^2(D_\cY)}\Big]\Big|\\
    &\qquad\qquad\lesssim \|G\|_{\Gamma_F}\|H\|_{\Gamma_F}\Big(\sqrt{\frac{\log(1/\delta)}{n}} \vee \frac{\log(1/\delta)}{n}\Big)\\
    &\qquad\qquad\lesssim \|G\|_{\Sigma_F}\|H\|_{\Sigma_F}\Big(\sqrt{\frac{\log(1/\delta)}{n}} \vee \frac{\log(1/\delta)}{n}\Big)
\end{split}
\end{equation}
where the last line follows since $\|G\|_{\Gamma_F} \lesssim \|G\|_{\Sigma_F}$ by Assumption~\ref{assump:strictsubg}.

Finally, since $\|\overline{G}_{\infty, \lambda} - G_\cG\|_{\Sigma_F}^2 \leq \lambda J(G_\cG)$ by the determinstic error derivations we conclude from~\eqref{eq:bernstein} that uniformly over all $H \in \cG$,
\begin{equation}
\begin{aligned}
    \Big(\delta\widehat{R}_\lambda(\overline{G}_{\infty, \lambda})[H]\Big)^2 &\leq \max\Big(1, \|G_\cG\|_\text{op}, \lambda J(G_\cG)\Big)\frac{\log(1/\delta)}{n}\|H\|_{\Gamma_F}^2\\
    &\leq \max\Big(1, \|G_\cG\|_\text{op}, \lambda J(G_\cG)\Big)\frac{\log(1/\delta)}{n}\|H\|_{\Sigma_F}^2
\end{aligned}
\end{equation}
with probability $1 - \delta$ where the last inequality again follows from Assumption~\ref{assump:strictsubg}. For shorthand, let us define $\kappa(G_\cG) = \max\Big(1, \|G_\cG\|_\text{op}, \lambda J(G_\cG)\Big)$. Now we can write
\begin{equation}
\begin{aligned}
    \|\tilde{G} - \overline{G}_{\infty, \lambda}\|_{\Sigma_F}^2 &= \Big\|\nabla^2R_\lambda^{-1}\Big(\nabla\widehat{R}_\lambda(\overline{G}_{\infty, \lambda})\Big)\Big\|_{\Sigma_F}^2 = \frac{1}{4}\Big\|\sum_{k=1}^\infty \nu_k^{-1}(1 + \lambda\gamma_k^{-1})^{-1}\Big(\nu_k\delta\widehat{R}_\lambda(\overline{G}_{\infty, \lambda})[\Omega_k]\Omega_k\Big)\Big\|_{\Sigma_F}^2\\
    &= \frac{1}{4}\sum_{k=1}^\infty (1 + \lambda\gamma_k^{-1})^{-2}\Big(\delta\widehat{R}_\lambda(\overline{G}_{\infty, \lambda})[\Omega_k]\Big)^2\\
    &\lesssim \kappa(G_\cG)\frac{\log(1/\delta)}{n}\sum_{k=1}^\infty (1 + \lambda\gamma_k^{-1})^{-2}\|\Omega_k\|_{\Sigma_F}^2\\
    &= \kappa(G_\cG)\frac{\log(1/\delta)}{n}\sum_{k=1}^\infty (1 + \lambda\gamma_k^{-1})^{-2}\\
    &\lesssim \kappa(G_\cG)\frac{\log(1/\delta)}{n}\sum_{k=1}^\infty (1 + \lambda k^r)^{-2}\\
    &\asymp \kappa(G_\cG)\frac{\log(1/\delta)}{n}\int_1^\infty (1 + \lambda x^r)^{-2}dx\\
    &= \kappa(G_\cG)\frac{\log(1/\delta)}{n}\lambda^{-\frac{1}{r}}\int_{\lambda^\frac{1}{r}}^\infty (1 + x^r)^{-2}dx
\end{aligned}
\end{equation}
Noting that $\int_{\lambda^\frac{1}{r}}^\infty (1 + x^r)^{-2}dx \leq \int_{\lambda^\frac{1}{r}}^1 (1 + x^r)^{-2}dx + \int_1^\infty x^{-2r}dx \leq 1 + \frac{1}{2r - 1} - \lambda^\frac{1}{r}$ for all $r > \frac{1}{2}$ we have that
\begin{equation}\label{eq:stochbound1}
    \|\tilde{G} - \overline{G}_{\infty, \lambda}\|_{\Sigma_F}^2 \lesssim \kappa(G_\cG)\frac{\log(1/\delta)}{n}\lambda^{-\frac{1}{r}}.
\end{equation}~\\

\noindent\underline{2. Bounding $\|\widehat{G}_{n, \lambda} - \tilde{G}\|_{\Sigma_F}$}~\\
Now we move on to bounding $\|\widehat{G}_{n, \lambda} - \tilde{G}\|_{\Sigma_F}$. First clearly $\nabla\widehat{R}_\lambda(\widehat{G}_{n, \lambda}) = 0$ by first-order optimality. Since $\widehat{R}_\lambda(G)$ is quadratic then we can in fact write a Taylor series expansion 
\begin{equation}
    \nabla\widehat{R}_\lambda(\widehat{G}_{n, \lambda}) = \nabla\widehat{R}_\lambda(\overline{G}_{\infty, \lambda}) + \nabla^2\widehat{R}_\lambda(\widehat{G}_{n, \lambda} - \overline{G}_{\infty, \lambda}) = 0.
\end{equation}

Also, since $\tilde{G} = \overline{G}_{\infty, \lambda} - \nabla^2R_\lambda^{-1}\Big(\nabla\widehat{R}_\lambda(\overline{G}_{\infty, \lambda})\Big)$ then
\begin{equation}
\nabla^2R_\lambda(\tilde{G} - \overline{G}_{\infty, \lambda}) = -\nabla\widehat{R}_\lambda(\overline{G}_{\infty, \lambda}).
\end{equation}
To conclude, we know that
\begin{equation}
    \nabla\widehat{R}_\lambda(\overline{G}_{\infty, \lambda}) = \nabla^2R_\lambda(\overline{G}_{\infty, \lambda} - \tilde{G}) = \nabla^2\widehat{R}_\lambda(\overline{G}_{\infty, \lambda} - \widehat{G}_{n, \lambda}).
\end{equation}
Hence we can write
\begin{equation}
\begin{aligned}
    \nabla^2R_\lambda(\widehat{G}_{n, \lambda} - \tilde{G}) &= \nabla^2R_\lambda(\widehat{G}_{n, \lambda} - \overline{G}_{\infty, \lambda}) + \nabla^2R_\lambda(\overline{G}_{\infty, \lambda} - \tilde{G})\\
    &= \nabla^2R_\lambda(\widehat{G}_{n, \lambda} - \overline{G}_{\infty, \lambda}) - \nabla^2\widehat{R}_\lambda(\widehat{G}_{n, \lambda} - \overline{G}_{\infty, \lambda})\\
    &= \nabla^2R(\widehat{G}_{n, \lambda} - \overline{G}_{\infty, \lambda}) - \nabla^2\widehat{R}(\widehat{G}_{n, \lambda} - \overline{G}_{\infty, \lambda})
\end{aligned}
\end{equation}
which proves that
\begin{equation}
    \widehat{G}_{n, \lambda} - \tilde{G} = \nabla^2R_\lambda^{-1}\Big(\nabla^2R(\widehat{G}_{n, \lambda} - \overline{G}_{\infty, \lambda}) - \nabla^2\widehat{R}(\widehat{G}_{n, \lambda} - \overline{G}_{\infty, \lambda})\Big).
\end{equation}
Now denoting $\widehat{G}_{n, \lambda} = \sum_{k=1}^\infty \widehat{b}_k\Omega_k$ and $\overline{G}_{\infty, \lambda} = \sum_{k=1}^\infty \overline{b}_k\Omega_k$ we get by Cauchy--Schwarz and~\eqref{eq:covconc} that
\begin{equation}
\begin{aligned}
    \|\widehat{G}_{n, \lambda} &- \tilde{G}\|_{\Sigma_F}^2\\
    &= \frac{1}{4}\Big\|\sum_{k=1}^\infty \nu_k^{-1}(1 + \lambda\gamma_k^{-1})^{-1}\Big(\nu_k\delta\widehat{R}^2(\widehat{G}_{n, \lambda} - \overline{G}_{\infty, \lambda})[\Omega_k]\Omega_k - \nu_k\delta R^2(\widehat{G}_{n, \lambda} - \overline{G}_{\infty, \lambda})[\Omega_k]\Omega_k\Big)\Big\|_{\Sigma_F}^2\\
    &= \frac{1}{4}\sum_{k=1}^\infty (1 + \lambda\gamma_k^{-1})^{-2}\Big(\sum_{j=1}^\infty (\widehat{b}_j - \overline{b}_j) \inprod{((\widehat{\Sigma}_F - \Sigma_F) \otimes I)\Omega_j, \Omega_k}_{L^2(D_\cX \times D_\cY)}\Big)^2\\
    &\lesssim \frac{\log(1/\delta)}{n}\sum_{k=1}^\infty (1 + \lambda\gamma_k^{-1})^{-2}\sum_{j=1}^\infty (\widehat{b}_j - \overline{b}_j)^2\|\Omega_j\|_{\Sigma_F}^2\|\Omega_k\|_{\Sigma_F}^2\\
    &\leq \frac{\log(1/\delta)}{n}\sum_{k=1}^\infty(1 + \lambda\gamma_k^{-1})^{-2}\sum_{j=1}^\infty (\widehat{b}_j - \overline{b}_j)^2\\
    &\leq \frac{\log(1/\delta)}{n}\lambda^{-\frac{1}{r}}\|\widehat{G}_{n, \lambda} - \overline{G}_{\infty, \lambda}\|_{\Sigma_F}^2
\end{aligned}
\end{equation}
for all $r > \frac{1}{2}$ with probability at least $1 - \delta$.

\noindent\underline{3. Combining both bounds}~\\
By the triangle inequality we know that
\begin{equation}
    \|\tilde{G} - \overline{G}_{\infty, \lambda}\|_{\Sigma_F} \geq \|\widehat{G}_{n, \lambda} - \overline{G}_{\infty, \lambda}\|_{\Sigma_F} - \|\widehat{G}_{n, \lambda} - \tilde{G}\|_{\Sigma_F} \geq \Big(1 - C\lambda^{-\frac{1}{2r}}\sqrt{\frac{\log(1/\delta)}{n}}\Big)\|\widehat{G}_{n, \lambda} - \overline{G}_{\infty, \lambda}\|_{\Sigma_F}
\end{equation}
for some absolute constant $C > 0$. Hence by~\eqref{eq:stochbound1}, if $\frac{\log(1/\delta)}{n}\lambda^{-\frac{1}{r}} \lesssim 1$ then for sufficiently large $n$,
\begin{equation}
    \|\widehat{G}_{n, \lambda} - \overline{G}_{\infty, \lambda}\|_{\Sigma_F} \leq \Big(1 - C\lambda^{-\frac{1}{2r}}\sqrt{\frac{\log(1/\delta)}{n}}\Big)^{-1}\|\tilde{G} - \overline{G}_{\infty, \lambda}\|_{\Sigma_F} \lesssim \sqrt{\kappa(G_\cG)}\lambda^{-\frac{1}{2r}}\sqrt{\frac{\log(1/\delta)}{n}}.
\end{equation}
where $\kappa(G_\cG) = \max\Big(1, \|G_\cG\|_\text{op}, \lambda J(G_\cG)\Big)$. So finally, squaring both sides proves that
\begin{equation}\label{eq:stochbound2}
    \|\widehat{G}_{n, \lambda} - \overline{G}_{\infty, \lambda}\|_{\Sigma_F}^2 \lesssim \max\Big(1, \|G_\cG\|_\text{op}, \lambda J(G_\cG)\Big)\frac{\log(1/\delta)}{n}\lambda^{-\frac{1}{r}}
\end{equation}
with probability at least $1 - \delta$.
\end{proof}

\newpage
\section{Enforcing Symmetries \& Invariances in RKHSs}\label{app:rkhs_constraints}
Here we described how to transform an RKHS so that functions in this Hilbert space satisfy constraints such as coordinate symmetries, time causality, and time invariance.

\subsection{Coordinate Symmetries}\label{app:coord_symm}
Suppose we have an RKHS of functions $\cH \in L^2(D \times D)$ and we would like to transform this space such that every function $f \in \cH$ is symmetric in its coordinates
\begin{equation}
    f(x, y) = f(y, x), \quad \forall x, y \in D.
\end{equation}
We assume here that our RKHS $\cH$ has a continuous, square-integrable, and positive semidefinite kernel $K(D^2 \times D^2) \to \mathbb{R}$ that satisfies the symmetry property
\begin{equation}
    K(x, y, \xi, \eta) = K(y, x, \eta, \xi), \quad \forall x, y, \xi, \eta \in D.
\end{equation}

As shorthand, for any $f \in L^2(D \times D)$ we define $f^T \in L^2(D \times D)$ given by $f^T(x, y) = f(y, x)$. Now we can state the following theorem.
\begin{theorem}\label{thm:symmrk}
    If the kernel $K: D^4 \to \mathbb{R}$ of $\cH$ satisfies the symmetry property $K(x, y, \xi, \eta) = K(y, x, \eta, \xi)$ then we know that $f^T \in \cH$ for any $f \in \cH$. Furthermore, we can define the symmetrized RKHS
    \begin{equation}
        \cS := \Big\{\frac{f + f^T}{2} : f \in \cH\Big\} = \Big\{f = f^T : f \in \cH\Big\}
    \end{equation}
    with inner product inherited from $\cH$ whose reproducing kernel takes the form
    \begin{equation}
        K_\text{symm}(x, y, \xi, \eta) = \frac{1}{4}\Big[K(x, y, \xi, \eta) + K(x, y, \eta, \xi) + K(y, x, \xi, \eta) + K(y, x, \eta, \xi)\Big].
    \end{equation}
\end{theorem}
We give the proof of this result below.

\begin{proof}
    We assume that $K$ is continuous, square integrable, and positive semidefinite so it is a Mercer kernel with the decomposition
    \begin{equation}
        K(x, y, \xi, \eta) = \sum_{k=1}^\infty \lambda_k\psi_k(x, y)\psi_k(\xi, \eta)
    \end{equation}
    where $\lambda_k$ are the eigenvalues and $\psi_k \in L^2(D_\cX \times D_\cY)$ are the $L^2$ orthonormal eigenfunctions.
    
    First we prove that if $f \in \cH$ then $f^T \in \cH$. A consequence of Mercer's theorem is that $\cH$ can be characterized as
    \begin{equation}
        \cH = \Big\{f \in L^2(D \times D)\Big| \sum_{k=1}^\infty \frac{\inprod{f, \psi_k}_{L^2(D \times D)}^2}{\lambda_k}\Big\}.
    \end{equation}
    Since $K$ satisfies the symmetry $K(x, y, \xi, \eta) = K(y, x, \eta, \xi)$, it is immediate that
    \begin{equation}
        K(x, y, \xi, \eta) = \sum_{k=1}^\infty \lambda_k\psi_k(y, x)\psi_k(\eta, \xi)
    \end{equation}
    so we can also write
    \begin{equation}
        \cH = \Big\{f \in L^2(D \times D)\Big| \sum_{k=1}^\infty \frac{\inprod{f, \psi_k^T}_{L^2(D \times D)}^2}{\lambda_k}\Big\}.
    \end{equation}
    For any $f \in \cH$ we know that
    \begin{equation}
        \sum_{k=1}^\infty \frac{\inprod{f, \psi_k}_{L^2(D \times D)}^2}{\lambda_k} < \infty
    \end{equation}
    which implies that
    \begin{equation}
        \sum_{k=1}^\infty \frac{\inprod{f^T, \psi_k^T}_{L^2(D \times D)}^2}{\lambda_k} < \infty
    \end{equation}
    proving that $f^T \in \cH$.
    
    Now define the symmetrized set of functions
    \begin{equation}
        \cS := \Big\{\frac{f + f^T}{2} : f \in \cH\Big\} \subseteq \cH.
    \end{equation}
    Let's prove that $\cS$ as a set of functions is equal to
    \begin{equation}
        \cS' = \Big\{f = f^T : f \in \cH\Big\}.
    \end{equation}
    This follows immediately since for any $f \in S$ we can check that $f \in \cH$ and $f = f^T$ so $f \in \cS'$. In the other direction, for any $f \in \cS'$ by definition $f \in \cH$ and $f = f^T$ so $\frac{f + f^T}{2} = f$ implying that $f \in \cS$. Since $\cS = \cS'$ is a closed subset of $\cH$ we can naturally equip it with the inner product from $\cH$, proving that it is a Hilbert space.
    
    Finally, we need to show that $\cS$ is an RKHS with kernel
    \begin{equation}
        K_\text{symm}(x, y, \xi, \eta) = \frac{1}{4}\Big[K(x, y, \xi, \eta) + K(x, y, \eta, \xi) + K(y, x, \xi, \eta) + K(y, x, \eta, \xi)\Big].
    \end{equation}
    Remember that we can interpret the reproducing kernels $K$ and $K_\text{symm}$ as maps $K: \cH \to \cH$ and $K_\text{symm}: \cS \to \cS$ by
    \begin{equation}
    \begin{gathered}
        K(f)(x, y) = \int_D\int_D K(x, y, \xi, \eta)f(\xi, \eta) \ud\xi\ud\eta\\ K_\text{symm}(f)(x, y) = \int_D\int_D K_\text{symm}(x, y, \xi, \eta)f(\xi, \eta) \ud\xi\ud\eta.
    \end{gathered}
    \end{equation}
    First note that $K_\text{symm}$ is continuous, square integrable, symmetric, and positive semidefinite. This last condition can be checked by noting that $K_\text{symm}(f) = K(f)$ for all $f \in \cS$ which implies that $K_\text{symm}$ inherits the positive semidefiniteness of $\cH$. Furthermore, ${K_\text{symm}}_{(x_0, y_0)} \in \cS$ for all $x_0, y_0 \in D$ since $K_{(x_0, y_0)}, K_{(y_0, x_0)} \in \cH$. Here the notation $K_{(x_0, y_0)}$ means we are centering the kernel at a point $(x_0, y_0) \in D \times D$ to get a function $K(x, y, x_0, y_0)$.
    
    We conclude by showing that $K_\text{symm}$ satisfies the reproducing property on $\cS$. First note that the symmetry property $K(x, y, \xi, \eta) = K(y, x, \eta, \xi)$ gives us that $K(\cS) \subseteq \cS$ and $K(\cA) \subseteq K(\cA)$ for the space of antisymmetric functions
    \begin{equation}
        \cA := \Big\{\frac{f - f^T}{2}: f \in \cH\Big\} = \Big\{f = -f^T: f \in \cH\Big\}.
    \end{equation}
    This proves that $K(\cS) = \cS$ and $K(\cA) = \cA$. Hence, for any $f \in \cS$ we know that $K^{-1}(f) \in \cS$ so
    \begin{equation}
    \begin{aligned}
        \inprod{{K_\text{symm}}_{(x, y)}, f}_\cS &:= \inprod{{K_\text{symm}}_{(x, y)}, f}_\cH = \inprod{{K_\text{symm}}_{(x, y)}, K^{-1}(f)}_{L^2(D \times D)}\\
        &= \inprod{{K_\text{symm}}_{(x, y)}, K^{-1}(f)}_{L^2(D \times D)} = K_\text{symm}(K^{-1}(f))(x, y)\\
        &= K(K^{-1}(f))(x, y) = f(x, y)
    \end{aligned}
    \end{equation}
    where the equality from the second to the third line follows since $K_\text{symm} = K$ on $\cS$. This proves that $K_\text{symm}: (D \times D)^2 \to \mathbb{R}$ is the reproducing kernel for $\cS$ with the inner product inherited from $\cH$.
\end{proof}

Theorem~\ref{thm:symmrk} can naturally be extended to coordinate symmetries (permutations) of $n$ variables by iterating the statement over pairs of coordinates at a time.

\subsection{Flip ($\mathbb{Z}_2$) Symmetries}\label{app:flip_symm}
Another important class of symmetries are flips along a coordinate axis. Suppose we have an RKHS of functions $\cH \in L^2([-a, a])$ and we would like to transform it so that every $f \in \cH$ satisfies
\begin{equation}
    f(x) = f(-x), \quad \forall x \in [-a, a].
\end{equation}
where $a > 0$ can be finite or infinite.

Let's assume that our initial RKHS $\cH$ has a continuous, square-integrable, and positive semidefinite kernel $K([-a, a]^2) \to \mathbb{R}$ that satisfies the symmetry property
\begin{equation}
    K(x, \xi) = K(-x, -\xi), \quad \forall x, \xi \in [-a, a].
\end{equation}
Through a nearly identical argument to the case of coordinate symmetries, we can show the following result. As shorthand, for any $f \in L^2([-a, a])$ we define $f^- \in L^2(-[a, a])$ given by $f^-(x) = f(-x)$.
\begin{theorem}\label{thm:fliprk}
    If the kernel $K: [-a, a]^2 \to \mathbb{R}$ of $\cH$ satisfies the symmetry property $K(x, \xi) = K(-x, -\xi)$ then we know that $f^- \in \cH$ for any $f \in \cH$. Furthermore, we can define the symmetrized RKHS
    \begin{equation}
        \cF := \Big\{\frac{f + f^-}{2} : f \in \cH\Big\} = \Big\{f = f^- : f \in \cH\Big\}
    \end{equation}
    with inner product inherited from $\cH$ whose reproducing kernel takes the form
    \begin{equation}
        K_\text{flip}(x, \xi) = \frac{1}{4}\Big[K(x, \xi) + K(x, -\xi) + K(-x, \xi) + K(-x, -\xi)\Big].
    \end{equation}
\end{theorem}

\subsection{Enforcing Time Causality}\label{app:time_causality}
Here we discuss how to transform an RKHS so that functions in this space satisfy a time constraint known as causality. We begin with an RKHS of functions $\cH \in L^2([a, b] \times [a, b])$ where $[a, b]$ denotes an interval of time and may generally have open, closed, finite, or infinite endpoints. We aim to transform our RKHS so that for every $f \in \cH$ we have that
\begin{equation}
    f(s, t) = \mathbf{1}_{t \geq s}f(s, t).
\end{equation}
The new function $\mathbf{1}_{t \geq s}f(s, t)$ is time causal because taking an input signal $p(s)$ and integrating it against the $s$-coordinate
\begin{equation}
    q(t) = \int_a^b \mathbf{1}_{t \geq s}f(s, t)p(s) \ud s
\end{equation}
generates an output signal $q(t)$ where the influence of the output at time $t$ only depends on the perturbation $p(s)$ at previous times $s \leq t$. Hence, the causal order of time is respected by the linear filter $\mathbf{1}_{t \geq s}f(s, t)$.

Multiplication of functions in an RKHS by a fixed positive weighting function is a standard procedure described in~\citet[Section~2.3.4, Corollary~2.5]{saitoh2016theory}. The difference here is that $\mathbf{1}_{t \geq s}$ is not strictly positive at all points $(s, t)$. Hence, applying this weighting to all functions in $\cH$ does not automatically produce a new RKHS that inherits the original inner product of $\cH$.

We proceed to construct this RKHS in the following way. Using Theorem~\ref{thm:symmrk} we symmetrize $\cH$ in the time coordinate to obtain the RKHS $\cS$ whose elements satisfy $f(s, t) = f(t, s)$. In order to do this, we implicitly assume that our kernel $K:[a, b]^4 \to \mathbb{R}$ for the RKHS $\cH$ satisfies the property $K(s, t, \sigma, \tau) = K(t, s, \tau, \sigma)$. Then we know that $\cS \subseteq \cH$ is an RKHS with inner product inherited from $\cH$ and reproducing kernel
\begin{equation}
\begin{split}
    K_\text{symm}(t, s, \tau, \sigma) &= \frac{1}{4}\Big[K(s, t, \sigma, \tau) + K(s, t, \tau, \sigma) + K(t, s, \sigma, \tau) + K(t, s, \tau, \sigma)\Big].
\end{split}
\end{equation}

Now define the causal RKHS
\begin{equation}
    \cH_\text{causal} := \Big\{\mathbf{1}_{\{t \geq s\}}f: f \in \cS\Big\}
\end{equation}
with the inherited inner product
\begin{equation}
    \inprod{\overline{f}, \overline{g}}_{\cH_\text{causal}} := \Big\langle f, g\Big\rangle_\cS.
\end{equation}
for all $\overline{f} = \mathbf{1}_{\{t \geq s\}}f$ and $\overline{g} = \mathbf{1}_{\{t \geq s\}}g \in \cH_\text{causal}$ where $f, g \in \cS$. Equipping $\cH_\text{causal}$ with this inner product implies that it is isometrically isomorphic to the time-symmetrized RKHS $\cS$.

It is not hard to check that $\cH_\text{causal}$ has the reproducing kernel
\begin{equation}
    K_\text{causal} = \mathbf{1}_{\{t \geq s\}}\mathbf{1}_{\{\tau \geq \sigma\}}K_\text{symm}
\end{equation}
since it is symmetric, positive semidefinite, and it satisfies the reproducing property
\begin{equation}
\begin{split}
    \inprod{{K_\text{causal}}_{(s, t)}, \overline{f}}_{\cH_\text{causal}} &= \mathbf{1}_{\{t \geq s\}}\inprod{{K_\text{symm}}_{(s, t)}, f}_{\cS}\\
    &= \mathbf{1}_{\{t \geq s\}}f(s, t) = \overline{f}(s, t)
\end{split}
\end{equation}
for any $\overline{f} \in \cH_\text{causal}$ where $\overline{f} = \mathbf{1}_{\{t \geq s\}}f$ for some $f \in \cS$. Here again the notation ${K_\text{symm}}_{(s, t)}$ denotes the kernel centered at a given point $K_\text{symm}(\cdot, \cdot, s, t)$.

In summary, to create an RKHS of causal functions from some initial RKHS $\cH$, we first symmetrize it in time to obtain the RKHS $\cS \subseteq \cH$ and then for each $f \in \cS$ we set $f(s, t) = 0$ for all $s > t$ to obtain the RKHS $\cH_\text{causal}$ along with the formula for its reproducing kernel $K_\text{causal}$.\\

In exactly the same way, we can define the space of \textit{anticausal functions}
\begin{equation}
    \cH_\text{anticausal} := \Big\{\mathbf{1}_{\{t \geq s\}}f: f \in \cS\Big\}
\end{equation}
where functions in this space satisfy
\begin{equation}
    f(s, t) = 0 \quad \text{for} \quad t > s.
\end{equation}
Such functions are less common but occur for example in systems with prescribed terminal conditions at time $t = T$.

\subsection{Enforcing Time Invariance}\label{app:time_invariance}
Another important constraint we may want functions in our RKHS to satisfy is invariance to time. A function $f(s, t)$ for $s, t \in [a, b]$ is time-invariant if it can be rewritten as the difference of its two coordinates
\begin{equation}
    f(s, t) = f(t - s).
\end{equation}
To build such a space of functions, we start with an RKHS of one-dimensional functions $\mathfrak{h} \subset L^2([-(b-a), (b-a)])$ with kernel $k: [-(b-a), (b-a)]^2 \to \mathbb{R}$. We then lift this RKHS to the space of two-dimensional functions
\begin{equation}
    \cC(\mathfrak{h}) := \Big\{f: [a, b] \times [a, b] \to \mathbb{R},\ f(s, t) = h(t - s), \ \forall h \in \mathfrak{h}\Big\}
\end{equation}
equipped with the inherited inner product for all $f_1, f_2 \in \cC(\mathfrak{h})$,
\begin{equation}
    \inprod{f_1, f_2}_{\cC(\mathfrak{h})} = \inprod{h_1, h_2}_\mathfrak{h}
\end{equation}
where $f_1(s, t) = h_1(t - s)$ and $f_2(s, t) = h_2(t - s)$ with $h_1, h_2 \in \mathfrak{h}$. It is easy to see that $\cC(\mathfrak{h}) \subset L^2([a, b]^2)$
is isometrically isomorphic to $\mathfrak{h}$ and by properties of the reproducing kernel, we can check that $K: [a, b]^4 \to \mathbb{R}$ where
\begin{equation}
    K(s, t, \sigma, \tau) = k(t - s, \tau - \sigma)
\end{equation}
is the reproducing kernel of $\cC(\mathfrak{h})$.

Computationally, time-invariant constraints are useful as they reduce the amount of information stored in a function (i.e. reduce it from a function of two variables $s, t$ to a function of one variable $t - s$). As discussed in Section~\ref{sec:implementation} our standard way of constructing a function $f(s, t) \in \cC(\mathfrak{h})$ is to take equally-spaced grid points $s_j = t_j = \frac{b-a}{m-1}(j-1) + a$ for $j = 1, \dots, m$ and a matrix of weights $W \in \mathbb{R}^{m \times m}$ where
\begin{equation}
    f(s, t) = \sum_{i=1}^m\sum_{j=1}^m K(s, t, s_i, t_j)W_{ij}.
\end{equation}
Using the convolutional form of $K$ we can further write
\begin{equation}
\begin{aligned}
    f(s, t) = \sum_{i=1}^m\sum_{j=1}^m k(t - s, t_j - s_i)W_{ij} &= \sum_{i=1}^m\sum_{j=1}^m k\Big(t - s, \frac{b-a}{m-1}(j-i)\Big)W_{ij}\\
    &= \sum_{i=-(m-1)}^{m-1} k\Big(t - s, \frac{b-a}{m-1}i\Big)w_i.
\end{aligned}
\end{equation}
where $w \in \mathbb{R}^{2m-1}$ is a new set of weights. Thus, making the physically relevant assumption that $f$ lies in an RKHS of convolutional operators can significantly decrease computation time and memory if optimized as in the equation above.

\newpage

\vskip 0.2in
\bibliography{bibliography}

\end{document}